\documentclass[reqno]{amsart}
\usepackage{amssymb}
\usepackage[mathcal]{euscript}
\usepackage[cmtip,all]{xy}
\usepackage{amsmath}\usepackage{amsfonts}
\usepackage{graphicx} 
\usepackage{verbatim} 
\usepackage{tikz}
\usepackage{hyperref} 
\usepackage{relsize} 
\usepackage[left=20mm, right=20mm, top=20mm, bottom=20mm]{geometry} 
\usepackage{mathtools}
\makeatletter
\DeclareRobustCommand\widecheck[1]{{\mathpalette\@widecheck{#1}}}
\def\@widecheck#1#2{%
    \setbox\z@\hbox{\m@th$#1#2$}%
    \setbox\tw@\hbox{\m@th$#1%
       \widehat{%
          \vrule\@width\z@\@height\ht\z@
          \vrule\@height\z@\@width\wd\z@}$}%
    \dp\tw@-\ht\z@
    \@tempdima\ht\z@ \advance\@tempdima2\ht\tw@ \divide\@tempdima\thr@@
    \setbox\tw@\hbox{%
       \raise\@tempdima\hbox{\scalebox{1}[-1]{\lower\@tempdima\box
\tw@}}}%
    {\ooalign{\box\tw@ \cr \box\z@}}}
\makeatother
\newcommand{\bi}{\textnormal{\textbf{i}}}  \newcommand{\cB}{\mathcal{B}}  \newcommand{\FF}{\mathbb{F}} \newcommand{\ZZ}{\mathbb{Z}} \newcommand{\NN}{\mathbb{N}} 
  \newcommand{\cV}{\mathcal{V}} \newcommand{\cW}{\mathcal{W}}   \newcommand{\med}{\;|\;}  \newcommand{\cF}{\mathcal{F}}
\DeclareMathOperator{\supp}{\mathrm{Supp}\,} \DeclareMathOperator{\Aut}{\mathrm{Aut}}  \DeclareMathOperator{\Stab}{\mathrm{Stab}}   
\DeclareMathOperator{\End}{\mathrm{End}}   
\DeclareMathOperator{\bAut}{\mathbf{Aut}} \DeclareMathOperator{\bmu}{\boldsymbol{\mu}} 
\DeclareMathOperator{\chr}{\mathrm{char}\,}  \DeclareMathOperator{\lspan}{\mathrm{span}} \DeclareMathOperator{\gl}{\mathfrak{gl}}

  \newcommand{\GL}{\mathrm{GL}}  
 
\newcommand{\cA}{\mathcal{A}} 
  
\DeclareMathOperator{\Hom}{\mathrm{Hom}} \DeclareMathOperator{\id}{\mathrm{id}}

 \newcommand{\cM}{\mathcal{M}}
\newcommand{\instr}{{\mathfrak{instr}}}
\newcommand{\Tr}{\mathsf{T}} 
\DeclareMathOperator{\tr}{\textnormal{tr}} 
\DeclareMathOperator{\bGL}{\mathbf{GL}} \DeclareMathOperator{\bSL}{\mathbf{SL}}
\DeclareMathOperator{\bG}{\mathbf{G}}

\DeclareMathOperator{\bT}{\mathbf{T}}
\DeclareMathOperator{\MLSM}{\mathbf{MLSM}} 
\DeclareMathOperator{\MFLSM}{\mathbf{MFLSM}} 
\DeclareMathOperator{\MFLM}{\mathbf{MFLM}} 
\DeclareMathOperator{\MGJSP}{\mathbf{MGJSP}} 
\DeclareMathOperator{\MGJP}{\mathbf{MGJP}} 

\newcommand{\bigboxtimes}{\raisebox{-1.5pt}{\scalebox{1.3}{$\boxtimes$}}}
\DeclareMathOperator{\sgn}{\textnormal{sgn}} 
\DeclareMathOperator{\per}{\textnormal{per}} 
\DeclareMathOperator{\detper}{\textnormal{detper}}
\DeclareMathOperator{\perdet}{\textnormal{perdet}}
\DeclareMathOperator{\diag}{\textnormal{diag}}
\DeclareMathOperator{\altI}{\widehat{\mathcal{I}}}
\DeclareMathOperator{\altB}{\widehat{\mathcal{B}}}
\DeclareMathOperator{\sgnalt}{\widehat{\textnormal{sgn}}}
\DeclareMathOperator{\varsigmahat}{\widehat{\varsigma}}
\DeclareMathOperator{\symI}{\widecheck{\mathcal{I}}}
\DeclareMathOperator{\symB}{\widecheck{\mathcal{B}}}
\DeclareMathOperator{\sgnsym}{\widecheck{\textnormal{sgn}}}
\DeclareMathOperator{\varsigmacheck}{\widecheck{\varsigma}}
\DeclareMathOperator{\bTGS}{\mathbf{1}} 
\DeclareMathOperator{\bker}{\mathbf{ker}} 
\DeclareMathOperator{\adj}{\textnormal{adj}} 
\DeclareMathOperator{\Orb}{\textnormal{Orb}} 

\newtheorem{theorem}{Theorem}
\newtheorem{proposition}[theorem]{Proposition}

\theoremstyle{definition}
\newtheorem{df}[theorem]{Definition}
\newtheorem{dfs}[theorem]{Definitions}
\newtheorem{example}[theorem]{Example}
\newtheorem{notation}[theorem]{Notation}
\newtheorem{remark}[theorem]{Remark}
\newtheorem{remarks}[theorem]{Remarks}

\numberwithin{equation}{section} 
\numberwithin{theorem}{section} 
\allowdisplaybreaks  

\begin{document}

\title[Alternating and symmetric superpowers of metric generalized Jordan superpairs]{Alternating and symmetric superpowers \\ of metric generalized Jordan superpairs}

\author[D. Aranda-Orna]{Diego Aranda-Orna}
\address{Departamento de Matem\'{a}ticas,
Universidad de Oviedo, 33007 Oviedo, Spain}
\email{diego.aranda.orna@gmail.com}

\author[A.S. C\'ordova-Mart\'inez]{Alejandra S. C\'ordova-Mart\'inez}
\address{
Departamento de Matem\'{a}tica Aplicada, E.T.S. Ingenier\'ia Inform\'{a}tica, Universidad de M\'{a}laga, 29071 M\'{a}laga, Spain}
\email{acordova@uma.es}

\thanks{Both authors are supported by grant PID2021-123461NB-C21, funded by MCIN/AEI/10.13039/501100011033 and by ``ERDF A way of making Europe''. The second author is also supported by S60\_23R (Gobierno de Arag\'on, Grupo de investigaci\'on ``Investigaci\'on en Educaci\'on Matem\'atica'').}
\date{}

\begin{abstract}
The aim of this paper is to define and study the constructions of alternating and symmetric (super)powers of metric generalized Jordan (super)pairs. These constructions are obtained by transference via the Faulkner construction. The construction of tensor (super)products for metric generalized Jordan (super)pairs is revisited. We always assume that the characteristic of the base field $\FF$ is different from $2$; in case of positive characteristic, sometimes we require that the characteristic is large enough to allow nondegeneracy of certain bilinear forms.
\end{abstract}

\maketitle

\section{Introduction} \label{section.introduction}

For theory of Lie superalgebras, the reader may consult \cite{S79}, \cite{CW12}, \cite{M12}, \cite{FSS00}. The rule of signs of the super case is explained in \cite[\S A.2.2]{M12}. The basics of multilinear algebra can be consulted in \cite{G78}, although some proofs there require characteristic $0$; there, the reader may consult the definition of the permanent of a matrix, and its properties. The dual pairings for alternating (exterior) powers and also for symmetric powers of a finite-dimensional vector space over an arbitrary field can be found in \cite[Appendix A]{K89} (we point out that it is incorrectly stated that the isomorphism in \cite[Corollary A.24(b)]{K89} holds in arbitrary characteristic); we have not found a good and detailed reference about this topic in the super case. Some surveys about metric Lie algebras (also known as quadratic Lie algebras, self-dual Lie algebras, and many other names) are \cite{O16}, \cite{BR23}. For affine group schemes, see \cite{W79} or \cite[Appendix A]{EKmon}. The Faulkner construction was discovered in \cite{F73}, and extended to the super case in \cite{A22}. In \cite{F84}, invariant bilinear forms on generalized Jordan pairs are studied.

A description of good bases of the exterior and symmetric superalgebras of a finite-dimensional vector superspace is well-known, as a certain version of the Poincar\'e-Birkhoff-Witt Theorem (PBW Theorem) \cite[Th. A.2.8]{M12}; we include a different proof of that result in Notations \ref{notation.basis.alternating} and \ref{notation.basis.symmetric} of the present paper, where we deal with bilinear forms; we have not found our approach for the super case in the literature (however, the bilinear forms that extend the determinant and permanent that were studied in \cite{GRS87} seem to be related to the ones in the present paper).

\bigskip

Without further mention, we will always assume that the base field $\FF$ has characteristic different from $2$.
If $\chr\FF > 0$, sometimes we will need to assume that the characteristic is large enough so that certain bilinear forms are nondegenerate. In the present work, as explained in $\S~\ref{subsec.Liesupermodules}$, when we say ``Lie superalgebra'' we always mean ``weak Lie superalgebra''.

\bigskip

This paper is structured as follows:

\smallskip

In $\S\ref{section.preliminaries}$, we will recall some of the most basic definitions used in this work. In $\S\ref{section.modules.tensor}$, the tensor superproduct of Lie supermodules is revisited and some notation is introduced to be used in further sections. In $\S\ref{section.pairs.tensor}$, we focus on metric generalized Jordan superpairs, for which the restricted tensor superpowers are studied, and the tensor superproduct is revisited.

In $\S\ref{section.supermodules.alternating}$ and $\S\ref{section.supermodules.symmetric}$, we recall the basics of alternating and symmetric superpowers of Lie supermodules; although not all results there are original, some results involving dual pairings and superminors are expected to be original (as generalizations of well-known results of the non-super case).

The main goal of this paper is reached in $\S\ref{section.superpairs.alternating}$ and $\S\ref{section.superpairs.symmetric}$, where the Faulkner construction is used to transfer the constructions of alternating and symmetric superpowers to the class of metric generalized Jordan superpairs. We also provide two examples involving decompositions of the simple Jordan pairs of types II and III, which was another of the main motivations of our study.

\bigskip

\noindent\textbf{Motivation:} Let $\cV^\text{(I)}_{p,q}$ denote the simple Jordan pair of type I of parameters $p<q$, with its generic trace $t$. In \cite[Ex.4.7, eq.(4.14)]{A22} it was shown that $\cV^\text{(I)}_{p,q}$ is isomorphic to a tensor-shift of the tensor product $\cV^\text{(I)}_{1,p} \otimes \cV^\text{(I)}_{1,q}$. Let $\cV^\text{(II)}_n$ and $\cV^\text{(III)}_n$ be the simple Jordan pairs of types II and III of parameter $n$, which are Jordan subpairs of $\cV^\text{(I)}_{n,n}$. Note that the vector spaces of $\cV^\text{(I)}_{p,q}$ are both $\cM_{p,q}(\FF)$, the vector spaces of $\cV^\text{(II)}_n$ are given by the antisymmetric matrices $A_n(\FF)$, and the vector spaces of $\cV^\text{(III)}_n$ are given by the symmetric matrices $H_n(\FF)$. Since $\cM_n(\FF) = A_n(\FF) \oplus H_n(\FF)$, we have a decomposition of subpairs (but not ideals) $\cV^\text{(I)}_{n,n} = \cV^\text{(II)}_n \oplus \cV^\text{(III)}_n$. On the other hand, recall from the Schur-Weyl duality (\cite[Chap.3]{CW12}, \cite[Chap.11]{M12}, \cite[\S6]{FH91}) that we have a decomposition of irreducible $\GL_n(\FF) \times S_2$-modules $\cM_n(\FF) \cong \cM_{1,n}(\FF)^{\otimes2} \cong \bigwedge^2 \cM_{1,n}(\FF) \oplus \bigvee^2 \cM_{1,n}(\FF)$ where we can identify $A_n(\FF) \cong \bigwedge^2 \cM_{1,n}(\FF)$ and $H_n(\FF) \cong \bigvee^2 \cM_{1,n}(\FF)$ through the isomorphism. From \cite[Prop.4.3.4) \& eq.(4.14)]{A22} it follows easily that $\big( \GL_n(\FF) \otimes_{\FF^\times} \GL_n(\FF) \big) \rtimes S_2 \leq \Aut(\cV^\text{(I)}_{n,n}, t)$ (because $\Aut(\cV^\text{(I)}_{1,n}, t) \cong \GL_n(\FF)$ and where $S_2$ corresponds to the swapping automorphism of $\cV^\text{(I)}_{1,n} \otimes \cV^\text{(I)}_{1,n}$), where the groups $\GL_n(\FF)$ and $S_2 \cong \ZZ_2$ inducing the Schur-Weyl decomposition appear. Then, it is natural to study which is the relation between the Jordan pairs $\cV^\text{(II)}_n$, $\cV^\text{(III)}_n$ and the Jordan pairs $\bigwedge^2 \cV^\text{(I)}_{1,n}$, $\bigvee^2 \cV^\text{(I)}_{1,n}$. In the present work, we will show that the former ones are isomorphic to the latter, up to a tensor-shift and up to similarity of the bilinear forms.

\bigskip

\noindent\textbf{Notation changes and corrigendum from \cite{A22}:} The classes of objects $\mathbf{GJSP}$, $\mathbf{FLSM}$ will be renamed, with a better notation, as $\MGJSP$, $\MFLSM$. By misuse of language, for bilinear forms, the term ``homogeneous'' was used with the meaning of ``homogeneous of degree $0$'' (in the present paper, we will use the term ``even'' instead). We point out that the definition of Kantor pair was miswritten. We point out that there is a nontrivial detail not explained in \cite[Rem.3.9]{A22}, which is how to identify $L_{-2}$ with the dual Lie module of $L_2$ in a natural way (it will not be explained in the present work). A notational error appearing in the proof of \cite[Prop.4.3]{A22} is also fixed here, see Remark~\ref{erratumNu} below.

\section{Preliminaries} \label{section.preliminaries}

\subsection{Lie supermodules} \label{subsec.Liesupermodules}

Given a group $G$, a $G$-grading on a vector space $V$ is a vector space decomposition $\Gamma:\;V = \bigoplus_{g\in G} V_g $. If $\cA$ is an $\FF$-algebra, then a \emph{$G$-grading} on $\cA$ is a grading on $\cA$ as vector space $$ \Gamma:\;\cA=\bigoplus_{g\in G} \cA_g $$
such that $\cA_g \cA_h\subseteq \cA_{gh}$ for all $g,h\in G$. The subspace $\cA_g$ is called {\em homogeneous component of degree $g$}, any $0 \neq x\in\cA_g$ is called a {\em homogeneous element of degree $g$}, and we write $\deg(x) = g$. The {\em support} of the grading is the set $\supp \Gamma := \{g\in G \med \cA_g\neq 0\}$.

Recall that if $V$,$W$ are $G$-graded vector spaces and $f \colon V \to W$ is a nonzero linear map, then $f$ is called \emph{homogeneous of degree $g$} if $f(V_h) \subseteq W_{gh}$ for each $h\in G$; if $V$ and $W$ are finite-dimensional, then $\Hom(V,W)$ is a $G$-graded vector space, and in particular, $\End(V)$ becomes a $G$-graded algebra.

\medskip

A \emph{vector superspace} is a $\ZZ_2$-graded vector space $V = V_{\bar0} \oplus V_{\bar1}$. The subspaces $V_{\bar0}$ and $V_{\bar1}$ are called the \emph{even} and \emph{odd} components, and their nonzero elements are called, respectively, \emph{even} and \emph{odd} elements. The degree map is given by $\varepsilon(x) := a$ if $0\neq x\in V_a$, and is called the \emph{parity} map. A homogeneous subset $\{v_i\}_{i=1}^m$ of $V$ will be said to be \emph{parity-ordered} if there is some $0\leq k\leq m$ such that $v_i$ is even for $i \leq k$ and odd for $i > k$. A \emph{superalgebra} is a $\ZZ_2$-graded algebra $\cA = \cA_{\bar0} \oplus \cA_{\bar1}$. 

Given a finite-dimensional vector superspace $V$, let $\cB$ be a parity-ordered basis of $V$. Then, taking coordinates in $\cB$, each $h \in \End(V)$ corresponds to a supermatrix
$M = \left( \begin{array}{c|c} M_{\bar0\bar0} & M_{\bar0\bar1} \\
\hline M_{\bar1\bar0} & M_{\bar1\bar1}\end{array} \right)$.
In the case that $h$ is even we have $M_{\bar0\bar1} = (0) = M_{\bar1\bar0}$, and we may denote $M_{\bar0} := M_{\bar0\bar0}$, $M_{\bar1} := M_{\bar1\bar1}$, so that $M = \diag(M_{\bar0}, M_{\bar1})$.
The superalgebra of supermatrices is denoted by $\cM_{(m|n)}(\FF)$, where $m = \dim V_{\bar 0}$, $n = \dim V_{\bar 1}$.

If $R$ is an associative, commutative, unital $\FF$-algebra, we will denote by $V_R := V \otimes R$ the corresponding scalar extension. The group of even invertible linear maps $V \to V$ will be denoted by $\GL^{\bar0}(V)$, instead of $\GL(V)$, to avoid ambiguity with the automorphisms of $V$ regarded as a vector space.

\bigskip

For homogeneous elements in a $\ZZ_2$-grading, we will denote
\begin{align}
\eta_{x,y} &:= (-1)^{\varepsilon(x)\varepsilon(y)}, \\
\eta_{x,y,z} &:= (-1)^{\varepsilon(x)\varepsilon(y)
+ \varepsilon(y)\varepsilon(z) + \varepsilon(z)\varepsilon(x)}.
\end{align}

\smallskip

Recall that a \emph{weak Lie superalgebra} (see \cite[Remark~3.1.1]{K24}) is a superalgebra $L = L_{\bar0} \oplus L_{\bar1}$, with product denoted by $[\cdot,\cdot]$, such that 
\begin{align}
[x, y] &= - \eta_{x,y} [y,x],  \\
[x,[y,z]] &= [[x,y],z] + \eta_{x,y} [y,[x,z]],
\end{align}
for any homogeneous elements $x,y,z\in L$. If $\chr\FF \neq 2,3$, the definition of \emph{Lie superalgebra} coincides with the definition of weak Lie superalgebra, and otherwise some additional conditions are required. In particular, if $\chr\FF = 3$, a Lie superalgebra is a weak Lie superalgebra such that $[x,[x,x]] = 0$ for any odd element $x$. (Since we always assume $\chr\FF \neq 2$, we are not concerned with the additional conditions required for characteristic $2$.) By abuse of terminology in the present paper, when we say ``Lie superalgebra'' we will always mean ``weak Lie superalgebra''.

Recall that if $\cA$ is an associative superalgebra, then $\cA$ becomes a Lie superalgebra with the Lie \emph{superbracket} 
\begin{equation} \label{superbracketAssociative}
[x,y] := xy - \eta_{x,y} yx
\end{equation}
for any homogeneous elements $x,y\in\cA$. In particular, if $V$ is a vector superspace, then $\End(V)$ becomes a Lie superalgebra with the Lie superbracket, which is denoted by $\gl(V_{\bar0}|V_{\bar1})$; similarly, the Lie superalgebra of supermatrices $\cM_{(m|n)}(\FF)$ with the Lie superbracket is denoted by $\gl(m|n)$.

Given a superalgebra $\cA$, a \emph{superderivation of degree $a$} is a homogeneous linear map $d \colon \cA \to \cA$ of degree $a\in\ZZ_2$ such that
\begin{equation}
d(xy) = d(x)y + \eta_{d,x} x d(y)
\end{equation}
for any homogeneous $x,y\in\cA$, where we denote $\varepsilon(d) := a$. The vector space of superderivations of $\cA$ is a Lie subsuperalgebra of $\gl(m|n)$, where $m$ and $n$ are respectively the even and odd dimensions of $\cA$.

\bigskip

Let $L$ be a Lie superalgebra. An \emph{$L$-supermodule} is a vector superspace $M = M_{\bar0} \oplus M_{\bar1}$ with a bilinear map $L \times M \to M$, $(x, v) \mapsto x \cdot v$, 
such that $L_a \cdot M_b \subseteq M_{a+b}$ for any $a,b\in\ZZ_2$ and
\begin{equation} \label{supermodule}
[x, y] \cdot v = x \cdot (y \cdot v) - \eta_{x,y} y \cdot (x \cdot v)
\end{equation}
for any homogeneous $x,y\in L$, $v\in M$.

\bigskip

Let $L$ be a Lie superalgebra and $b \colon L \times L \to \FF$ a bilinear form. Then we will say that $b$ is \emph{homogeneous of degree $0$}, or \emph{even}, if $b(x,y) = 0$ for any homogeneous elements $x,y\in L$ with $\varepsilon(x) + \varepsilon(y) \neq \bar0$. On the other hand, if 
\begin{equation} \label{invariant}
b([x,y],z) = b(x,[y,z])
\end{equation}
for any $x,y,z\in L$, then $b$ is called \emph{invariant}. Besides, if
\begin{equation} \label{supersymmetric_b}
b(x,y) = \eta_{x,y} b(y,x)
\end{equation}
for any homogeneous elements $x,y\in L$, then we say that $b$ is \emph{supersymmetric}.

\bigskip

Let $L$ be a finite-dimensional Lie superalgebra and $M$ a finite-dimensional $L$-supermodule. Recall that the dual space $M^*$ inherits a dual $\ZZ_2$-grading such that the duality bilinear form is even. We will usually denote the dual pairing by $\langle\cdot,\cdot\rangle \colon M^* \times M \to \FF$. The \emph{dual} (or \emph{left-dual}) \emph{$L$-supermodule} of $M$ is the $\ZZ_2$-graded vector space $M^*$ with the dual action $x \cdot f$ given by
\begin{equation} \label{dualityAction}
\langle x \cdot f, v \rangle =
(x \cdot f)(v) := -\eta_{x,f} f(x\cdot v)
= -\eta_{x,f} \langle f, x\cdot v \rangle
\end{equation}
for any homogeneous $x\in L$, $f\in M^*$, $v\in M$. We will denote it as $M^\gets$ or $M^*$ (there is also a notion of right-dual supermodule). 

Given a homogeneous $\varphi\in\End(M)$, its \emph{dual} (or \emph{left-dual}) map, denoted by $\varphi^*$ (or $\varphi^\gets$), is defined by
\begin{equation}
\langle \varphi^\gets(f), v \rangle = \eta_{\varphi, f} \langle f, \varphi(v) \rangle,
\end{equation}
for any homogeneous $f\in M^*$, $v\in M$. There is an analogous notion of \emph{right-dual}, denoted $\varphi^\to$.

\bigskip

For a Lie supermodule $(L,M)$, define
\begin{equation}
\Aut(L,M) := \{ (\varphi, h) \in \Aut(L) \times \GL^{\bar0}(M) \med
h(x \cdot v) = \varphi(x) \cdot h(v) \; \forall x\in L, \; v\in M\}.
\end{equation}
The automorphism group scheme $\bAut(L,M)$ is defined similarly.

\subsection{Generalized Jordan superpairs}

Recall that a \emph{trilinear pair} is a pair of vector spaces $\cV = (\cV^-, \cV^+)$ with a pair of trilinear maps $\{\cdot,\cdot,\cdot\}^\sigma \colon \cV^\sigma \times \cV^{-\sigma} \times \cV^\sigma \to \cV^\sigma$, $\sigma \in \{+, -\}$. Denote
\begin{equation} 
D^\sigma_{x,y}(z) := \{x,y,z\}^\sigma
\end{equation}
for $x,z\in \cV^\sigma$, $y\in \cV^{-\sigma}$, $\sigma = \pm$. The superscript $\sigma$ may be omitted for short.

Let $G$ be an abelian group and $\cV$ a trilinear pair; consider two decompositions of vector spaces $\Gamma^{\sigma} \colon \cV^\sigma = \bigoplus_{g\in G} \cV_g^{\sigma}$; then we will say that $\Gamma=(\Gamma^-,\Gamma^+)$ is a {\em $G$-grading on $\cV$} if $\lbrace \cV^{\sigma}_g, \cV^{-\sigma}_h, \cV^{\sigma}_k \rbrace \subseteq \cV^{\sigma}_{g+h+k}$ for any $g,h,k\in G$ and $\sigma\in \lbrace +,- \rbrace$. The vector space $\cV^-_g \oplus \cV^+_g$ is called the {\em homogeneous component of degree} $g$. If $0\neq x\in \cV^{\sigma}_g$ we say $x$ is {\em homogeneous of degree $g$}. For homogeneous elements $x,y$ of degrees $g,h$, the map $D^\sigma_{x,y}$ is homogeneous of degree $g+h$.

\bigskip

A \emph{trilinear superpair} is just a $\ZZ_2$-graded trilinear pair $\cV = (\cV^-, \cV^+)$; in particular, $\cV^-$ and $\cV^+$ are vector superspaces. A \emph{generalized Jordan superpair} is a trilinear superpair $\cV$ where we have that
\begin{equation} \label{defGJSP}
[D^\sigma_{x,y}, D^\sigma_{z,w}] = 
D^\sigma_{D^\sigma_{x,y}z, w} - \eta_{x,y,z} D^\sigma_{z, D^{-\sigma}_{y,x}w}.
\end{equation}
for any homogeneous $x,z\in \cV^\sigma$, $y,w\in \cV^{-\sigma}$, and $\sigma = \pm$ (here $[\cdot,\cdot]$ denotes the Lie superbracket in \eqref{superbracketAssociative}).

\bigskip

Let $\cV$ be a generalized Jordan superpair, $D = (D^-, D^+) \in \End(\cV^-) \times \End(\cV^+)$, and fix $a\in\ZZ_2$. We say that $D$ is a \emph{superderivation of degree $a$} of $\cV$, of parity $\varepsilon(D) := a$, if $D^\sigma \cV^\sigma_b \subseteq \cV^\sigma_{a+b}$ for any $\sigma=\pm$, $b\in\ZZ_2$, and
\begin{equation} \label{derivGJSPdef1}
D^\sigma(\{x,y,z\}) = \{D^\sigma(x),y,z\} + \eta_{D, x} \{x,D^{-\sigma}(y),z\}
+ \eta_{D, D_{x,y}} \{x,y,D^\sigma(z)\}
\end{equation}
for any homogeneous $x,z\in\cV^\sigma$, $y\in\cV^{-\sigma}$.

\bigskip

Given a generalized Jordan superpair $\cV$ and homogeneous elements $x\in\cV^-$, $y\in\cV^+$, denote
\begin{equation}
\nu(x,y) := (D^-_{x,y}, - \eta_{x,y} D^+_{y,x}) \in\End(\cV^-) \times \End(\cV^+),
\end{equation}
and \begin{equation} \label{nuAntisymmetry}
\nu(y, x) := - \eta_{x,y} \nu(x, y).
\end{equation}
Write $\cV_{a} := \cV^-_a \oplus \cV^+_a$ for $a\in\ZZ_2$. Recall that the \emph{inner structure (Lie) superalgebra} of a generalized Jordan superpair is the Lie superalgebra
\begin{equation}
\instr(\cV) := \lspan\{ \nu(x,y) \med x\in\cV^-, y\in\cV^+ \} \leq \gl(\cV_{\bar0} | \cV_{\bar1}),
\end{equation}
and its elements are superderivations called \emph{inner superderivations} of $\cV$.

\bigskip

Let $\cV = (\cV^-, \cV^+)$ be a generalized Jordan superpair with a bilinear form $\langle\cdot,\cdot\rangle \colon \cV^- \times \cV^+ \to \FF$. We say that $\langle\cdot,\cdot\rangle$ is \emph{superinvariant} (or \emph{left-superinvariant}) if
\begin{equation} \label{superinvariant}
\langle D_{x,y}z, w \rangle = \eta_{x,y,z} \langle z, D_{y,x}w \rangle
\end{equation}
for any homogeneous elements $x,z\in\cV^-$, $y,w\in\cV^+$.
If $\langle x, y \rangle = 0$ for any homogeneous $x\in\cV^-$, $y\in\cV^+$ such that $\varepsilon(x) \neq \varepsilon(y)$, then the bilinear form is said to be \emph{homogeneous of degree $0$} or \emph{even}. We say that $\langle\cdot,\cdot\rangle$ is \emph{supersymmetric} (or \emph{left-supersymmetric}), if we have
\begin{equation}\label{symmetricForm}
\begin{split}
\langle D_{x,y}z, w \rangle &= \eta_{D_{x,y}, D_{z,w}} \langle D_{z,w}x, y \rangle, \\
\langle x, D_{y,z}w \rangle &= \eta_{D_{x,y}, D_{z,w}} \langle z, D_{w,x}y \rangle
\end{split}
\end{equation}
for any homogeneous $x,z\in\cV^-$, $y,w\in\cV^+$.

\bigskip

A \emph{homomorphism} $\varphi \colon \cV \to \cW$ of trilinear superpairs (in particular, for generalized Jordan superpairs) is a pair of even linear maps $\varphi = (\varphi^-, \varphi^+)$, with $\varphi^\sigma \colon \cV^\sigma \to \cW^\sigma$, such that $\varphi^\sigma(\{x,y,z\}^\sigma) = \{ \varphi^\sigma(x), \varphi^{-\sigma}(y), \varphi^\sigma(z) \}^\sigma$ for any $x,z\in\cV^\sigma$, $y\in\cV^{-\sigma}$, $\sigma = \pm$. The automorphism group and the automorphism group scheme of $\cV$ will be denoted, respectively, as $\Aut(\cV)$ and $\bAut(\cV)$. Note that $\bAut(\cV)(R) := \Aut_R(\cV_R)$ where we denote $\cV_R := (\cV^-_R, \cV^+_R)$, where $\cV^\sigma_R := \cV^\sigma \otimes R$.

\subsection{Faulkner construction} \hfill\\
Now we will recall some facts that we need from the Faulkner construction \cite[Lemma~1.1]{F73}, \cite[Prop.3.2]{A22}, and the Faulkner correspondence \cite[Th.3.6]{A22}.

\begin{notation}
If $L$ is a Lie superalgebra and $b \colon L \times L \to \FF$ is a nondegenerate even invariant supersymmetric bilinear form, then we will say that $(L,b)$ is a \emph{metric Lie superalgebra}. If $(L,M)$ is a Lie supermodule and $(L,b)$ is a metric Lie superalgebra, then we will say that $(L,M,b)$ is a \emph{metric Lie supermodule}. Let $\MLSM$ denote the class of finite-dimensional metric Lie supermodules. We will also denote by $\MFLSM$ (or $\MFLM$ for the non-super case) the subclass of elements $(L,M,b)$ of $\MLSM$ such that $(L,M)$ is faithful. The subgroup of elements of $\Aut(L,M)$ preserving the bilinear form $b$ will denoted by $\Aut(L,M,b)$, and similarly we can consider the subgroup scheme $\bAut(L, M, b) \leq \bAut(L, M)$.

If $\cV$ is a generalized Jordan superpair and $\langle\cdot,\cdot\rangle \colon \cV^- \times \cV^+ \to \FF$ is a nondegenerate even superinvariant supersymmetric bilinear form, then we will say that $(\cV, \langle\cdot,\cdot\rangle)$ is a \emph{metric generalized Jordan superpair}. We will denote by $\MGJSP$ (or $\MGJP$ for the non-super case) the class of finite-dimensional metric generalized Jordan superpairs. The subgroup of elements of $\Aut(\cV)$ preserving the bilinear form $\langle\cdot,\cdot\rangle$ will be denoted by $\Aut(\cV, \langle\cdot,\cdot\rangle)$, and similarly we can consider the subgroup scheme $\bAut(\cV, \langle\cdot,\cdot\rangle) \leq \bAut(\cV)$.

We recall that the Faulkner correspondence gives a bijective correspondence (for the isomorphism classes) between $\MFLSM$ and $\MGJSP$ (which restricts to a bijection between $\MFLM$ and $\MGJP$); also, for corresponding objects we have $\bAut(L,M,b) \simeq \bAut(\cV, \langle\cdot,\cdot\rangle)$. On the other hand, the Faulkner construction sends each object in $\MLSM$ (the supermodules are not required to be faithful) to another object in $\MGJSP$.
\end{notation}

\begin{notation}
Given $(L,M,b) \in \MLSM$, the Faulkner construction produces an object $(\cV, \langle\cdot,\cdot\rangle) \in \MGJSP$, defined as follows. The vector superspaces are given by $\cV = \cV_{L,M} := (M^*, M)$. The bilinear form is just the dual pairing of Lie supermodules $\langle\cdot,\cdot\rangle \colon M^* \times M \to \FF$, $(f,v) \mapsto \langle f,v \rangle$, and we will use the convention $\langle v, f \rangle = \eta_{f, v}\langle f,v \rangle$. By nondegeneracy of $b$, a term $[f, v]\in L$ is defined for each $v\in M$, $f\in M^*$, if we impose
\begin{equation} \label{bilinearFormsCorrespondence}
b\big(x, [f, v] \big) = \langle x \cdot f, v \rangle,
\end{equation}
for any $x \in L$, $v \in M = \cV^+$, $f \in M^* = \cV^-$. Similarly we define $[v, f]$, which satisfies $[v,f] = -\eta_{f,v}[f,v]$. Then the triple products of $\cV$ are defined by
\begin{equation} \label{action.tripleproducts.correspondence}
\{f,v,g\}^- := [f, v] \cdot g, \qquad \{v,f,w\}^+ := [v, f] \cdot w,
\end{equation}
for homogeneous $v,w \in M = \cV^+$, $f,g \in M^* = \cV^-$. The \emph{inner structure algebra} of $(L,M,b)$, denoted $\instr(L,M)$, is the Lie subsuperalgebra of $L$ spanned by the elements of the form $[f,v]$. Also, the map 
\begin{equation}\begin{split} \label{epimorphismFaulkner}
\Upsilon \colon \instr(L, M) &\longrightarrow \instr(\cV_{L,M}) \leq \gl(M^* \oplus M), \\
\quad [f,v] &\longmapsto \nu(f,v) := (D_{f,v}, - \eta_{f,v} D_{v,f}).
\end{split}\end{equation}
defines an epimorphism of Lie superalgebras. If $(L,M)$ is faithful, then $L = \instr(L, M) \cong \instr(\cV_{L,M})$ (see \cite[Prop.3.3]{A22}).
\end{notation}

\begin{remark} \label{erratumNu}
In the present paper, we will use the Faulkner construction (but not the correspondence) to transfer the definitions of alternating and symmetric superpowers. In the proof of \cite[Prop.4.3-2)]{A22}, there is a notational error where terms of the form $[f,v]$ should appear instead of $\nu(f,v)$, inside the bilinear form $b$, and it is also necessary to apply the epimorphism $\Upsilon$ in \eqref{epimorphismFaulkner} at the end of the proof; we will revisit that result in Prop.\ref{generalTensorProductPairs} below.
\end{remark}

\begin{notation}
We will denote by $\bG_m := \bGL_1$ the multiplicative group scheme and by $\bmu_n$ the group scheme of the $n$-th roots of unity. Fix $\cV, \cW \in \MGJSP$. Recall that for each $\lambda \in R^\times$ there is an automorphism $c_\lambda = (c_\lambda^-, c_\lambda^+) \in \Aut_R(\cV_R)$ defined by $c_\lambda^\sigma(x) := \lambda^{\sigma 1}x$ for $x\in\cV^\sigma_R$, and we can identify $\bG_m(R) = R^\times \simeq \{c_\lambda \med \lambda \in R^\times\} \leq \bAut(\cV)(R) = \Aut_R(\cV_R)$. Thus we can consider the central product relative to $\bG_m$ (see definition for groups in \cite[Chap.2, p.29]{G80})
\begin{equation}
\bAut(\cV) \otimes_{\bG_m} \bAut(\cW) := \big(\bAut(\cV) \times \bAut(\cW)\big)/\bT_1,
\end{equation}
where $\bT_1(R) := \{(c_\lambda, c^{-1}_\lambda) \in \Aut_R(\cV_R) \times \Aut_R(\cW_R) \med
\lambda \in R^\times \}$.
\end{notation}

\begin{remark}
Let the map $\cF \colon \MLSM \longrightarrow \MGJSP$ denote the Faulkner construction. The Faulkner correspondence (bijection for isomorphism classes) is given by
\begin{align*}
\widetilde{\cF} := \cF|_{\MFLSM} \colon \MFLSM \longrightarrow \MGJSP.
\end{align*}
Let $\kappa := \widetilde{\cF}^{-1} \circ \cF \colon \MLSM \longrightarrow \MFLSM$, thus
\begin{align} \label{eqFaulknerTransfer}
\widetilde{\cF} \circ \kappa = \cF.
\end{align}
Any $n$-ary operator $\theta \colon \MLSM^n \longrightarrow \MLSM$ induces a map for faithful supermodules,
\begin{align}
\theta_F := \kappa \circ \theta \colon \MFLSM^n \longrightarrow \MFLSM,
\end{align}
which can be transferred through $\widetilde{\cF}$ to a map
\begin{align}
\widetilde{\theta} := \widetilde{\cF} \circ \theta_F
  \circ (\widetilde{\cF}^{-1} \times \ldots \times \widetilde{\cF}^{-1})
\colon \MGJSP^n \longrightarrow \MGJSP.
\end{align}
From \eqref{eqFaulknerTransfer}, we get that
\begin{align}
\widetilde{\theta} = \cF \circ \theta \circ (\widetilde{\cF}^{-1} \times \ldots \times \widetilde{\cF}^{-1}),
\end{align}
which is used in further sections to transfer the definitions of tensors from $\MLSM$ to $\MGJSP$.
\end{remark}

\begin{notation}
Recall from \cite[Notation~4.6]{A22} that each $1$-dimensional object in $\MGJSP$ is determined uniquely by a parameter $\alpha = (\lambda, a)\in G := \FF \times \ZZ_2$ (where $a$ corresponds to the parity), thus we can denote it by $\cV_\alpha$. Given $\cV \in \MGJSP$, the tensor superproduct $\cV^{[\alpha]} := \cV \otimes \cV_\alpha$ in $\MGJSP$ will be referred to as a \emph{tensor-shift by $\alpha$} of $\cV$. Recall also that, up to isomorphism, we can identify $\cV^{[\alpha]}$ with the vector superspaces of $\cV$, with shifted degrees $\varepsilon_{[\alpha]}(x) := \varepsilon(x) + a$, with the shifted metric
\begin{equation} \label{general-shifted-metrics}
\begin{split}
& \langle x^+,y^- \rangle_{[\alpha]}^+ := \eta_{a,x} \langle x^+,y^- \rangle^+ = \eta_{a,y} \langle x^+,y^- \rangle^+, \\
& \langle x^-,y^+ \rangle_{[\alpha]}^- := \eta_a\eta_{a,x} \langle x^-,y^+ \rangle^- = \eta_a\eta_{a,y} \langle x^-,y^+ \rangle^-,
\end{split}
\end{equation}
and with shifted triple products given by
\begin{equation} \label{general-shifted-products} \begin{split}
\{ x, y, z \}^+_{[\alpha]} &:= \eta_{a,y} (\{ x, y, z \}^+ + \lambda \langle x, y \rangle z), \\
\{ x, y, z \}^-_{[\alpha]} &:= \eta_a \eta_{a,y} (\{ x, y, z \}^- + \lambda \langle x, y \rangle z),
\end{split} \end{equation}
where $\eta_{a,y} := (-1)^{a \varepsilon(y)}$, $\eta_a := (-1)^a$. In the non-super case, we take $\alpha \equiv \lambda \in \FF$.

Note that for the metric in \eqref{general-shifted-metrics} we are using the notation $\langle x,y \rangle^- := \langle x,y \rangle$ and $\langle y,x \rangle^+ := \eta_{x,y} \langle x,y \rangle^-$ to distinguish both maps of the metric (which arise by ``the rule of signs'' in the Faulkner construction).
\end{notation}

\section{Tensor superproducts of Lie supermodules} \label{section.modules.tensor} \hfill

For $i = 1,\dots,n$, let $M_i$ be an $L_i$-supermodule for a Lie superalgebra $L_i$. We may write $\otimes_i v_i$ to denote a pure tensor $v_1 \otimes \dots \otimes v_n \in\bigotimes_{i=1}^n M_i$. Consider the trivial action of $L_i$ on $M_j$ for $i \neq j$. Then, it is well-known that the pair $(\bigoplus_{i=1}^n L_i, \bigotimes_{i=1}^n M_i)$ defines a Lie supermodule with the action determined by
\begin{equation} \label{def.supermodule}
x \cdot (\otimes_i v_i) := \sum_{i=1}^n \Bigl( \prod_{k < i} \eta_{x, v_k} \Bigr)
v_1 \otimes \dots \otimes (x \cdot v_i) \otimes \dots \otimes v_n
\end{equation}
for each homogeneous $x \in \bigoplus_{i=1}^n L_i$, $v_i \in M_i$, where the parity map is given by $\varepsilon(\otimes_i v_i) := \sum_{i=1}^n \varepsilon(v_i)$. We will refer to $(\bigoplus_{i=1}^n L_i, \bigotimes_{i=1}^n M_i)$ as the \emph{(general) tensor superproduct} of the Lie supermodules $(L_i, M_i)$. In particular, we can consider the \emph{$n$-th (general) tensor superpower} $(\bigoplus_{i=1}^n L, \bigotimes^n M)$ of a Lie supermodule $(L,M)$.

Consider now the case where $L_i = L$ are the same Lie superalgebra. It is also well-known that eq.\eqref{def.supermodule}, for homogeneous $x\in L$ and $v_i \in M_i$, defines an $L$-supermodule on $\bigotimes_{i=1}^n M_i$, where the parity is defined again by $\varepsilon(\otimes_i v_i) := \sum_{i=1}^n \varepsilon(v_i)$. When this action is considered, we will denote $\bigboxtimes_{i=1}^n M_i := \bigotimes_{i=1}^n M_i$, and the pair $(L, \bigboxtimes_{i=1}^n M_i)$ will be referred to as the \emph{restricted tensor superproduct} of the Lie supermodules $(L, M_i)$. In particular, we can consider the \emph{$n$-th restricted tensor superpower} $(L, \bigboxtimes^n M)$ of a Lie supermodule $(L,M)$. The aim of the ``$\bigboxtimes$'' notation is just to avoid ambiguity with the general tensor superproduct, and for dealing with vector superspaces it is unnecessary.
Note that for a general tensor superproduct $(\bigoplus_{i=1}^n L, \bigotimes_{i=1}^n M_i)$ and the diagonal Lie subsuperalgebra
\begin{equation} \label{diagonal.restriction}
\widetilde{L} = \text{diag}\Bigl(\bigoplus_{i=1}^n L \Bigr)
:= \{(x,\dots,x) \med x\in L\} \leq \bigoplus_{i=1}^n L,
\end{equation}
the restricted action defines a Lie supermodule $(\widetilde{L}, \bigotimes_{i=1}^n M_i)$ which is isomorphic to the restricted tensor superproduct $(L, \bigboxtimes_{i=1}^n M_i)$, which follows from the isomorphism $\varphi \colon L \to \widetilde{L}$, $x \mapsto (x,\dots,x)$. In particular, for objects $(L,M_i,b) \in \MLSM$, we have $(\widetilde{L}, \bigotimes_{i=1}^n M_i, \widetilde{b}) \cong (L, \bigboxtimes_{i=1}^n M_i, nb)$ with $\widetilde{b} = \perp_{i=1}^n b$, because $\widetilde{b}(\varphi(x), \varphi(y)) = nb(x,y)$.
However, the pairs of the last isomorphism are metric only if $\chr\FF$ does not divide $n$ (which ensures nondegeneracy of the bilinear forms $\widetilde{b}$ and $nb$).
The binary operators $\boxtimes$ and $\otimes$ are both associative. (Some authors refer to $\otimes$ and $\boxtimes$ as outer and inner tensor products, and denote them with different notation \cite{M12}.)

Given Lie supermodules $(L_i, M_i)$, for $i = 1,\dots,n$, the Lie supermodule $(\bigoplus_{i=1}^n L_i, \bigoplus_{i=1}^n M_i)$ is called their \emph{direct sum}, where we consider the trivial action of $L_i$ on $M_j$ for $i\neq j$. Similarly, the direct sum of objects $(L_i,M_i,b_i) \in \MLSM$ is defined by $(\bigoplus_{i=1}^n L_i, \bigoplus_{i=1}^n M_i, \perp_{i=1}^n b_i)$. 

\medskip

\begin{remarks} \label{remarks.referee} \hfill
\begin{itemize}
\item[1)]
Recall that for vector superspaces $V$, $W$, we have the braiding of superspaces $c_{V,W} \colon V \otimes W \to W \otimes V$, $v \otimes w \mapsto \eta_{v,w} w \otimes v$. This determines natural isomorphisms (by permutation of the components and the ``rule of signs'')
\begin{align} \label{permutation.isomorphism}
c^\sigma_{V_1,\ldots,V_n} \colon V_1 \otimes \cdots \otimes V_n &\longrightarrow V_{\sigma^{-1}(1)} \otimes \cdots \otimes V_{\sigma^{-1}(n)},
\end{align}
for any vector superspaces $V_1, \ldots, V_n$ and $\sigma \in S_n$, where the composition of isomorphisms corresponds to the composition of permutations; in particular this defines an action of $S_n$ on $V^{\otimes n}$. Note that if $V, W$ are $L$-supermodules for some Lie superalgebra $L$, then $f = c_{V,W}$ is a supermodule isomorphism, because
\begin{align*}
& f\big(x \cdot (v \otimes w)\big) = 
f\big((x \cdot v) \otimes w + \eta_{x,v} v \otimes (x \cdot w)\big) \\
&\quad= \eta_{x,w} \eta_{v,w} w \otimes (x \cdot v)
+ \eta_{v,w} (x \cdot w) \otimes v
= \eta_{v,w} x \cdot (w \otimes v)
= x \cdot f(v \otimes w),
\end{align*} 
for each homogeneous $x \in L$, $v \in V$, $w \in W$. Consequently, the natural maps $c^\sigma_{V_1,\ldots,V_n}$ are supermodule isomorphisms.

\item[2)]
Given an $L$-supermodule $M$ over a Lie superalgebra $L$, since the bilinear pairing $M^* \times M \to \FF$, $(f,v) \mapsto f(v)$, is a pairing of Lie supermodules, it follows that the evaluation map
\begin{align}
\text{eval} \colon M^* \otimes M \longrightarrow \FF, \qquad f \otimes v \longmapsto \langle f,v \rangle := f(v),
\end{align}
is a homomorphism of $L$-supermodules (with the trivial action on $\FF$). Indeed, for each homogeneous $x \in L$, $v \in M$, $f \in M^*$, we have
\begin{align*}
& \text{eval}\big(x \cdot (f \otimes v)\big)
= \text{eval}\big( (x \cdot f) \otimes v + \eta_{x,f} f \otimes (x \cdot v) \big)
= \langle x \cdot f, v \rangle + \eta_{x,f} \langle f, x \cdot v \rangle \\
&\quad= -\eta_{x,f} \langle f, x \cdot v \rangle + \eta_{x,f} \langle f, x \cdot v \rangle
= 0 = x \cdot \text{eval}(f \otimes v).
\end{align*}
\end{itemize}
\end{remarks}

\smallskip

The following result is well-known, but the authors have not found a good reference.

\begin{proposition} \label{prop.tensor.pairing} \hfill
\begin{itemize}
\item[1)] For $i=1,\dots,n$, let $L_i$ be a Lie superalgebra and $M_i$ a finite-dimensional $L_i$-supermodule, and consider the dual pairings of $L_i$-supermodules $\langle\cdot,\cdot\rangle \colon M_i^* \times M_i \to \FF$. Then the bilinear form $\langle \cdot,\cdot \rangle \colon \bigotimes_{i=1}^n M_i^* \times \bigotimes_{i=1}^n M_i \to \FF$ defined by
\begin{equation} \label{eq.tensor.superproduct}
\langle \otimes_i f_i, \otimes_j v_j \rangle
:= \Bigl( \prod_{1 \leq j < i \leq n} \eta_{f_i, v_j} \Bigr)
\Bigl( \prod_{i=1}^n \langle f_i, v_i \rangle \Bigr),
\end{equation}
produces a dual pairing of $\widetilde{L}$-supermodules for the (general) tensor superproduct, where $\widetilde{L} := \bigoplus_{i=1}^n L_i$. Consequently, $\langle\cdot,\cdot\rangle$ defines an isomorphism of $\widetilde{L}$-supermodules $\bigotimes_{i=1}^n M_i^* \cong (\bigotimes_{i=1}^n M_i)^*$, $f \mapsto \langle f,\cdot\rangle$.

\item[2)] Let $L$ be a Lie superalgebra. For $i=1,\dots,n$, let $M_i$ be a finite-dimensional $L$-supermodule, and consider the dual pairings of $L$-supermodules $\langle\cdot,\cdot\rangle \colon M_i^* \times M_i \to \FF$. Then the bilinear form $\langle \cdot,\cdot \rangle \colon \bigboxtimes_{i=1}^n M_i^* \times \bigboxtimes_{i=1}^n M_i \to \FF$ defined by the formula \eqref{eq.tensor.superproduct} produces a dual pairing of $L$-supermodules for the restricted tensor superproduct. Consequently, $\langle\cdot,\cdot\rangle$ defines an isomorphism of $L$-supermodules $\bigboxtimes_{i=1}^n M_i^* \cong (\bigboxtimes_{i=1}^n M_i)^*$, $f \mapsto \langle f,\cdot\rangle$.
\end{itemize}
\end{proposition}
\begin{proof}
Let $F \colon (\bigotimes_{i=1}^n M_i^*) \otimes (\bigotimes_{i=1}^n M_i) \to \FF$ be the composition of the isomorphism $M_1^* \otimes \cdots \otimes M_n^* \otimes M_1 \otimes \cdots \otimes M_n \to M_1^* \otimes M_1 \otimes M_2^* \otimes M_2 \otimes \cdots \otimes M_n^* \otimes M_n$ defined as in \eqref{permutation.isomorphism} (here, the permutation $\sigma$ is a shuffle that moves $M_i$ after the corresponding $M_i^*$), and the tensor product of the evaluation maps $M_i^* \otimes M_i \to \FF$. Note that $F$ is exactly the linear map associated to the bilinear pairing $\langle\cdot,\cdot\rangle$ in eq.\eqref{eq.tensor.superproduct}. By Remarks~\ref{remarks.referee}, it follows that $F$ is a composition of supermodule homomorphisms. Therefore $F$ is a supermodule homomorphism, thus the associated bilinear pairing in eq.\eqref{eq.tensor.superproduct} is a pairing of Lie supermodules.
Case 2) follows by restriction to the diagonal in eq.\eqref{diagonal.restriction}.
\end{proof}

\begin{df}
The bilinear form $\langle \cdot,\cdot \rangle \colon \bigotimes^n_{i=1} M_i^* \times \bigotimes^n_{i=1} M_i \to \FF$ in Proposition~\ref{prop.tensor.pairing} is called the \emph{tensor superproduct} of the dual pairings of supermodules $M_i^* \times M_i \to \FF$.
\end{df}

\section{Tensor superpowers of metric generalized Jordan superpairs} \label{section.pairs.tensor} \hfill

In this section, restricted tensor superpowers in $\MGJSP$ are introduced. Also, (general) tensor superproducts in $\MGJSP$ are revisited (these were studied in \cite{A22}).

\begin{df}
Take objects $(\cV_i, \langle\cdot,\cdot\rangle) \in\MGJSP$ and their corresponding objects $(L_i,M_i,b_i) \in\MFLSM$, for $i=1,\dots,n$. By the Faulkner construction, the tensor superproduct of supermodules $(\bigoplus_{i=1}^n L_i, \bigotimes_{i=1}^n M_i, \perp_{i=1}^n b_i)$ defines an object $(\bigotimes_{i=1}^n \cV_i, \langle\cdot,\cdot\rangle) \in\MGJSP$ that we will call the \emph{(general) tensor superproduct} of the objects 
$(\cV_i, \langle\cdot,\cdot\rangle)$. Given $(\cV, \langle\cdot,\cdot\rangle) \in\MGJSP$ and its corresponding object $(L,M,b) \in\MFLSM$, the restricted tensor superpower $(L, \bigboxtimes^n M, b)$ defines an object $(\bigboxtimes^n \cV, \langle\cdot,\cdot\rangle) \in\MGJSP$ that we will call the \emph{restricted tensor superpower} of $(\cV, \langle\cdot,\cdot\rangle)$. We can also consider the \emph{(general) tensor superpower} $\bigotimes^n \cV$ of $\cV$. By the Faulkner correspondence, we can consider the corresponding operations in $\MFLSM$, which define a (general) tensor superproduct and a restricted tensor superpower in $\MFLSM$, although we are not interested in these.
\end{df}

\begin{proposition} \label{restrictedTensorPowerPairs}
Let $\cV$ be a nonzero object in $\MGJSP$, $1 < n \in\NN$, and consider the restricted tensor superpower $\cW = \bigboxtimes^n \cV$. Then:
\begin{itemize}
\item[1)] The bilinear form $\langle\cdot,\cdot\rangle$ on $\cW$ is given by the tensor superpower of the bilinear form of $\cV$, that is,
$$ \langle \otimes_i f_i, \otimes_i v_i \rangle
= \prod_{i=1}^n \Bigl( \prod_{k < i} \eta_{f_i, v_k} \Bigr)
\langle f_i, v_i \rangle
= \Bigl( \prod_{1 \leq j < i \leq n} \eta_{f_i, v_j} \Bigr)
\Bigl( \prod_{i=1}^n \langle f_i, v_i \rangle \Bigr). $$
\item[2)] The generators of $\instr(\cW)$ are of the form
$$ \nu(\otimes_i f_i, \otimes_i v_i) = 
\Bigl( \prod_{1 \leq j < i \leq n} \eta_{f_i, v_j} \Bigr)
\sum_{i=1}^n \Bigl( \prod_{k \neq i} \langle f_k, v_k \rangle \Bigr)
\nu(f_i, v_i). $$
\item[3)] The triple products on $\cW$, for homogeneous elements $x_i, z_i \in \cV^\sigma$, $y_i \in \cV^{-\sigma}$, are given by
\begin{align*}
& \{\otimes_i x_i, \otimes_i y_i, \otimes_i z_i \} = \\
& \quad = \Bigl( \prod_{1 \leq j < i \leq n} \eta_{x_i, y_j} \Bigr)
\sum_{i,j=1}^n \Bigl( \prod_{t < j} \eta_{x_i, z_t} \eta_{y_i, z_t} \Bigr)
\Bigl( \prod_{k \neq i} \langle x_k, y_k \rangle \Bigr)
z_1 \otimes\dots\otimes \{x_i, y_i, z_j\} \otimes\dots\otimes z_n.
\end{align*}
\item[4)] $\bAut(\cV, \langle\cdot,\cdot\rangle) / \bmu_n \lesssim \bAut(\cW, \langle\cdot,\cdot\rangle)$.
\end{itemize}
\end{proposition}
\begin{proof}
1) Consequence of the isomorphism of supermodules $(\bigboxtimes^n M)^* \cong \bigboxtimes^n M^*$ produced by the bilinear pairing of $L$-supermodules $\bigboxtimes^n M^* \times \bigboxtimes^n M \to \FF$, which is given by the tensor superpower of the bilinear pairing of $L$-supermodules $M^* \times M \to \FF$.

2) Take homogeneous elements $x\in \instr(\cV)$, $f\in\cV^-$, $v\in\cV^+$. We claim that
\begin{equation} \label{auxEq1}
\eta_{x,f}\eta_{x,v} \langle f, v \rangle = \langle f, v \rangle.
\end{equation}
Indeed, if $\varepsilon(f) = \varepsilon(v)$ we have that $\eta_{x,f}\eta_{x,v} = 1$, otherwise we have that $\langle f, v \rangle = 0$ because $\langle\cdot,\cdot\rangle$ is even, and in both cases the claim follows. Then
\begin{align*}
b\big(&x, [\otimes_i f_i, \otimes_i v_i]\big) =_\eqref{bilinearFormsCorrespondence}
\quad \langle x \cdot (\otimes_i f_i), \otimes_j v_j \rangle \\
&= \sum_i \Bigl( \prod_{k < i} \eta_{x, f_k} \Bigr)
\langle f_1 \otimes\dots\otimes (x \cdot f_i) \otimes\dots\otimes f_n, \otimes_j v_j \rangle \\
&= \sum_i \Bigl( \prod_{k < i} \eta_{x, f_k} \Bigr)
\Bigl( \prod_{j=1}^n \prod_{k < j} \eta_{f_j, v_k} \Bigr)
\Bigl( \prod_{k < i} \eta_{x, v_k} \Bigr)
\Bigl( \prod_{k \neq i} \langle f_k, v_k \rangle \Bigr)
\langle x \cdot f_i, v_i \rangle
=_\eqref{bilinearFormsCorrespondence} \\
&= \Bigl( \prod_j \prod_{k < j} \eta_{f_j, v_k} \Bigr)
\sum_i \Bigl( \prod_{k < i} \eta_{x, f_k} \eta_{x, v_k} \Bigr)
\Bigl( \prod_{k \neq i} \langle f_k, v_k \rangle \Bigr)
b\big( x, [f_i, v_i] \big) =_\eqref{auxEq1} \\
&= \Bigl( \prod_{j < i} \eta_{f_i, v_j} \Bigr)
\sum_i \Bigl( \prod_{k \neq i} \langle f_k, v_k \rangle \Bigr)
b\big( x, [f_i, v_i] \big) \\
&= b\big( x, 
\Bigl( \prod_{j < i} \eta_{f_i, v_j} \Bigr)
\sum_i \Bigl( \prod_{k \neq i} \langle f_k, v_k \rangle \Bigr)
[f_i, v_i] \big),
\end{align*}
and since $b$ is nondegenerate we get
$[\otimes_i f_i, \otimes_i v_i] = \Bigl( \prod_{j < i} \eta_{f_i, v_j} \Bigr)
\sum_i \Bigl( \prod_{k \neq i} \langle f_k, v_k \rangle \Bigr)
[f_i, v_i]$,
and then applying the epimorphism $\Upsilon$ in \eqref{epimorphismFaulkner}, the property follows.

3) The triple product for homogeneous elements is given by
\begin{align*}
\{& \otimes_i x_i, \otimes_j y_j, \otimes_k z_k \} =
\nu(\otimes_i x_i, \otimes_j y_j) \cdot (\otimes_k z_k) \\
&= \Bigl( \prod_{j < i} \eta_{x_i, y_j} \Bigr)
\sum_i \Bigl( \prod_{k \neq i} \langle x_k, y_k \rangle \Bigr)
\nu(x_i, y_i) \cdot (\otimes_k z_k) \\
&= \Bigl( \prod_{j < i} \eta_{x_i, y_j} \Bigr)
\sum_i \Bigl( \prod_{k \neq i} \langle x_k, y_k \rangle \Bigr)
\sum_j \Bigl( \prod_{t < j} \eta_{x_i, z_t} \eta_{y_i, z_t} \Bigr) \cdot \\
& \qquad\qquad \cdot z_1 \otimes\dots\otimes (\nu(x_i, y_i) \cdot z_j) \otimes\dots\otimes z_n \\
&= \Bigl( \prod_{j < i} \eta_{x_i, y_j} \Bigr)
\sum_{i,j} \Bigl( \prod_{t < j} \eta_{x_i, z_t} \eta_{y_i, z_t} \Bigr)
\Bigl( \prod_{k \neq i} \langle x_k, y_k \rangle \Bigr)
z_1 \otimes\dots\otimes \{x_i, y_i, z_j\} \otimes\dots\otimes z_n. \\
\end{align*}

4) Let $\varphi \in \Aut_R(\cV_R, \langle\cdot,\cdot\rangle)$ and consider the pair of maps $\varphi^{\otimes n} := \big((\varphi^-)^{\otimes n}, (\varphi^+)^{\otimes n}\big)$ where $\otimes = \otimes_R$. We will first show that $\langle\cdot,\cdot\rangle$ is $\varphi^{\otimes n}$-invariant. Given homogeneous elements $f_i \in \cV^-$, $v_i \in \cV^+$, we have that
\begin{align*}
\langle &(\varphi^-)^{\otimes n}( \otimes_i f_i ), (\varphi^+)^{\otimes n}( \otimes_i v_i ) \rangle
= \langle \otimes_i \varphi^-(f_i), \otimes_i \varphi^+(v_i) \rangle \\
&= \prod_{i=1}^n \Bigl( \prod_{k < i} \eta_{\varphi^-(f_i), \varphi^+(v_k)} \Bigr)
\langle \varphi^-(f_i), \varphi^+(v_i) \rangle \\
&= \prod_{i=1}^n \Bigl( \prod_{k < i} \eta_{f_i, v_k} \Bigr)
\langle f_i, v_i \rangle
= \langle \otimes_i f_i, \otimes_i v_i \rangle,
\end{align*}
which proves that $\langle\cdot,\cdot\rangle$ is $\varphi^{\otimes n}$-invariant. On the other hand,
\begin{align*}
\{ & (\varphi^-)^{\otimes n}(\otimes_i x_i), (\varphi^+)^{\otimes n}(\otimes_j y_j),
(\varphi^-)^{\otimes n}(\otimes_k z_k) \}
= \{ \otimes_i \varphi^-(x_i), \otimes_j \varphi^+(y_j), \otimes_k \varphi^-(z_k) \} \\
&= \Bigl( \prod_{j < i} \eta_{\varphi^-(x_i), \varphi^+(y_j)} \Bigr)
\sum_{i,j} \Bigl( \prod_{t < j} \eta_{\varphi^-(x_i), \varphi^-(z_t)}
\eta_{\varphi^+(y_i), \varphi^-(z_t)} \Bigr)
\Bigl( \prod_{k \neq i} \langle \varphi^-(x_k), \varphi^+(y_k) \rangle \Bigr)
\cdot \\
& \qquad \cdot\varphi^-(z_1)
\otimes\dots\otimes \{\varphi^-(x_i), \varphi^+(y_i), \varphi^-(z_j)\}
\otimes\dots\otimes \varphi^-(z_n) \\
&= \Bigl( \prod_{j < i} \eta_{x_i, y_j} \Bigr)
\sum_{i,j} \Bigl( \prod_{t < j} \eta_{x_i, z_t} \eta_{y_i, z_t} \Bigr)
\Bigl( \prod_{k \neq i} \langle x_k, y_k \rangle \Bigr)
\varphi^-(z_1) \otimes\dots\otimes \varphi^- \big(\{ x_i, y_i, z_j \}\big)
\otimes\dots\otimes \varphi^-(z_n) \\
&= (\varphi^-)^{\otimes n}(\{ \otimes_i x_i, \otimes_j y_j, \otimes_k z_k \}),
\end{align*}
and similarly we get 
$$ \{ (\varphi^+)^{\otimes n}(\otimes_i x_i), (\varphi^-)^{\otimes n}(\otimes_j y_j),
(\varphi^+)^{\otimes n}(\otimes_k z_k) \}
= (\varphi^+)^{\otimes n}(\{ \otimes_i x_i, \otimes_j y_j, \otimes_k z_k \}). $$
We have proven that $\varphi^{\otimes n}\in\Aut_R(\cW_R,\langle\cdot,\cdot\rangle)$. Thus we have a morphism of affine group schemes 
\begin{equation}
\Phi^\otimes_n \colon \bAut(\cV, \langle\cdot,\cdot\rangle) \to \bAut(\cW, \langle\cdot,\cdot\rangle),
\end{equation}
whose kernel is given by $\ker(\Phi^\otimes_n)_R = \{ c_\lambda \med \lambda\in R^\times, \; \lambda^n = 1 \} \simeq \bmu_n(R)$, and the result follows.
\end{proof}

\bigskip

The following result is a minor generalization of the case $n=2$ in \cite[Prop.4.3]{A22}, and includes (general) tensor superpowers as a particular case. (See also Remark~\ref{erratumNu} above.)

\begin{proposition} \label{generalTensorProductPairs}
Let $\cV_i$ be nonzero objects in $\MGJSP$ for $i = 1,\dots,n$ and consider the tensor superproduct $\cW = \bigotimes_{i=1}^n\cV_i$. Then:
\begin{itemize}
\item[1)] The bilinear form $\langle\cdot,\cdot\rangle$ on $\cW$ is given by the tensor superproduct of the bilinear forms of the $\cV_i$, that is,
$$ \langle \otimes_i f_i, \otimes_i v_i \rangle
= \prod_{i=1}^n \Bigl( \prod_{k < i} \eta_{f_i, v_k} \Bigr)
\langle f_i, v_i \rangle
= \Bigl( \prod_{1 \leq j < i \leq n} \eta_{f_i, v_j} \Bigr)
\Bigl( \prod_{i=1}^n \langle f_i, v_i \rangle \Bigr). $$
\item[2)] The generators of $\instr(\cW)$ are of the form
$$ \nu(\otimes_i f_i, \otimes_i v_i) = 
\Bigl( \prod_{1 \leq j < i \leq n} \eta_{f_i, v_j} \Bigr)
\sum_{i=1}^n \Bigl( \prod_{k \neq i} \langle f_k, v_k \rangle \Bigr)
\nu(f_i, v_i). $$
\item[3)] The triple products on $\cW$, for homogeneous elements $x_i, z_i \in \cV^\sigma$, $y_i \in \cV^{-\sigma}$, are given by
\begin{align*}
& \{\otimes_i x_i, \otimes_i y_i, \otimes_i z_i \} = \\
& \quad = \Bigl( \prod_{1 \leq j < i \leq n} \eta_{x_i, y_j} \Bigr)
\sum_{i=1}^n \Bigl( \prod_{t < i} \eta_{x_i, z_t} \eta_{y_i, z_t} \Bigr)
\langle x_1, y_1 \rangle z_1 \otimes\dots\otimes \{x_i, y_i, z_i\}
\otimes\dots\otimes \langle x_n, y_n \rangle z_n.
\end{align*}
\item[4)] For the automorphism group schemes, we have
$$\bigotimes_{i=1}^n {}_{\bG_m} \bAut(\cV_i, \langle\cdot,\cdot\rangle)
\leq \bAut(\cW, \langle\cdot,\cdot\rangle).$$
\end{itemize}
\end{proposition}
\begin{proof}
The four properties follow easily by induction from the case $n=2$ in \cite[Prop.4.3]{A22}. Note that 2) also follows from the calculations in the proof of Prop.~\ref{restrictedTensorPowerPairs}-2). Property 3) can also be proven using the calculations in the proof of Prop.~\ref{restrictedTensorPowerPairs}-3) as a shortcut, where we get that the triple product for homogeneous elements is given by
\begin{align*}
\{& \otimes_i x_i, \otimes_j y_j, \otimes_k z_k \} = \dots = \\
&= \Bigl( \prod_{j < i} \eta_{x_i, y_j} \Bigr)
\sum_i \Bigl( \prod_{k \neq i} \langle x_k, y_k \rangle \Bigr)
\sum_j \Bigl( \prod_{t < j} \eta_{x_i, z_t} \eta_{y_i, z_t} \Bigr)
z_1 \otimes\dots\otimes (\nu(x_i, y_i) \cdot z_j) \otimes\dots\otimes z_n \\
&= \Bigl( \prod_{j < i} \eta_{x_i, y_j} \Bigr)
\sum_i \Bigl( \prod_{k \neq i} \langle x_k, y_k \rangle \Bigr)
\Bigl( \prod_{t < i} \eta_{x_i, z_t} \eta_{y_i, z_t} \Bigr)
z_1 \otimes\dots\otimes (\nu(x_i, y_i) \cdot z_i) \otimes\dots\otimes z_n \\
&= \Bigl( \prod_{j < i} \eta_{x_i, y_j} \Bigr)
\sum_i \Bigl( \prod_{t < i} \eta_{x_i, z_t} \eta_{y_i, z_t} \Bigr)
\langle x_1, y_1 \rangle z_1 \otimes\dots\otimes \{x_i, y_i, z_i\}
\otimes\dots\otimes \langle x_n, y_n \rangle z_n. \\
\end{align*}
\end{proof}

\section{Alternating superpowers of Lie supermodules} \label{section.supermodules.alternating} \hfill

Throughout this section, unless otherwise stated, we will assume that $M$ is a nonzero finite-dimensional $L$-supermodule, where $L$ is a Lie superalgebra. Note that the results for supermodules also hold for vector superspaces, since these can be thought as Lie supermodules over $L = 0$. The results where the $L$-action is unimportant will be stated in terms of a finite-dimensional vector superspace $V$.

\begin{dfs}
Let $n \geq 2$ and consider the $L$-supermodule $\bigboxtimes^n M$. For $1\leq i < j \leq n$, let $\tau_{ij} \in \End(\bigotimes^n M )$ the linear map swapping the $i$-th and $j$-th components of pure tensors, that is, 
\begin{equation} \label{tauSwap}
\tau_{ij}(v_1 \otimes \dots \otimes v_n) := v_1 \otimes \dots \otimes v_{i-1} \otimes v_j \otimes v_{i+1} \otimes \dots \otimes
v_{j-1} \otimes v_i \otimes v_{j+1} \otimes \dots \otimes v_n.
\end{equation}
Consider the vector subsuperspace of $\bigboxtimes^n M$ given by
\begin{equation} \label{submodule_alternating}
\widehat{R}_n = \widehat{R}_n(M) := \lspan\{ \otimes_k v_k + \eta_{v_i, v_{i+1}} \tau_{i,i+1}( \otimes_k v_k )
\med 0 \neq v_k \in M_{\bar0} \cup M_{\bar1}, 1 \leq i < n \}.
\end{equation}
We will also denote $\widehat{R}^*_n = \widehat{R}^*_n(M) := \widehat{R}_n(M^*)$.
It is not too hard to see that
\begin{equation} \label{submodule_alternating2}
\widehat{R}_n = \{ v - \sgn(\sigma) \sigma \cdot v \med v \in \bigboxtimes^n M, \sigma \in S_n \},
\end{equation}
with the action $\sigma \cdot v$ defined as in Remarks~\ref{remarks.referee}-1).

The elements of the vector superspace $\bigwedge^n M := (\bigboxtimes^n M) / \widehat{R}_n$ will be called \emph{alternating supertensors}. Note that if $M$ is even, then $\bigwedge^n M$ is an alternating power of $M$, and if $M$ is odd, then (as a vector space) $\bigwedge^n M$ is a symmetric power of $M$. We will use the convention $\bigwedge^1 M := M$. The projection of a pure supertensor $v_1 \otimes \dots \otimes v_n$ in $\bigwedge^n M$ will be denoted by $v_1 \wedge \dots \wedge v_n$, or just $\wedge_i v_i$, and referred to as a \emph{pure alternating supertensor}. Note that the parity map of $\bigwedge^n M$ is given by $\varepsilon(\wedge_i v_i) := \sum_i \varepsilon(v_i)$ for homogeneous elements $v_i \in M$. We will say that $\wedge_i v_i$ is \emph{parity-ordered} if there exists $k\in\{0,1,\dots,n\}$ such that $v_i\in M_{\bar0}$ for $i \leq k$ and $v_j\in M_{\bar1}$ for $j > k$.

Note that $\widehat{R}_n$ is an $L$-subsupermodule of $\bigboxtimes^n M$, because for elements $x \in L$, $v \in \bigboxtimes^n M$, $\sigma \in S_n$, we have
$$ x \cdot \big( v - \sgn(\sigma) \sigma \cdot v \big)
= x \cdot v - \sgn(\sigma) \sigma \cdot (x \cdot v) \in \widehat{R}_n. $$
Consequently, $\bigwedge^n M$ becomes an $L$-supermodule with the action given by
\begin{equation} \label{eq.action.alternating}
x \cdot (\wedge_i v_i) := \sum_{i=1}^n  \Bigl( \prod_{k < i} \eta_{x, v_k} \Bigr)
v_1 \wedge \dots \wedge (x \cdot v_i) \wedge \dots \wedge v_n
\end{equation}
for each homogeneous $x \in L$, $v_i \in M$. We will refer to $(L, \bigwedge^n M)$ as the \emph{$n$-th alternating (or exterior) superpower} of the Lie supermodule $(L,M)$.
\end{dfs}

\begin{remark} \label{remark.restrictions.alternating}
Let $d_{\bar0} := \dim M_{\bar0}$, $d_{\bar1} := \dim M_{\bar1}$ be the even and odd dimensions of the Lie supermodule $M$, and $d := \dim M = d_{\bar0} + d_{\bar1}$. Note that for an even Lie supermodule $M$ (i.e., a Lie module), $\bigwedge^n M$ is just the usual alternating power, so that $\bigwedge^n M = 0$ for $n > \dim M$. Without further mention, we will only consider the cases with $\bigwedge^n M \neq 0$, i.e., we will assume
$\boxed{\text{$n \leq d_{\bar0}$ if $d_{\bar1} = 0$}}$.
On the other hand, we will also assume
$\boxed{\text{$\chr\FF = 0$ or $\chr\FF > n$ if $d_{\bar1} > 0$}}$,
which will grant nondegeneracy for certain bilinear form $\widehat{F}$ defined in \eqref{eq.F.hat}.
\end{remark}

\medskip

\begin{notation} \label{notationDetper1}
Let $L$ be a Lie superalgebra and $M$ a finite-dimensional $L$-supermodule. Our next goal is to construct a bilinear pairing of $L$-supermodules $\bigwedge^n M^* \times \bigwedge^n M \to \FF$.

Let $\langle\cdot,\cdot\rangle \colon \bigotimes^n M^* \times \bigotimes^n M \to \FF$ be the bilinear form defined as in \eqref{eq.tensor.superproduct}, where we consider $M_i = M$ for each $i = 1,\ldots,n$. Recall that $S_n$ acts on $\bigotimes^n M$ by means of the automorphisms in \eqref{permutation.isomorphism}. We claim that
\begin{align} \label{action.property}
\langle \sigma \cdot f, v \rangle = \langle f, \sigma^{-1} \cdot v \rangle
\end{align}
for all $f \in \bigotimes^n M^*$, $v \in \bigotimes^n M$, $\sigma \in S_n$. To prove the claim, first consider the case $n = 2$. Let $\tau = (1 \quad 2) \in S_2$, fix homogeneous elements $v,w \in M$, $f,g \in M^*$, and note that
\begin{align*}
& \langle \tau \cdot (f \otimes g), v \otimes w \rangle
= \langle \eta_{f,g} g \otimes f,  v \otimes w \rangle
= \eta_{f,g} \eta_{f,v} \langle g,v \rangle \langle f,w \rangle \\
&\quad= \eta_{w,v} \eta_{w,g} \langle g,v \rangle \langle f,w \rangle
= \langle f \otimes g, \eta_{v,w} w \otimes v \rangle
= \langle f \otimes g, \tau \cdot (v \otimes w) \rangle.
\end{align*}
For an arbitrary $n$ and any elementary transposition $\tau = (i \quad i+1) \in S_n$, the calculation is analogous; and since elementary transpositions generate $S_n$, the claim follows easily.

Now define a new bilinear form by
\begin{align} \label{def.F.no.hat}
F \colon \bigboxtimes^n M^* \times \bigboxtimes^n M \longrightarrow \FF, \qquad F(f,v) := \sum_{\sigma \in S_n} \sgn(\sigma) \langle f, \sigma \cdot v \rangle,
\end{align}
for all $f \in \bigotimes^n M^*$, $v \in \bigotimes^n M$. Now we claim that
\begin{align} \label{property.F.permutations}
F(f, \sigma \cdot v) = \sgn(\sigma) F(f,v) = F(\sigma \cdot f, v),
\end{align}
for all $f \in \bigotimes^n M^*$, $v \in \bigotimes^n M$, $\sigma \in S_n$.
Indeed, the left equality follows from a straightforward calculation, and the right equality follows using \eqref{def.F.no.hat} and \eqref{action.property}.
Since $S_n$ acts by automorphisms of $L$-supermodules, we have that
\begin{align*}
&F(f, x \cdot v) = \sum_{\sigma \in S_n} \sgn(\sigma) \langle f, \sigma \cdot (x \cdot v) \rangle
= \sum_{\sigma \in S_n} \sgn(\sigma) \langle f, c^\sigma_{M,\ldots,M} (x \cdot v) \rangle
= \sum_{\sigma \in S_n} \sgn(\sigma) \langle f, x \cdot c^\sigma_{M,\ldots,M}(v) \rangle \\
&\quad= \sum_{\sigma \in S_n} \sgn(\sigma) \langle f, x \cdot (\sigma \cdot v) \rangle
= -\eta_{x,f} \sum_{\sigma \in S_n} \sgn(\sigma) \langle x \cdot f, \sigma \cdot v \rangle
= -\eta_{x,f} F(x \cdot f, v),
\end{align*}
for all $x \in L$, $f \in \bigotimes^n M^*$, $v \in \bigotimes^n M$; in other words, $F$ is also a bilinear pairing of $L$-supermodules.

By \eqref{property.F.permutations}, it follows that $F$ satisfies the properties
\begin{align*}
F\big( f, v - \sgn(\sigma)\sigma \cdot v \big) = 0 = F\big( f - \sgn(\sigma)\sigma \cdot f, v\big),
\end{align*}
so that $F(\bigotimes^n M^*, \widehat{R}_n) = 0 = F(\widehat{R}^*_n, \bigotimes^n M)$.
Thus $F$ induces a bilinear map
\begin{equation} \label{eq.F.hat}
\widehat{F} \colon \bigwedge^n M^* \times \bigwedge^n M \to \FF,
\end{equation}
which is also a pairing of $L$-supermodules. (We will show in Notation~\ref{notation.basis.alternating} that $\widehat{F}$ is nondegenerate.)
\end{notation}

\medskip

\begin{notation} \label{notationDetper2}
Our aim now is to obtain an explicit expression for the bilinear form $\widehat{F}$ in \eqref{eq.F.hat}.

For $\alpha = (\alpha_1,\dots,\alpha_n) \in \ZZ_2^n$, consider the homogeneous subspace of $\bigotimes^n M$ given by $\bigotimes_\alpha M := \bigotimes_{i=1}^n M_{\alpha_i}$ and note that $\bigotimes^n M = \bigoplus_{\alpha\in\ZZ_2^n} \bigotimes_\alpha M$. Given $\alpha\in\ZZ_2^n$, define the ordered sets $\iota_{\bar0}(\alpha) := \{i_1,\dots,i_k\}$ and $\iota_{\bar1}(\alpha) := \{j_1,\dots,j_{n-k}\}$, where $i_1 < \dots < i_k$ are the subscripts $i$ where $\alpha_i = \bar0$ and $j_1 < \dots < j_{n-k}$ are the subscripts $j$ where $\alpha_j = \bar1$. For $0\leq k\leq n$, let $\bigotimes^{(k,n-k)}M$ denote the direct sum of the subspaces $\bigotimes_\alpha M$ such that $\alpha$ has $\bar0$ appearing in $k$ entries and $\bar1$ appearing in $n-k$ entries. We will denote the image of $\bigotimes^{(k,n-k)}M$ on the quotient $\bigwedge^n M$ by $\bigwedge^{(k, n-k)} M$. (For $k > d_{\bar0}$, we have $\bigwedge^{(k,n-k)}M = 0$, which can be proven as for alternating powers in the non-super case.) The relations in $\widehat{R}_n$ show that $\bigwedge^{(k,n-k)}M$ is also the image of $\bigotimes^k M_{\bar0} \otimes \bigotimes^{n-k} M_{\bar1}$ on the quotient.

Let $\widehat{R}_n^{(k,n-k)} := \widehat{R}_n \cap \bigotimes^{(k,n-k)}M$, then it is easy to see that $\widehat{R}_n = \bigoplus_{k=0}^n \widehat{R}_n^{(k,n-k)}$, and it follows that
\begin{equation} \label{decompositionAlternating}
\bigwedge^n M = \bigoplus_{k=0}^{\min(d_{\bar0},n)} \bigwedge^{(k, n-k)} M.
\end{equation}
Let $\widehat{F}^{(k,n-k)}$ denote the restriction of $\widehat{F}$ to $\bigwedge^{(k, n-k)} M^* \times \bigwedge^{(k, n-k)} M$. By \eqref{def.F.no.hat} and \eqref{decompositionAlternating}, we have that
\begin{equation} \label{decompositionAlternatingF}
\widehat{F} = \perp_{k=0}^{\min(d_{\bar0},n)} \widehat{F}^{(k,n-k)}.
\end{equation}

For convenience, we introduce the $\detper$ operators, defined as a combination of the determinant and the permanent. For $0\leq k\leq n$ and $A = (a_{ij})_{ij} \in \cM_n(\FF)$, set
\begin{equation} \label{detper}
\detper_{k,n-k}(A) :=
\det\Bigl( (a_{ij})_{1 \leq i,j \leq k} \Bigr)
\per\Bigl( (a_{ij})_{k < i,j \leq n} \Bigr),
\end{equation}
where we use the convention $\det(\emptyset) = 1 = \per(\emptyset)$ for the ``empty submatrix''.
For $0 \leq k \leq n$, denote
\begin{align} \label{def.omega_n}
\omega_k := (-1)^{\sum_{0 \leq i < k} i} = (-1)^{\frac{k(k-1)}{2}} = (-1)^{k \choose 2}.
\end{align}
Note that
\begin{equation} \label{omegaProp}
\omega_k\omega_{k+1} = (-1)^k.
\end{equation}

Fix $0\leq k \leq \min(d_{\bar0}, n)$. We will now prove that $\widehat{F}^{(k,n-k)}$ is given, for parity-ordered elements, by
\begin{equation} \label{eq.F.hat.restriction} \begin{split}
\widehat{F}^{(k,n-k)} \colon &
\bigwedge^{(k, n-k)} M^* \times \bigwedge^{(k, n-k)} M \longrightarrow \FF, \\
& (\wedge_i f_i, \wedge_j v_j) \longmapsto
\omega_{n-k} \detper_{k,n-k}\Bigl( \big( \langle f_i, v_j \rangle \big)_{ij} \Bigr).
\end{split} \end{equation}
Fix parity-ordered elements $\wedge_i f_i \in \bigwedge^{(k, n-k)} M^*$, $\wedge_j v_j \in \bigwedge^{(k, n-k)} M$. Let $f'_1 = \otimes_{i \leq k} f_i$, $f'_2 = \otimes_{i > k} f_i$, $v'_1 = \otimes_{i \leq k} v_i$, $v'_2 = \otimes_{i > k} v_i$.
Identify $S_k \times S_{n-k}$ with the subgroup of $S_n$ that fixes the sets $\{1,\dots,k\}$ and $\{k+1,\dots,n\}$.
Note that the nonzero terms contributing to the sum $\widehat{F}(\wedge_i f_i, \wedge_j v_j)$ correspond to the permutations in $S_k \times S_{n-k}$, thus
\begin{align*}
&\widehat{F}(\wedge_i f_i, \wedge_j v_j) = F(\otimes_i f_i, \otimes_j v_j)
= \sum_{\sigma = \sigma_1\sigma_2 \in S_k \times S_{n-k}}
\sgn(\sigma) \langle \otimes_i f_i, \sigma \cdot \otimes_j v_j \rangle \\
&\quad= \Bigg( \sum_{\sigma_1 \in S_k} \sgn(\sigma_1) \langle f'_1, \sigma_1 \cdot v'_1 \rangle \Bigg)
\Bigg( \sum_{\sigma_2 \in S_{n-k}} \sgn(\sigma_2)
\langle f'_2, \sigma_2 \cdot v'_2 \rangle \Bigg) \\
&\quad= \Bigg( \sum_{\sigma_1 \in S_k} \sgn(\sigma_1)
\prod_{1 \leq i,j \leq k} \langle f_i, v_{\sigma_1^{-1}(j)} \rangle \Bigg)
\Bigg( \sum_{\sigma_2 \in S_{n-k}} \sgn(\sigma_2) 
\Big( \sgn(\sigma_2) \omega_{n-k} \Big)
\prod_{k < i,j \leq n} \langle f_i, v_{\sigma_2^{-1}(j)} \rangle \Bigg) \\
&\quad= \omega_{n-k} 
\det\Bigl( (\langle f_i, v_j \rangle)_{1 \leq i,j \leq k} \Bigr)
\per\Bigl( (\langle f_i, v_j \rangle)_{k < i,j \leq n} \Bigr)
= \omega_{n-k} \detper_{k,n-k}\Bigl( \big( \langle f_i, v_j \rangle \big)_{ij} \Bigr).
\end{align*}
We have proven eq.\eqref{eq.F.hat.restriction}.

In general, for parity-ordered elements $\wedge_i f_i \in\bigwedge^{(r, n-r)} M^*$, $\wedge_j v_j \in\bigwedge^{(s, n-s)} M$, we have
\begin{equation} \label{eq.F.hat.unrestricted}
\widehat{F}( \wedge_i f_i, \wedge_j v_j )
= \omega_{n-r} \detper_{r,n-r}\Bigl( \big( \langle f_i, v_j \rangle \big)_{ij} \Bigr)
= \omega_{n-s} \detper_{s,n-s}\Bigl( \big( \langle f_i, v_j \rangle \big)_{ij} \Bigr),
\end{equation}
because $\detper_{r,n-r}\Bigl( \big( \langle f_i, v_j \rangle \big)_{ij} \Bigr) = 0
= \detper_{s,n-s}\Bigl( \big( \langle f_i, v_j \rangle \big)_{ij} \Bigr)$ if $r \neq s$.
\end{notation}

\medskip

\begin{notation} \label{notation.basis.alternating}
Finally, we will describe some properties of the bilinear form $\widehat{F}$ in \eqref{eq.F.hat}.

For $0\leq k \leq \min(d_{\bar0},n)$, consider the family of ordered $n$-tuples
\begin{equation}
\altI^{(k,n-k)} = \altI^{(k,n-k)}(M) := \{ I=(i_1,\dots,i_n)
\med 1 \leq i_1 < \dots <i_k \leq d_{\bar0} < i_{k+1} \leq \dots \leq i_n \leq d \},
\end{equation}
and
\begin{equation}
\altI_n = \altI_n(M) := \bigcup_{0 \leq k \leq \min(d_{\bar0},n)} \altI^{(k,n-k)}(M).
\end{equation}
The elements of $\altI_n$ and $\altI^{(k,n-k)}$ will be used as sets of ordered indices of supermatrices, where $k$ and $n-k$ correspond to the number of indices coming from the even and odd subspaces, respectively.

Fix a basis $\cB = \{ v_i \}_{i=1}^d$ of $M$ such that $\cB_{\bar0} := \{ v_i \}_{i=1}^{d_{\bar0}}$ and $\cB_{\bar1} := \{ v_i \}_{i=d_{\bar0}+1}^d$ are bases of $M_{\bar0}$ and $M_{\bar1}$, respectively. Let $\cB^* = \{ f_i \}_{i=1}^d$ be the dual basis of $\cB$. Thus
$\cB^*_{\bar0} := \{ f_i \}_{i=1}^{d_{\bar0}}$
and
$\cB^*_{\bar1} := \{ f_i \}_{i=d_{\bar0}+1}^d$
are the dual bases of $\cB_{\bar0}$ and $\cB_{\bar1}$.
The relations given by $\widehat{R}_n$ show that $\bigwedge^n M$ is spanned by the set
\begin{equation} \label{basisAlternating}
\altB_n := \{ \hat{e}_I := \wedge_{i\in I} v_i \med I \in\altI_n \},
\end{equation}
and similarly $\bigwedge^n M^*$ is spanned by
\begin{equation} \label{dualBasisAlternating}
\altB^*_n := \{ \hat{e}^*_I := \wedge_{i\in I} f_i \med I \in\altI_n \}.
\end{equation}

Note that $\widehat{F}(\hat{e}^*_I, \hat{e}_J) = 0$ for all $\hat{e}^*_I \in \altB^*_n$, $\hat{e}_J \in \altB_n$ with $I \neq J$. 
On the other hand, given $I \in\altI_n$, let $M_I = (m_1,\ldots,m_t)$ denote the sequence of multiplicities of the entries of $I$ (thus $\sum_i m_i = n$), and $k$ the number of even coordinates $v_i$ for the element $\hat{e}_I = \wedge_{i\in I} v_i $. Then, define
\begin{equation} \label{kappa.coef}
\kappa_I := \omega_{n-k} \prod_{m \in M_I} m! = \omega_{n-k} m_1! \cdots m_t!,
\end{equation}
with $\omega_{n-k}$ as in \eqref{def.omega_n}.
Then, by \eqref{eq.F.hat.unrestricted} it is clear that
\begin{equation} \label{dual.pairing.alternating}
\widehat{F}(\hat{e}^*_I, \hat{e}_J) = \kappa_I \delta_{I,J} = \kappa_J \delta_{I,J}.
\end{equation}

Since we are assuming that $\chr\FF = 0$ or $\chr\FF > n$ whenever $d_{\bar1} > 0$, it follows from \eqref{dual.pairing.alternating} that $\widehat{F}$ defines a dual pairing (i.e., $\widehat{F}$ is nondegenerate), and also that $\altB_n$ and $\altB^*_n$ are bases of $\bigwedge^n M$ and $\bigwedge^n M^*$, respectively. Unfortunately, $\altB_n$ and $\altB^*_n$ are not $\widehat{F}$-dual bases in general. 
It is clear that
\begin{equation}
\dim \bigwedge^n M = |\altB_n| = |\altI_n|
= \sum_{k=0}^{\min(d_{\bar0},n)} {d_{\bar0} \choose k} {d_{\bar1} + (n - k) - 1 \choose n-k}.
\end{equation}
Since $\widehat{F}$ is a dual pairing of supermodules, it defines an isomorphism of $L$-supermodules $\bigwedge^n M^* \cong (\bigwedge^n M)^*$, $f \mapsto \widehat{F}(f,\cdot)$.
\end{notation}

\begin{df}
The Lie supermodules duality map $\widehat{F}$ defined in \eqref{eq.F.hat} will be referred to as the \emph{$n$-th alternating superpower} of the corresponding Lie supermodules duality map $M^* \times M \to \FF$. It will be denoted by $\langle\cdot,\cdot\rangle$ in further sections.
\end{df}

\begin{notation}
Let $L$ be a Lie superalgebra. Given a homomorphism of finite-dimensional Lie supermodules, $h \in \Hom_L(M,N)$, it is clear that $h^{\otimes n} \in \Hom_L(\bigotimes^n M, \bigotimes^n N)$, and $h^{\otimes n}\big(\widehat{R}_n(M)\big) \subseteq \widehat{R}_n(N)$ because $h$ is even. Thus $h^{\otimes n}$ induces an element $h^{\wedge n} \in \Hom_L(\bigwedge^n M, \bigwedge^n N)$, given by $h^{\wedge n}(\wedge_i x_i) = \wedge_i h(x_i)$ for any elements $x_i\in M$. It is also clear that the composition of two homomorphisms, $h_1$ and $h_2$, satisfies the property
$(h_2 \circ h_1)^{\wedge n} = h_2^{\wedge n} \circ h_1^{\wedge n}$.

Consider now the case $M = N$, i.e., $h \in \End_L(M)$. Let $h^*$ be the dual map of $h$ for the bilinear pairing $\langle\cdot,\cdot\rangle \colon M^* \times M \to \FF$, and $(h^{\wedge n})^*$ the $\widehat{F}$-dual map of $h^{\wedge n}$. Then for parity-ordered elements $\wedge_i f_i \in \bigwedge^{(k,n-k)} M^*$ and $\wedge_j v_j \in \bigwedge^{(k,n-k)} M$, we get
\begin{align*}
\widehat{F}\big( \wedge_i & f_i, h^{\wedge n}(\wedge_j v_j) \big)
= \widehat{F}\big( \wedge_i f_i, \wedge_j h(v_j) \big)
= \omega_{n-k} \detper_{k,n-k}\Bigl(\big( \langle f_i,h(v_j) \rangle \big)_{ij}\Bigr) \\
&= \omega_{n-k} \detper_{k,n-k}\Bigl(\big( \langle h^*(f_i),v_j \rangle \big)_{ij}\Bigr)
= \widehat{F}\big( \wedge_i h^*(f_i), \wedge_j v_j \big)
= \widehat{F}\big( (h^*)^{\wedge n}(\wedge_i f_i), \wedge_j v_j \big),
\end{align*}
thus $(h^*)^{\wedge n} = (h^{\wedge n})^*$.
\end{notation}

\smallskip

\begin{notation} \label{notationSgnAlt}
Let $0\leq k \leq \min(d_{\bar0}, n)$ and identify $S_k \times S_{n-k}$ with the subgroup of $S_n$ that fixes the sets $\{1,\dots,k\}$ and $\{k+1,\dots,n\}$. Let $I = (i_1,\dots,i_n) \in\altI^{(k,n-k)}$ and consider the parity-ordered element $\hat{e}_I := \wedge_{i\in I} v_i = \wedge_t v_{i_t} \in\altB_n$. For each permutation $\sigma\in S_n$, let $\sgnalt_{k,n-k}(\sigma)$ denote the sign defined by
\begin{equation} \label{sgnAlt}
\wedge_{i\in I} v_i = \sgnalt_{k,n-k}(\sigma) \wedge_t v_{i_{\sigma(t)}},
\end{equation}
and note that $\sgnalt_{k,n-k}(\sigma\rho) = \sgn(\sigma)$ for $\sigma\in S_k$, $\rho\in S_{n-k}$.
\end{notation}

\smallskip

\begin{notation} \label{notationPermutationRepetitions}
Let $V$ be a finite-dimensional vector superspace. Consider the action of $S_n$ on $S = \{1,\dots,d\}^n$ given by $\sigma\big( (i_1,\dots,i_n) \big) := (i_{\sigma^{-1}(1)},\dots,i_{\sigma^{-1}(n)})$. Fix $I = (i_1,\dots,i_n) \in\altI_n(V) \subseteq S$ and take $H = \Stab_{S_n}(I)$, $O = \Orb_{S_n}(I)$. By abuse of notation, we will denote by $S_n(I)$ any left transversal of $H$ in $S_n$. Note that by the Orbit-Stabilizer Theorem, the multiset $\{ \sigma(I) \med \sigma\in S_n \}$ is the orbit $O$, where each element has multiplicity $|H|$, whereas the multiset $\{ \sigma(I) \med \sigma\in S_n(I) \}$ has the same elements with multiplicity one. We will use the notation $S_n(I)$ for parametrizations of $O$ without repetitions. In particular, $S_n(I)$ will be used as a set of indices for sums where we do not want to repeat elements of $O$ (if $\chr\FF = 0$, this is equivalent to iterate on $S_n$ and divide the sum by $|H|$). If all entries in $I$ are different, then we have $S_n(I) = S_n$.
\end{notation}

\smallskip

\begin{notation}
Let $V$ be a finite-dimensional vector superspace. Let $I = (i_1,\dots,i_n) \in\altI^{(p,n-p)}(V)$, $J = (j_1,\dots,j_n) \in\altI^{(q,n-q)}(V)$ with $0\leq p,q \leq \min(d_{\bar0},n)$. Take an even supermatrix $A = (a_{ij})_{ij} = \diag(A_{\bar0}, A_{\bar1}) \in \cM_{(d_{\bar0}|d_{\bar1}) \times (r|s)}(\FF) = \cM_{d \times n}(\FF)$, with $A_{\bar0} \in\cM_{d_{\bar0} \times r}(\FF)$ and $A_{\bar1} \in\cM_{d_{\bar1} \times s}(\FF)$. Identify as subgroup $S_p \times S_{n-p} \leq S_n$ (as in Notation~\ref{notationSgnAlt}). We define the \emph{(alternating) $(I,J)$-superminor} of $A$ by
\begin{equation}
\widehat{\cM}_{I,J}(A) := \sum_{\sigma \in S_n(I)}
\sgnalt_{p,n-p}(\sigma) \prod_{t=1}^n a_{i_{\sigma(t)}, j_t}.
\end{equation}
We may also refer to alternating superminors as \emph{$\detper$-superminors}.
By the block structure of $A$ it follows that
\begin{equation}
\widehat{\cM}_{I,J}(A) = 
\Bigl( \sum_{\sigma \in S_p} \sgn(\sigma)
\prod_{t=1}^p a_{i_{\sigma(t)}, j_t} \Bigr)
\Bigl( \sum_{\rho \in S_{n-p}(I)}
\prod_{t=p+1}^n a_{i_{\rho(t)}, j_t} \Big).
\end{equation}
Note that $\widehat{\cM}_{I,J}(A) = 0$ if $p \neq q$. If $I = J$, then the superminor will be said to be a \emph{principal} superminor.

Consider the case with $n = d-1$. For $1 \leq i,j \leq n$, let $I_i = (1,\dots,i-1,i+1,\dots,d)$, $I_j = (1,\dots,j-1,j+1,\dots,d)$. Then the term $\widehat{\cM}_{ij}(A) := \widehat{\cM}_{I_i,I_j}(A)$ will be called the \emph{(alternating) $(i,j)$-superminor} of $A$. Also, $(i,j)$-superminors will be referred to as \emph{(alternating) first superminors}.
\end{notation}

\bigskip

The following result generalizes \cite[Chap.3, \S8.5, Prop.9 \& Prop.10]{B89}.

\begin{proposition} \label{altBourbaki}
Let $V$ and $V'$ be finite-dimensional vector superspaces. Let $\cB = \{ v_i \}_{i=1}^d$ and $\cB' = \{ v'_i \}_{i=1}^{d'}$ be parity-ordered bases of $V$ and $V'$, respectively. Consider the associated bases $\altB_n = \{\hat{e}_I\}_{I\in\altI_n(V)}$ and $\altB'_n = \{\hat{e}'_{I'}\}_{I'\in\altI_n(V')}$ of $\bigwedge^n V$ and $\bigwedge^n V'$, defined as in \eqref{basisAlternating} by using $\cB$ and $\cB'$. Then:

\begin{itemize}
\item[$1)$] Take a parity-ordered subset $\{w_j\}_{j=1}^n \subseteq V$, with $\{w_j\}_{j=1}^r \subseteq V_{\bar0}$ and $\{w_j\}_{j=r+1}^n \subseteq V_{\bar1}$ for some $0 \leq r \leq \min(d_{\bar0}, n)$, let $s = n - r$, and set $w_j = \sum_{i=1}^d a_{ij} v_i$. Consider the even supermatrix $A = (a_{ij})_{ij} = \diag(A_{\bar0}, A_{\bar1}) \in \cM_{(d_{\bar0}|d_{\bar1}) \times (r|s)}(\FF) = \cM_{d \times n}(\FF)$, with $A_{\bar0} \in\cM_{d_{\bar0} \times r}(\FF)$ and $A_{\bar1} \in\cM_{d_{\bar1} \times s}(\FF)$. Let $J = (1,\dots,n)$. Then
\begin{equation}
\wedge_j w_j = \sum_{I \in\altI_n(V)} \widehat{\cM}_{I,J}(A) \hat{e}_I,
\end{equation}
where $\widehat{\cM}_{I,J}(A) = 0$ if $I \notin \altI^{(r,n-r)}(V)$.

\item[$2)$] Let $h \colon V \to V'$ be an even homomorphism of vector superspaces, and let $A = (a_{ij})_{ij} = \diag(A_{\bar0}, A_{\bar1})\in \cM_{(d'_{\bar0}|d'_{\bar1}) \times (d_{\bar0}|d_{\bar1})}(\FF) = \cM_{d' \times d}(\FF)$ be its coordinate matrix on the bases $\cB$ and $\cB'$, which is an even supermatrix with $A_{\bar0} \in\cM_{d'_{\bar0} \times d_{\bar0}}(\FF)$ and $A_{\bar1} \in\cM_{d'_{\bar1} \times d_{\bar1}}(\FF)$. Then, the coordinate matrix of $h^{\wedge n}$ in the bases $\altB_n$ and $\altB'_n$ is
\begin{equation} \label{matrixAltPower}
A^{\wedge n} := \big( \widehat{\cM}_{I',I}(A) \big)_{I'\in\altI_n(V'), I \in\altI_n(V)}.
\end{equation}
(Here, $\widehat{\cM}_{I',I}(A) = 0$ if $r \neq s$, where $I\in\altI^{(r,n-r)}(V)$, $I'\in\altI^{(s,n-s)}(V')$.)
\end{itemize}
\end{proposition}
\begin{proof}
\noindent 1) Let $\cB^* = \{f_i\}_{i=1}^d$ be the dual basis of $\cB$, and $\altB^*_n$ the $\widehat{F}$-dual basis of $\altB_n$ constructed from $\cB^*$ as in \eqref{dualBasisAlternating}. Then
\begin{align*}
\wedge_j w_j &= \wedge_j\Bigl( \sum_{i=1}^d a_{ij} v_i \Bigr)
= \sum_{k_1,\dots,k_n} \wedge_t (a_{k_t, t} v_{k_t})
= \sum_{k_1,\dots,k_n} \Bigl( \prod_{t=1}^n a_{k_t, t} \Bigr) \Bigl(\wedge_t v_{k_t}\Bigr).
\end{align*}
For $I = (i_1,\dots,i_n) \in\altI^{(k,n-k)}(V)$, the coefficient of $\wedge_{i\in I}v_i \in \altB_n$ in the coordinates of $\wedge_j w_j$ in $\altB_n$ is given by $\kappa_I^{-1} \widehat{F}\big(\wedge_{i\in I} f_i, \wedge_j w_j\big)$. Besides,
\begin{align*} 
\kappa_I^{-1} \widehat{F}( & \wedge_{i\in I} f_i, \wedge_j w_j)
= \kappa_I^{-1} \widehat{F}\big( \wedge_{i\in I} f_i, \sum_{k_1,\dots,k_n}
\Bigl( \prod_{t=1}^n a_{k_t, t} \Bigr) \Bigl(\wedge_t v_{k_t}\Bigr) \big)
=_\text{$(\star)$, Not.\ref{notationPermutationRepetitions}} \\
&= \kappa_I^{-1} \widehat{F}\big( \wedge_{i\in I} f_i, \sum_{\sigma \in S_n(I)}
\Bigl( \prod_{t=1}^n a_{i_{\sigma(t)},t} \Bigr)\Bigl(\wedge_t v_{i_{\sigma(t)}}\Bigr) \big)
=_{\eqref{sgnAlt}} \\
&= \kappa_I^{-1} \sum_{\sigma \in S_n(I)} \sgnalt_{k,n-k}(\sigma) \widehat{F}\big( \wedge_{i\in I} f_i,
\Bigl( \prod_{t=1}^n a_{i_{\sigma(t)},t} \Bigr) \wedge_{i\in I} v_i \big) \\
&= \sum_{\sigma \in S_n(I)} \sgnalt_{k,n-k}(\sigma) \prod_{t=1}^n a_{i_{\sigma(t)},t}
= \widehat{\cM}_{I,J}(A).
\end{align*}
In equality $(\star)$ we have used eq.\eqref{dual.pairing.alternating}, so that the terms not proportional to $\wedge_{i\in I} v_i$ can be dropped (any permutation of the indices in $I$ has to be included).

\smallskip

\noindent 2) Let $I = (i_1,\dots,i_n) \in\altI^{(k,n-k)}(V)$. Then
\begin{align*}
h^{\wedge n}(\hat{e}_I) = h^{\wedge n} (\wedge_j v_{i_j})
= \wedge_j \big( h(v_{i_j}) \big)
= \wedge_j\Bigl( \sum_{t=1}^{d'} a_{t,i_j} v'_t \Bigr)
= \sum_{I'\in\altI_n(V')} \widehat{\cM}_{I',I}(A) \hat{e}'_{I'},
\end{align*}
where the last equality follows by 1).
\end{proof}

\begin{df} The matrix $A^{\wedge n}$, defined as in \eqref{matrixAltPower} from an even supermatrix
$A = \diag(A_{\bar0}, A_{\bar1})$, will be called the \emph{$n$-th alternating superpower} of $A$. Then $A^{\wedge n}$ can be regarded as an even supermatrix whose rows and columns are indexed by $\altI_n(V')$ and $\altI_n(V)$, respectively.
\end{df}

\begin{notation} \label{notationKernelAlt}
Let $V$ be a finite-dimensional vector superspace. Consider the morphism of affine group schemes $\widehat{\Psi}_n \colon \bGL^{\bar0}(V) \to \bGL^{\bar0}(\bigwedge^n V)$ given by
\begin{equation}
(\widehat{\Psi}_n)_R \colon \bGL^{\bar0}(V)(R) := \GL^{\bar0}_R(V_R)
\longrightarrow \bGL^{\bar0}(\bigwedge^n V)(R) := \GL^{\bar0}_R(\bigwedge^n V_R),
\quad \varphi \longmapsto \varphi^{\wedge n}.
\end{equation}
Given $\varphi \in \GL^{\bar0}_R(V_R)$, it is clear that $\varphi^{\wedge n}$ is even and invertible, with inverse $(\varphi^{\wedge n})^{-1} = (\varphi^{-1})^{\wedge n}$, and therefore $\widehat{\Psi}_n$ is well-defined. Moreover, 
$$\bmu_n(R) \cong \{ r \id_V \med r\in R, \; r^n = 1 \} \leq \ker(\widehat{\Psi}_n)_R \leq \GL^{\bar0}_R(V_R).$$
A description of the affine group scheme $\bker\widehat{\Psi}_n$ is given by the following result.
\end{notation}

\begin{proposition} \label{propKernelAlt}
Let $V$ be a finite-dimensional vector superspace. Let $\widehat{\Psi}_n$ be defined as above.
\begin{itemize}
\item[$1)$] If $V$ is even and has dimension $n$, then $\bker\widehat{\Psi}_n \simeq \bSL_n$.
\item[$2)$] Otherwise, $\bker\widehat{\Psi}_n \simeq \bmu_n$.
\end{itemize}
\end{proposition}
\begin{proof}
Let $\varphi \in \ker(\widehat{\Psi}_n)_R$ for some associative commutative unital $\FF$-algebra $R$.
Notice that Prop.~\ref{altBourbaki} also holds for the scalar extension $V_R$. Let $A = (a_{ij})_{ij} = \diag(A_{\bar0},A_{\bar1}) \in\GL^{\bar0}_{(d_{\bar0}|d_{\bar1})}(R)$ be the coordinate matrix of $\varphi \in \GL^{\bar0}_R(V_R)$ in a parity-ordered basis $\cB$, and consider the associated basis $\altB_n$ of $\bigwedge^n V$ as in \eqref{basisAlternating}. As usual, we can regard $A$ as an even supermatrix. Since $\varphi \in \ker(\widehat{\Psi}_n)_R$ and by Prop.~\ref{altBourbaki}-2), we have that $(\delta_{I,J})_{I,J\in\altI_n(V)} = A^{\wedge n} = \big( \widehat{\cM}_{I,J}(A) \big)_{I,J\in\altI_n(V)}$. 

In particular, if $V$ is even of dimension $n$ we get $(1) = A^{\wedge n} = \big(\det(A)\big)$, which proves case 1). Now consider the case where $V$ is even and $n < d = \dim V$. Let $S$ be a principal submatrix of $A$ of order $n+1$ (by principal, we mean that we take the same indices of rows and columns). Then, all minors (respectively, principal minors) of $S$ are minors (respectively, principal minors) of $A$. It follows that the principal first minors of $S$ are $1$ and the non-principal first minors of $S$ are $0$. Thus, the adjugate matrix of $S$ is $\adj(S) = I_{n+1}$. By the inversion formula, we get $\det(S)I_{n+1} = S \adj(S) = S$. Since this holds for any principal submatrix of $A$ of order $n+1$, it follows that $A = r I_d$ for some $r \in R^\times$, and therefore $\varphi = r\id_V$. Since $A^{\wedge n} = (\delta_{I,J})_{I,J\in\altI_n(V)}$, we must have $r^n = 1$. We have proven the property for the case where $V$ is even.

From now on, assume that the odd dimension is $d_{\bar1}>0$. We claim that $A_{\bar1} = r I_{d_{\bar1}}$ for some $r\in R^\times$ such that $r^n = 1$. To show the claim, we will only consider indices $i$ and $j$ corresponding to the rows and columns of the block $A_{\bar1}$. For $I = (i, \dots, i) \in \altI_n$, we get $1 = \delta_{I,I} = \widehat{\cM}_{I,I}(A) = a_{ii}^n$, and consequently $a_{ii} \in R^\times$. For $i<j$, take $I = (i,\dots,i) \in\altI_n$ and $J = (i,\dots,i,j) \in\altI_n$, and we get $0 = \delta_{I,J} = \widehat{\cM}_{I,J}(A) = a_{ii}^{n-1} a_{ij}$, so that $a_{ij} = 0$. For $i<j$, take $I = (j,\dots,j) \in\altI_n$ and $J = (i,j,\dots,j) \in\altI_n$, and we get $0 = \delta_{I,J} = \widehat{\cM}_{I,J}(A) = a_{jj}^{n-1} a_{ji}$, so that $a_{ji} = 0$. For $i<j$ and $I = (i,\dots,i,j) \in\altI_n$, we get $1 = \delta_{I,I} = \widehat{\cM}_{I,I}(A) = (n-1)a_{ii}^{n-2} a_{ij} a_{ji} + a_{ii}^{n-1} a_{jj} = a_{ii}^{n-1} a_{jj}$, so that $a_{ii} = a_{jj}$. We have proven the claim, and the case 2) follows if $V$ is odd. Assume now that $V$ is not odd. If $i$ is an index corresponding to $A_{\bar0}$ and $j$ is an index corresponding to $A_{\bar1}$, take $I = (i,j,\dots,j) \in \altI_n$, and we get $1 = \delta_{I,I} = \widehat{\cM}_{I,I}(A) = a_{ii} a_{jj}^{n-1} = a_{ii} r^{n-1}$, thus $a_{ii} = r$. If $i\neq j$ are indices corresponding to $A_{\bar0}$ and $k$ is an index corresponding to $A_{\bar1}$, take $I = (i,k,\dots,k) \in \altI_n$ and $J = (j,k,\dots,k) \in \altI_n$, so that $0 = \delta_{I,J} = \widehat{\cM}_{I,J}(A) = a_{ij} a_{kk}^{n-1} = a_{ij} r^{n-1}$, thus $a_{ij} = 0$. We have proven that $A = r I_d$, and the result follows.
\end{proof}

\bigskip

\begin{notation} \label{notationKernelModuleAlt}
We claim that if $(\varphi, h) \in\Aut_R(L_R,M_R)$, then we have $(\varphi, h^{\wedge n}) \in\Aut_R(L_R,\bigwedge^n M_R)$. Indeed, by Notation~\ref{notationKernelAlt}, we know that $h^{\wedge n} \in\GL^{\bar0}_R(\bigwedge^n M_R)$. Besides, for homogeneous $x\in L_R$, $v_k \in M_R$ we have that
\begin{align*}
h^{\wedge n}\big(x \cdot (\wedge_i v_i)\big)
&= h^{\wedge n}\Biggl(\sum_i \Bigl(\prod_{k < i} \eta_{x, v_k}\Bigr)
v_1 \wedge\cdots\wedge (x\cdot v_i) \wedge\cdots\wedge v_n\Biggr) \\
&= \sum_i \Bigl(\prod_{k < i} \eta_{x, v_k}\Bigr)
h(v_1) \wedge\cdots\wedge h(x\cdot v_i) \wedge\cdots\wedge h(v_n) \\
&= \sum_i \Bigl(\prod_{k < i} \eta_{\varphi(x), h(v_k)}\Bigr)
h(v_1) \wedge\cdots\wedge \Bigl( \varphi(x)\cdot h(v_i) \Bigr) \wedge\cdots\wedge h(v_n) \\
&= \varphi(x) \cdot \big(\wedge_i h(v_i)\big)
= \varphi(x) \cdot \big( h^{\wedge n}(\wedge_i v_i) \big),
\end{align*}
and there exists $(\varphi, h^{\wedge n})^{-1} = \big(\varphi^{-1}, (h^{-1})^{\wedge n}\big) \in \Aut_R(L_R) \times \GL^{\bar0}_R(\bigwedge^n M_R)$, which proves the claim. Consequently, there is a morphism of group schemes $\widehat{\Phi}_n \colon \bAut(L,M) \to \bAut(L,\bigwedge^n M)$ determined by
\begin{equation} \label{inclusionSchemesAlternating} \begin{split}
(\widehat{\Phi}_n)_R \colon \bAut(L,M)(R) := \Aut_R(L_R,M_R)
& \longrightarrow \bAut(L,\bigwedge^n M)(R) := \Aut_R(L_R,\bigwedge^n M_R), \\
(\varphi, h) & \longmapsto (\varphi, h^{\wedge n}).
\end{split} \end{equation}
Moreover,
$$\bmu_n(R) \cong \{ (\id_L, r \id_M) \med r\in R, \; r^n = 1 \} \leq \Aut_R(L_R,M_R),$$
thus $\bmu_n \lesssim \bker (\widehat{\Phi}_n)$.
\end{notation}

\smallskip

\begin{proposition} \label{propKernelModuleAlt}
Let $\widehat{\Phi}_n$ be defined as above. Then
$\bker\widehat{\Phi}_n = \bAut(L,M) \cap (\bTGS \times \bker\widehat{\Psi}_n)$. In particular:
\begin{itemize}
\item[$1)$] If $M$ is even and has dimension $n$, then
$\bker\widehat{\Phi}_n = \bAut(L,M) \cap (\bTGS \times \bSL_n)$.
\item[$2)$] Otherwise $\bker\widehat{\Phi}_n = \bTGS \times \bmu_n \simeq \bmu_n$,
and therefore $\bAut(L,M)/\bmu_n \lesssim \bAut(L,\bigwedge^n M)$.
\end{itemize}
\end{proposition}
\begin{proof}
This follows from Proposition~\ref{propKernelAlt} and Notation~\ref{notationKernelModuleAlt}.
\end{proof}

\section{Symmetric superpowers of Lie supermodules} \label{section.supermodules.symmetric} \hfill

Throughout this section, unless otherwise stated, we will assume that $M$ is a nonzero finite-dimensional $L$-supermodule, where $L$ is a Lie superalgebra. As in the previous section, the results where the $L$-action is unimportant will be stated in terms of a finite-dimensional vector superspace $V$ (which is just the case $M = V$, $L = 0$).

\begin{dfs}
Let $n \geq 2$ and consider the vector subsuperspace of $\bigboxtimes^n M$ given by
\begin{equation} \label{submodule_symmetric}
\widecheck{R}_n = \widecheck{R}_n(M) := \lspan\{ \otimes_k v_k - \eta_{v_i, v_{i+1}} \tau_{i,i+1}( \otimes_k v_k )
\med 0 \neq v_k \in M_{\bar0} \cup M_{\bar1}, 1 \leq i < n \},
\end{equation}
where $\tau_{ij}$ is the map defined in eq.\eqref{tauSwap} for $1\leq i < j \leq n$.
We will also denote $\widecheck{R}^*_n = \widecheck{R}^*_n(M) := \widecheck{R}_n(M^*)$.
It is not too hard to see that
\begin{equation} \label{submodule_symmetric2}
\widecheck{R}_n = \{ v - \sigma \cdot v \med v \in \bigboxtimes^n M, \sigma \in S_n \},
\end{equation}
with the action $\sigma \cdot v$ defined as in Remarks~\ref{remarks.referee}-1).

The elements of the vector superspace $\bigvee^n M := (\bigboxtimes^n M) / \widecheck{R}_n$ will be called \emph{symmetric supertensors}. Note that if $M$ is even, then $\bigvee^n M$ is a symmetric power of $M$, and if $M$ is odd, then (as a vector space) $\bigvee^n M$ is an alternating power of $M$. We will use the convention $\bigvee^1 M := M$. The projection of a pure supertensor $v_1 \otimes \dots \otimes v_n$ in $\bigvee^n M$ will be denoted by $v_1 \vee \dots \vee v_n$, or just $\vee_i v_i$, and referred to as a \emph{pure symmetric supertensor}. 
Note that the parity map of $\bigvee^n M$ is given by $\varepsilon(\vee_i v_i) := \sum_i \varepsilon(v_i)$ for homogeneous elements $v_i \in M$. We will say that $\vee_i v_i$ is \emph{parity-ordered} if there exists $k\in\{0,1,\dots,n\}$ such that $v_i\in M_{\bar0}$ for $i \leq k$ and $v_j\in M_{\bar1}$ for $j > k$.

Note that $\widecheck{R}_n$ is an $L$-subsupermodule of $\bigboxtimes^n M$, because for elements $x \in L$, $v \in \bigboxtimes^n M$, $\sigma \in S_n$, we have
$$ x \cdot \big( v - \sigma \cdot v \big)
= x \cdot v - \sigma \cdot (x \cdot v) \in \widecheck{R}_n. $$
Consequently, $\bigvee^n M$ becomes an $L$-supermodule with the action given by
\begin{equation} \label{eq.action.symmetric}
x \cdot (\vee_i v_i) := \sum_{i=1}^n  \Bigl( \prod_{k < i} \eta_{x, v_k} \Bigr)
v_1 \vee \dots \vee (x \cdot v_i) \vee \dots \vee v_n
\end{equation}
for each homogeneous $x \in L$, $v_i \in M$. 
We will refer to $(L, \bigvee^n M)$ as the \emph{$n$-th symmetric superpower} of the Lie supermodule $(L,M)$.
\end{dfs}

\begin{remark} \label{remark.restrictions.symmetric}
Let $d_{\bar0} := \dim M_{\bar0}$ and $d_{\bar1} := \dim M_{\bar1}$ be the even and odd dimensions of the Lie supermodule $M$, and $d := \dim M = d_{\bar0} + d_{\bar1}$. Note that for an odd Lie supermodule $M$, $\bigvee^n M$ (regarded as a vector space) is just the usual alternating power, so that $\bigvee^n M = 0$ for $n > \dim M$. Without further mention, we will only consider the cases with $\bigvee^n M \neq 0$, i.e., we will assume $\boxed{\text{$n \leq d_{\bar1}$ if $d_{\bar0} = 0$}}$.
On the other hand, we will also assume
$\boxed{\text{$\chr\FF = 0$ or $\chr\FF > n$ if $d_{\bar0} > 0$}}$,
which will grant nondegeneracy for certain bilinear form $\widecheck{F}$ defined in \eqref{eq.F.vee}.
\end{remark}

\medskip

\begin{notation} \label{notationPerdet1}
Let $L$ be a Lie superalgebra and $M$ a finite-dimensional $L$-supermodule. Our next goal is to construct a bilinear pairing of $L$-supermodules $\bigvee^n M^* \times \bigvee^n M \to \FF$.

Let $\langle\cdot,\cdot\rangle \colon \bigotimes^n M^* \times \bigotimes^n M \to \FF$ be the bilinear form defined as in \eqref{eq.tensor.superproduct}, where we consider $M_i = M$ for each $i = 1,\ldots,n$. Now define a new bilinear form by
\begin{align} \label{def.F.no.vee}
F \colon \bigboxtimes^n M^* \times \bigboxtimes^n M \longrightarrow \FF, \qquad F(f,v) := \sum_{\sigma \in S_n} \langle f, \sigma \cdot v \rangle,
\end{align}
for all $f \in \bigotimes^n M^*$, $v \in \bigotimes^n M$. 
With the same arguments used in Notation~\ref{notationDetper1}, one checks that
\begin{align} \label{property.F.permutations.bis}
F(f, \sigma \cdot v) = F(f,v) = F(\sigma \cdot f, v),
\end{align}
for all $f \in \bigotimes^n M^*$, $v \in \bigotimes^n M$, $\sigma \in S_n$,
and also that $F$ is a bilinear pairing of $L$-supermodules.

By \eqref{property.F.permutations.bis}, it follows that $F$ satisfies the properties
\begin{align*}
F\big( f, v - \sigma \cdot v \big) = 0 = F\big( f - \sigma \cdot f, v\big),
\end{align*}
so that $F(\bigotimes^n M^*, \widecheck{R}_n) = 0 = F(\widecheck{R}^*_n, \bigotimes^n M)$.
Thus $F$ induces a bilinear map
\begin{equation} \label{eq.F.vee}
\widecheck{F} \colon \bigvee^n M^* \times \bigvee^n M \to \FF,
\end{equation}
which is also a pairing of $L$-supermodules. (We will show in Notation~\ref{notation.basis.symmetric} that $\widecheck{F}$ is nondegenerate.)
\end{notation}

\medskip

\begin{notation} \label{notationPerdet2}
Our aim now is to obtain an explicit expression for the bilinear form $\widecheck{F}$ in \eqref{eq.F.vee}.

As in Notation~\ref{notationDetper2}, we will consider the vector superspaces $\bigotimes_\alpha M$ for $\alpha\in \ZZ_2^n$, and the vector superspaces $\bigotimes^{(k,n-k)}M$. For $\alpha = (\alpha_1,\dots,\alpha_n)\in \ZZ_2^n$, consider again the ordered sets $\iota_{\bar0}(\alpha)$ and $\iota_{\bar1}(\alpha)$. We will denote the image of $\bigotimes^{(k,n-k)}M$ on the quotient $\bigvee^n M$ by $\bigvee^{(k, n-k)} M$. (For $k$ such that $n-k > d_{\bar1}$, we have $\bigvee^{(k,n-k)}M = 0$, which can be proven as for alternating powers in the non-super case.) The relations in $\widecheck{R}_n$ show that $\bigvee^{(k,n-k)}M$ is also the image of $\bigotimes^k M_{\bar0} \otimes \bigotimes^{n-k} M_{\bar1}$ on the quotient.

Let $\widecheck{R}_n^{(k,n-k)} := \widecheck{R}_n \cap \bigotimes^{(k,n-k)}M$, then it is easy to see that $\widecheck{R}_n = \bigoplus_{k=0}^n \widecheck{R}_n^{(k,n-k)}$, and it follows that
\begin{equation} \label{decompositionSymmetric}
\bigvee^n M = \bigoplus_{k=0}^{\min(d_{\bar1},n)} \bigvee^{(n-k, k)} M.
\end{equation}
Let $\widecheck{F}^{(k,n-k)}$ denote the restriction of $\widecheck{F}$ to $\bigvee^{(k, n-k)} M^* \times \bigvee^{(k, n-k)} M$. By \eqref{def.F.no.hat} and \eqref{decompositionSymmetric}, we have that
\begin{equation} \label{decompositionSymmetricF}
\widecheck{F} = \perp_{k=0}^{\min(d_{\bar1},n)} \widecheck{F}^{(n-k,k)}.
\end{equation}

For convenience, we introduce the $\perdet$ operators, defined as a combination of the permanent and the determinant. For $0\leq k\leq n$ and $A = (a_{ij})_{ij} \in \cM_n(\FF)$, set
\begin{equation} \label{perdet}
\perdet_{k,n-k}(A) :=
\per\Bigl( (a_{ij})_{1 \leq i,j \leq k} \Bigr)
\det\Bigl( (a_{ij})_{k < i,j \leq n} \Bigr),
\end{equation}
where we use the convention $\det(\emptyset) = 1 = \per(\emptyset)$ for the ``empty submatrix''.
Also, consider $\omega_k$ defined as in \eqref{def.omega_n}.

Fix $k$ with $0\leq n-k \leq \min(d_{\bar1}, n)$. With the same arguments as in Notation~\ref{notationDetper2}, it is not too hard to see that $\widecheck{F}^{(k,n-k)}$ is given, for parity-ordered elements, by
\begin{equation} \label{eq.F.vee.restriction} \begin{split}
\widecheck{F}^{(k,n-k)} \colon &
\bigvee^{(k, n-k)} M^* \times \bigvee^{(k, n-k)} M \longrightarrow \FF, \\
& (\vee_i f_i, \vee_j v_j) \longmapsto
\omega_{n-k} \perdet_{k,n-k}\Bigl( \big( \langle f_i, v_j \rangle \big)_{ij} \Bigr).
\end{split} \end{equation}

In general, for parity-ordered elements $\vee_i f_i \in\bigvee^{(r, n-r)} M^*$, $\vee_j v_j \in\bigvee^{(s, n-s)} M$, we have
\begin{equation} \label{eq.F.vee.unrestricted}
\widecheck{F}( \vee_i f_i, \vee_j v_j )
= \omega_{n-r} \perdet_{r,n-r}\Bigl( \big( \langle f_i, v_j \rangle \big)_{ij} \Bigr)
= \omega_{n-s} \perdet_{s,n-s}\Bigl( \big( \langle f_i, v_j \rangle \big)_{ij} \Bigr),
\end{equation}
because $\perdet_{r,n-r}\Bigl( \big( \langle f_i, v_j \rangle \big)_{ij} \Bigr) = 0
= \perdet_{s,n-s}\Bigl( \big( \langle f_i, v_j \rangle \big)_{ij} \Bigr)$ if $r \neq s$.
\end{notation}

\medskip

\begin{notation} \label{notation.basis.symmetric}
Finally, we will describe some properties of the bilinear form $\widecheck{F}$ in \eqref{eq.F.vee}.

For $k$ such that $0\leq n-k \leq \min(d_{\bar1},n)$, consider the family of ordered $n$-tuples
\begin{equation}
\symI^{(k,n-k)} = \symI^{(k,n-k)}(M) := \{ I=(i_1,\dots,i_n)
\med 1 \leq i_1 \leq \dots \leq i_k \leq d_{\bar0} < i_{k+1} < \dots < i_n \leq d \},
\end{equation}
and 
\begin{equation}
\symI_n = \symI_n(M) := \bigcup_{0 \leq k \leq \min(d_{\bar1},n)} \symI^{(n-k,k)}(M).
\end{equation}
The elements of $\symI_n$ and $\symI^{(k,n-k)}$ will be used as sets of ordered indices of supermatrices, where $k$ and $n-k$ correspond to the number of indices coming from the even and odd subspaces, respectively.

Fix a basis $\cB = \{ v_i \}_{i=1}^d$ of $M$ such that $\cB_{\bar0} := \{ v_i \}_{i=1}^{d_{\bar0}}$ and $\cB_{\bar1} := \{ v_i \}_{i=d_{\bar0}+1}^d$ are bases of $M_{\bar0}$ and $M_{\bar1}$, respectively. Let $\cB^* = \{ f_i \}_{i=1}^d$ be the dual basis of $\cB$. Thus
$\cB^*_{\bar0} := \{ f_i \}_{i=1}^{d_{\bar0}}$
and
$\cB^*_{\bar1} := \{ f_i \}_{i=d_{\bar0}+1}^d$
are the dual bases of $\cB_{\bar0}$ and $\cB_{\bar1}$.
The relations given by $\widecheck{R}_n$ show that $\bigvee^n M$ is spanned by the set
\begin{equation} \label{basisSymmetric}
\symB_n := \{ \check{e}_I := \vee_{i\in I} v_i \med I \in\symI_n \},
\end{equation}
and similarly $\bigvee^n M^*$ is spanned by
\begin{equation} \label{dualBasisSymmetric}
\symB^*_n := \{ \check{e}^*_I := \vee_{i\in I} f_i \med I \in\symI_n \}.
\end{equation}

Consider $\kappa_I$ defined as in \eqref{kappa.coef}. Then, by \eqref{eq.F.vee.unrestricted} it is clear that
\begin{equation} \label{dual.pairing.symmetric}
\widecheck{F}(\check{e}^*_I, \check{e}_J) = \kappa_I \delta_{I,J} = \kappa_J \delta_{I,J}.
\end{equation}

Since we are assuming that $\chr\FF = 0$ or $\chr\FF > n$ whenever $d_{\bar0} > 0$, it follows from \eqref{dual.pairing.symmetric} that $\widecheck{F}$ defines a dual pairing (i.e., $\widecheck{F}$ is nondegenerate), and also that $\symB_n$ and $\symB^*_n$ are bases of $\bigvee^n M$ and $\bigvee^n M^*$, respectively. Unfortunately, $\symB_n$ and $\symB^*_n$ are not $\widecheck{F}$-dual bases in general. It is clear that
\begin{equation}
\dim \bigvee^n M = |\symB_n| = |\symI_n|
= \sum_{k=0}^{\min(d_{\bar1},n)} {d_{\bar1} \choose k} {d_{\bar0} + (n - k) - 1 \choose n-k}.
\end{equation}
Since $\widecheck{F}$ is a dual pairing of supermodules, it defines an isomorphism of $L$-supermodules $\bigvee^n M^* \cong (\bigvee^n M)^*$, $f \mapsto \widecheck{F}(f,\cdot)$.
\end{notation}

\begin{df}
The Lie supermodules duality map $\widecheck{F}$ defined in \eqref{eq.F.vee} will be referred to as the \emph{$n$-th symmetric superpower} of the corresponding Lie supermodules duality map $M^* \times M \to \FF$. It will be denoted by $\langle\cdot,\cdot\rangle$ in further sections.
\end{df}

\begin{notation}
Let $L$ be a Lie superalgebra. Given a homomorphism of finite-dimensional Lie supermodules, $h \in \Hom_L(M,N)$, it is clear that $h^{\otimes n} \in \Hom_L(\bigotimes^n M, \bigotimes^n N)$, and $h^{\otimes n}\big(\widecheck{R}_n(M)\big) \subseteq \widecheck{R}_n(N)$ because $h$ is even. Thus $h^{\otimes n}$ induces an element $h^{\vee n} \in \Hom_L(\bigvee^n M, \bigvee^n N)$, given by $h^{\vee n}(\vee_i x_i) = \vee_i h(x_i)$ for any elements $x_i\in M$. It is also clear that the composition of two homomorphisms, $h_1$ and $h_2$, satisfies the property
$(h_2 \circ h_1)^{\vee n} = h_2^{\vee n} \circ h_1^{\vee n}$.

Consider now the case $M = N$, i.e., $h \in \End_L(M)$. Let $h^*$ be the dual map of $h$ for the bilinear pairing $\langle\cdot,\cdot\rangle \colon M^* \times M \to \FF$, and $(h^{\vee n})^*$ the $\widecheck{F}$-dual map of $h^{\vee n}$. Then for parity-ordered elements $\vee_i f_i \in \bigvee^{(k,n-k)} M^*$ and $\vee_j v_j \in \bigvee^{(k,n-k)} M$, we get
\begin{align*}
\widecheck{F}\big( \vee_i & f_i, h^{\vee n}(\vee_j v_j) \big)
= \widecheck{F}\big( \vee_i f_i, \vee_j h(v_j) \big)
= \omega_{n-k} \perdet_{k,n-k}\Bigl(\big( \langle f_i,h(v_j) \rangle \big)_{ij}\Bigr) \\
&= \omega_{n-k} \perdet_{k,n-k}\Bigl(\big( \langle h^*(f_i),v_j \rangle \big)_{ij}\Bigr)
= \widecheck{F}\big( \vee_i h^*(f_i), \vee_j v_j \big)
= \widecheck{F}\big( (h^*)^{\vee n}(\vee_i f_i), \vee_j v_j \big),
\end{align*}
thus $(h^*)^{\vee n} = (h^{\vee n})^*$.
\end{notation}

\smallskip

\begin{notation} \label{notationSgnSym}
Let $k$ be such that $0\leq n-k \leq \min(d_{\bar1}, n)$ and identify $S_k \times S_{n-k}$ with the subgroup 
of $S_n$ that fixes the sets $\{1,\dots,k\}$ and $\{k+1,\dots,n\}$. Let $I = (i_1,\dots,i_n) \in \symI^{(k,n-k)}$ 
and consider the parity-ordered element $\check{e}_I := \vee_{i\in I} v_i = \vee_t v_{i_t} \in \symB_n$. For each permutation $\sigma\in S_n$, let $\sgnsym_{k,n-k}(\sigma)$ denote the sign defined by
\begin{equation} \label{sgnSym}
\vee_{i\in I} v_i = \sgnsym_{k,n-k}(\sigma) \vee_t v_{i_{\sigma(t)}},
\end{equation}
and note that $\sgnsym_{k,n-k}(\sigma\rho) = \sgn(\rho)$ for $\sigma\in S_k$, $\rho\in S_{n-k}$.
\end{notation}

\smallskip

\begin{notation}
Let $V$ be a finite-dimensional vector superspace. Let $I = (i_1,\dots,i_n) \in\symI^{(p,n-p)}(V)$, $J = (j_1,\dots,j_n) \in\symI^{(q,n-q)}(V)$ with $p$ and $q$ such that $0\leq n-p,n-q \leq \min(d_{\bar1}, n)$. Take an even supermatrix $A = (a_{ij})_{ij} = \diag(A_{\bar0}, A_{\bar1}) \in \cM_{(d_{\bar0}|d_{\bar1}) \times (r|s)}(\FF) = \cM_{d \times n}(\FF)$, with $A_{\bar0} \in\cM_{d_{\bar0} \times r}(\FF)$ and $A_{\bar1} \in\cM_{d_{\bar1} \times s}(\FF)$. Identify as subgroup $S_p \times S_{n-p} \leq S_n$ (as in Notation~\ref{notationSgnSym}). We define the \emph{(symmetric) $(I,J)$-superminor} of $A$ by
\begin{equation}
\widecheck{\cM}_{I,J}(A) := \sum_{\sigma \in S_n(I)}
\sgnsym_{p,n-p}(\sigma) \prod_{t=1}^n a_{i_{\sigma(t)}, j_t}.
\end{equation}
We may also refer to symmetric superminors as \emph{$\perdet$-superminors}.
By the block structure of $A$ it follows that
\begin{equation}
\widecheck{\cM}_{I,J}(A) = 
\Bigl( \sum_{\sigma \in S_p(I)} 
\prod_{t=1}^p a_{i_{\sigma(t)}, j_t} \Bigr)
\Bigl( \sum_{\rho \in S_{n-p}} \sgn(\rho)
\prod_{t=p+1}^n a_{i_{\rho(t)}, j_t} \Big).
\end{equation}
Note that $\widecheck{\cM}_{I,J}(A) = 0$ if $p \neq q$. If $I = J$, then the superminor will be said to be a \emph{principal} superminor.

Consider the case with $n = d-1$. For $1 \leq i,j \leq n$, let $I_i = (1,\dots,i-1,i+1,\dots,d)$, $I_j = (1,\dots,j-1,j+1,\dots,d)$. Then the term $\widecheck{\cM}_{ij}(A) := \widecheck{\cM}_{I_i,I_j}(A)$ will be called the \emph{(symmetric) $(i,j)$-superminor} of $A$. Also, $(i,j)$-superminors will be referred to as \emph{(symmetric) first superminors}.
\end{notation}

\bigskip

\begin{proposition} \label{symBourbaki}
Let $V$ and $V'$ be finite-dimensional vector superspaces. Let $\cB = \{ v_i \}_{i=1}^d$ and $\cB' = \{ v'_i \}_{i=1}^{d'}$ be parity-ordered bases of $V$ and $V'$, respectively. Consider the associated bases $\symB_n = \{\check{e}_I\}_{I\in\symI_n(V)}$ and $\symB'_n = \{\check{e}'_{I'}\}_{I'\in\symI_n(V')}$ of $\bigvee^n V$ and $\bigvee^n V'$, defined as in \eqref{basisSymmetric} by using $\cB$ and $\cB'$. Then:

\begin{itemize}
\item[$1)$] Take a parity-ordered subset $\{w_j\}_{j=1}^n \subseteq V$, with $\{w_j\}_{j=1}^r \subseteq V_{\bar0}$ and $\{w_j\}_{j=r+1}^n \subseteq V_{\bar1}$ for some $r$ such that $0\leq n-r \leq \min(d_{\bar1}, n)$, let $s = n - r$, and set $w_j = \sum_{i=1}^d a_{ij} v_i$. Consider the even supermatrix $A = (a_{ij})_{ij} = \diag(A_{\bar0}, A_{\bar1}) \in \cM_{(d_{\bar0}|d_{\bar1}) \times (r|s)}(\FF) = \cM_{d \times n}(\FF)$, with $A_{\bar0} \in\cM_{d_{\bar0} \times r}(\FF)$ and $A_{\bar1} \in\cM_{d_{\bar1} \times s}(\FF)$. Let $J = (1,\dots,n)$. Then
\begin{equation}
\vee_j w_j = \sum_{I \in\symI_n(V)} \widecheck{\cM}_{I,J}(A) \check{e}_I,
\end{equation}
where $\widecheck{\cM}_{I,J}(A) = 0$ if $I \notin \symI^{(r,n-r)}(V)$.

\item[$2)$] Let $h \colon V \to V'$ be an even homomorphism of vector superspaces, and let $A = (a_{ij})_{ij} = \diag(A_{\bar0}, A_{\bar1})\in \cM_{(d'_{\bar0}|d'_{\bar1}) \times (d_{\bar0}|d_{\bar1})}(\FF) = \cM_{d' \times d}(\FF)$ be its coordinate matrix on the bases $\cB$ and $\cB'$, which is an even supermatrix with $A_{\bar0} \in\cM_{d'_{\bar0} \times d_{\bar0}}(\FF)$ and $A_{\bar1} \in\cM_{d'_{\bar1} \times d_{\bar1}}(\FF)$. Then, the coordinate matrix of $h^{\vee n}$ in the bases $\symB_n$ and $\symB'_n$ is
\begin{equation} \label{matrixSymPower}
A^{\vee n} := \big( \widecheck{\cM}_{I',I}(A) \big)_{I'\in\symI_n(V'), I \in\symI_n(V)}.
\end{equation}
(Here, $\widecheck{\cM}_{I',I}(A) = 0$ if $r \neq s$, where $I\in\symI^{(r,n-r)}(V)$, $I'\in\symI^{(s,n-s)}(V')$.)
\end{itemize}
\end{proposition}
\begin{proof}
The proof is analogous to the one of Proposition~\ref{altBourbaki}.
\end{proof}

\begin{df} The matrix $A^{\vee n}$, defined as in \eqref{matrixSymPower} from an even supermatrix
$A = \diag(A_{\bar0}, A_{\bar1})$, will be called the \emph{$n$-th symmetric superpower} of $A$. Then $A^{\vee n}$ can be regarded as an even supermatrix whose rows and columns are indexed by $\symI_n(V')$ and $\symI_n(V)$, respectively.
\end{df}

\begin{notation} \label{notationKernelSym}
Let $V$ be a finite-dimensional vector superspace.
Consider the morphism of affine group schemes $\widecheck{\Psi}_n \colon \bGL^{\bar0}(V) \to \bGL^{\bar0}(\bigvee^n V)$ given by
\begin{equation}
(\widecheck{\Psi}_n)_R \colon \bGL^{\bar0}(V)(R) := \GL^{\bar0}_R(V_R)
\longrightarrow \bGL^{\bar0}(\bigvee^n V)(R) := \GL^{\bar0}_R(\bigvee^n V_R),
\quad \varphi \longmapsto \varphi^{\vee n}.
\end{equation}
Given $\varphi \in \GL^{\bar0}_R(V_R)$, it is clear that $\varphi^{\vee n}$ is even and invertible, with inverse $(\varphi^{\vee n})^{-1} = (\varphi^{-1})^{\vee n}$, and therefore $\widecheck{\Psi}_n$ is well-defined. Moreover, 
$$\bmu_n(R) \cong \{ r \id_V \med r\in R, \; r^n = 1 \} \leq \ker(\widecheck{\Psi}_n)_R \leq \GL^{\bar0}_R(V_R).$$
A description of the affine group scheme $\bker\widecheck{\Psi}_n$ is given by the following result.
\end{notation}

\begin{proposition} \label{propKernelSym}
Let $\widecheck{\Psi}_n$ be defined as above.
\begin{itemize}
\item[$1)$] If $V$ is odd and has dimension $n$, then $\bker\widecheck{\Psi}_n \simeq \bSL_n$.
\item[$2)$] Otherwise, $\bker\widecheck{\Psi}_n \simeq \bmu_n$.
\end{itemize}
\end{proposition}
\begin{proof}
The proof is analogous to the one in Proposition \ref{propKernelAlt}
\end{proof}

\begin{notation} \label{notationKernelModuleSym}
We claim that if $(\varphi, h) \in\Aut_R(L_R,M_R)$, then we have $(\varphi, h^{\vee n}) \in\Aut_R(L_R,\bigvee^n M_R)$. Indeed, by Notation~\ref{notationKernelSym}, we know that $h^{\vee n} \in\GL^{\bar0}_R(\bigvee^n M_R)$. Besides, for homogeneous $x\in L_R$, $v_k \in M_R$ 
it is easy to see that
$$
h^{\vee n}\big(x \cdot (\vee_i v_i)\big)
= \varphi(x) \cdot \big( h^{\vee n}(\vee_i v_i) \big),
$$
and there exists $(\varphi, h^{\vee n})^{-1} = \big(\varphi^{-1}, (h^{-1})^{\vee n}\big) \in \Aut_R(L_R) \times \GL^{\bar0}_R(\bigvee^n M_R)$, which proves the claim. Consequently, there is a morphism of group schemes $\widecheck{\Phi}_n \colon \bAut(L,M) \to \bAut(L,\bigvee^n M)$ determined by
\begin{equation} \label{inclusionSchemesSymmetric} \begin{split}
(\widecheck{\Phi}_n)_R \colon \bAut(L,M)(R) := \Aut_R(L_R,M_R)
& \longrightarrow \bAut(L,\bigvee^n M)(R) := \Aut_R(L_R,\bigvee^n M_R), \\
(\varphi, h) & \longmapsto (\varphi, h^{\vee n}).
\end{split} \end{equation}
Moreover,
$$\bmu_n(R) \cong \{ (\id_L, r \id_M) \med r\in R, \; r^n = 1 \} \leq \Aut_R(L_R,M_R),$$
thus $\bmu_n \lesssim \bker (\widecheck{\Phi}_n)$.
\end{notation}

\smallskip

\begin{proposition} \label{propKernelModuleSym}
Let $\widecheck{\Phi}_n$ be defined as above. Then
$\bker\widecheck{\Phi}_n = \bAut(L,M) \cap (\bTGS \times \bker\widecheck{\Psi}_n)$. In particular:
\begin{itemize}
\item[$1)$] If $M$ is odd and has dimension $n$, then
$\bker\widecheck{\Phi}_n = \bAut(L,M) \cap (\bTGS \times \bSL_n)$.
\item[$2)$] Otherwise $\bker\widecheck{\Phi}_n = \bTGS \times \bmu_n \simeq \bmu_n$,
and therefore $\bAut(L,M)/\bmu_n \lesssim \bAut(L,\bigvee^n M)$.
\end{itemize}
\end{proposition}
\begin{proof}
This follows from Proposition~\ref{propKernelSym} and Notation~\ref{notationKernelModuleSym}.
\end{proof}

\section{Alternating superpowers of metric generalized Jordan superpairs} \label{section.superpairs.alternating} \hfill 

\begin{df}
Let $(L,M,b) \in \MFLSM$ and $(\cV,\langle\cdot,\cdot\rangle) \in \MGJSP$ be nonzero corresponding objects through the Faulkner correspondence. The Lie supermodule $(L,\bigwedge^n M, b)$ is not neccesarily faithful, but by the Faulkner construction, it defines an object $(\bigwedge^n \cV, \langle\cdot,\cdot\rangle) \in \MGJSP$ that will be called the \emph{$n$-th alternating} (or \emph{exterior}) \emph{superpower} of $(\cV, \langle\cdot,\cdot\rangle)$ in the class $\MGJSP$. The object in $\MFLSM$ that corresponds to $(\bigwedge^n \cV, \langle\cdot,\cdot\rangle)$ will be called the \emph{$n$-th alternating superpower} of $(L,M,b)$ in the class $\MFLSM$, which is given by  $(\widetilde{L}, \bigwedge^n M, \widetilde{b})$, where $\widetilde{L}$ is a quotient of $\instr(L,\bigwedge^n M)$, and $\widetilde{b}$ is determined by $b$ (this follows from \cite[Prop.3.3]{A22}).
\end{df}

\begin{remark}
Throughout this section, and without further mention unless otherwise stated, we will only consider objects $(\cV, \langle\cdot,\cdot\rangle)\in\MGJSP$ and $n > 1$ such that both vector superspaces $\cV^+$ and $\cV^-$ satisfy the conditions from Remark~\ref{remark.restrictions.alternating}. This will avoid considering the case where $\bigwedge^n\cV = 0$, and the restrictions of $\chr(\FF)$ are necessary for nondegeneracy of the bilinear form (which is used in the Faulkner construction).
\end{remark}

\bigskip

\begin{notation}
Again, we need more auxiliary notation. For $1\leq i,j,p,q \leq n$, define:
\begin{equation}
\varsigmahat_{p,q}(i,j) = \varsigmahat_{p,q,n}(i,j) :=
\begin{cases}
(-1)^{i+j}, \quad & \text{for $i\leq p, j \leq q$}, \\
(-1)^{i+n}, \quad & \text{for $i \leq p$, $j > q$}, \\
(-1)^{j+n}, \quad & \text{for $i > p$, $j \leq q$}, \\
1, \quad & \text{for $i > p, j > q$}.
\end{cases}
\end{equation}
\end{notation}

\smallskip

\begin{proposition} \label{alternatigPowerPairs}
Let $\cV$ be a nonzero object in $\MGJSP$, $1 < n\in\NN$, and $\cW = \bigwedge^n\cV$. Then:
\begin{itemize}
\item[1)] The bilinear form $\langle\cdot,\cdot\rangle$ on $\cW$ is given by the $n$-th alternating superpower of the bilinear form of $\cV$. That is, it is determined, for parity-ordered elements $\wedge_i f_i \in \bigwedge^{(p,n-p)}\cV^-$, $\wedge_j v_j \in \bigwedge^{(q,n-q)}\cV^+$, by
$$ \langle \wedge_i f_i, \wedge_j v_j \rangle = \omega_{n-p} \detper_{p,n-p}\Bigl( \big(\langle f_i,v_j \rangle\big)_{ij}\Bigr), $$
which is zero if $p \neq q$.
\item[2)] Fix $\sigma \in\{+, - \}$. For parity-ordered elements $\wedge_i f_i \in \bigwedge^{(p,n-p)}\cV^{-\sigma}$ and $\wedge_j v_j \in \bigwedge^{(q,n-q)}\cV^\sigma$, the spanning elements of $\instr(\cW)$ are of the form
\begin{equation} \label{inderAlt}
\nu( \wedge_i f_i, \wedge_j v_j )
= \omega_{n-p} \sum_{i,j = 1}^n \varsigmahat_{p,q}(i,j) \widehat{\cM}_{ij}(B) \nu(f_i, v_j),
\end{equation}
where $\widehat{\cM}_{ij}(B)$ is the alternating $(i,j)$-superminor of the even supermatrix
$$B = \diag(B_{\bar0},B_{\bar1}) := (\langle f_i, v_j \rangle)_{ij}
\in \cM_{(p|n-p) \times (q|n-q)}(\FF).$$
\item[3)] Fix parity-ordered elements $\wedge_i f_i \in \bigwedge^{(p,n-p)}\cV^{-\sigma}$, $\wedge_j v_j \in \bigwedge^{(q,n-q)}\cV^\sigma$, $\wedge_k g_k \in \bigwedge^{(r,n-r)}\cV^{-\sigma}$, for some $\sigma = \pm$. Then the triple products of $\cW$ are given by
\begin{equation}
\{ \wedge_i f_i, \wedge_j v_j, \wedge_k g_k \}
= \omega_{n-p} \sum_{i,j,k = 1}^n \varsigmahat_{p,q}(i,j) \widehat{\cM}_{ij}(B)
\Bigl( \prod_{t<k} \eta_{g_t, D(f_i,v_j)} \Bigr)
g_1 \wedge \dots \wedge \{ f_i, v_j, g_k \} \wedge \dots \wedge g_n,
\end{equation}
where $B := (\langle f_i, v_j \rangle)_{ij} \in \cM_{(p|n-p) \times (q|n-q)}(\FF)$.
\item[4)] There is a morphism of affine group schemes
$ \widehat{\Omega}_n \colon \bAut(\cV, \langle\cdot,\cdot\rangle)
\to \bAut(\cW, \langle\cdot,\cdot\rangle) $
given by
\begin{equation} \begin{split}
(\widehat{\Omega}_n)_R \colon \Aut_R(\cV_R, \langle\cdot,\cdot\rangle)
& \longrightarrow \Aut_R(\cW_R, \langle\cdot,\cdot\rangle), \\
\varphi = (\varphi^-, \varphi^+) & \longmapsto \varphi^{\wedge n}
:= \big( (\varphi^-)^{\wedge n}, (\varphi^+)^{\wedge n}) \big).
\end{split}\end{equation}
Furthermore: \\
\noindent $i)$ If $\cV$ is even and $\dim\cV = n$, then
$\bker \widehat{\Omega}_n = \bSL_n \cap \bAut(\cV, \langle\cdot,\cdot\rangle)$
and 
\begin{equation}
\bAut(\cV, \langle\cdot,\cdot\rangle)/\bker \widehat{\Omega}_n
\lesssim \bAut(\cW, \langle\cdot,\cdot\rangle) \simeq \bG_m .
\end{equation}
\noindent $ii)$ Otherwise, $\bker \widehat{\Omega}_n = \bmu_n$ and
\begin{equation}
\bAut(\cV, \langle\cdot,\cdot\rangle)/\bmu_n \lesssim \bAut(\cW, \langle\cdot,\cdot\rangle).
\end{equation}
\end{itemize}
\end{proposition}
\begin{proof}
1) The property follows from the Faulkner construction.

\smallskip

2) We will prove the property by nondegeneracy of $b$.
Fix a homogeneous element $x\in \instr(\cW)$ and parity-ordered elements $\wedge_i f_i \in \bigwedge^{(p,n-p)}\cV^{-\sigma}$, $\wedge_j v_j \in \bigwedge^{(q,n-q)}\cV^\sigma$. There are three nontrivial cases to check.

\noindent $\bullet$ First, consider the case where $x$ is even and $p = q$. Then:
\begin{align*}
b\big(&x, [\wedge_i f_i, \wedge_j v_j]\big)
=_\eqref{bilinearFormsCorrespondence} \langle x \cdot (\wedge_i f_i), \wedge_j v_j \rangle
= \sum_i \langle f_1 \wedge \dots \wedge (x \cdot f_i) \wedge \dots \wedge f_n, \wedge_j v_j \rangle \\
&= \omega_{n-p} \sum_{i\leq p} \det \begin{pmatrix}
\langle f_1, v_1 \rangle & \cdots & \langle f_1, v_p \rangle \\
\vdots & \cdots & \vdots \\
\langle x \cdot f_i, v_1 \rangle & \cdots & \langle x\cdot f_i, v_p \rangle \\
\vdots & \cdots & \vdots \\
\langle f_p, v_1 \rangle & \cdots & \langle f_p, v_p \rangle
\end{pmatrix}
\per\Bigl( (\langle f_r,v_s \rangle)_{r,s > p} \Bigr) \\
& \quad + \omega_{n-p} \sum_{i > p}
\det\Bigl( (\langle f_r,v_s \rangle)_{r,s \leq p} \Bigr)
\per \begin{pmatrix}
\langle f_{p+1}, v_{p+1} \rangle & \cdots & \langle f_{p+1}, v_n \rangle \\
\vdots & \cdots & \vdots \\
\langle x \cdot f_i, v_{p+1} \rangle & \cdots & \langle x\cdot f_i, v_n \rangle \\
\vdots & \cdots & \vdots \\
\langle f_n, v_{p+1} \rangle & \cdots & \langle f_n, v_n \rangle
\end{pmatrix} \\
&= \omega_{n-p} \sum_{i,j \leq p} (-1)^{i+j} \langle x \cdot f_i, v_j \rangle
\det\Biggl( (\langle f_r,v_s \rangle)_{\substack{i \neq r \leq p \\ j \neq s \leq p}} \Biggr)
\per\Bigl( (\langle f_r,v_s \rangle)_{r,s > p} \Bigr) \\
& \quad + \omega_{n-p} \sum_{i,j > p}
\det\Bigl( (\langle f_r,v_s \rangle)_{r,s \leq p} \Bigr)
\langle x \cdot f_i, v_j \rangle
\per\Biggl( (\langle f_r,v_s \rangle)_{\substack{i \neq r > p \\ j \neq s > p}} \Biggr)
=_\eqref{bilinearFormsCorrespondence} \\
&= \omega_{n-p} \sum_{i,j \leq p} (-1)^{i+j} \widehat{\cM}_{ij}(B) b\big(x, [f_i, v_j]\big)
+ \omega_{n-p} \sum_{i,j > p} \widehat{\cM}_{ij}(B) b\big(x, [f_i, v_j]\big) \\
&= b\Big(x, \omega_{n-p} \sum_{i,j} \varsigmahat_{p,q}(i,j)
\widehat{\cM}_{ij}(B) [f_i, v_j] \Big).
\end{align*}

\noindent $\bullet$ Second, consider the case where $x$ is odd and $q = p+1$. Then:
\begin{align*}
b\big(&x, [\wedge_i f_i, \wedge_j v_j]\big)
=_\eqref{bilinearFormsCorrespondence} \langle x \cdot (\wedge_i f_i), \wedge_j v_j \rangle
= \sum_i \Bigl( \prod_{t < i} \eta_{x, f_t} \Bigr)
\langle f_1 \wedge \dots \wedge (x \cdot f_i) \wedge \dots \wedge f_n, \wedge_j v_j \rangle \\
&= \sum_{i > p} (-1)^{i+p+1}
\langle f_1 \wedge\dots\wedge (x \cdot f_i) \wedge\dots\wedge f_n, \wedge_j v_j \rangle \\
&= \sum_{i > p}
\langle f_1 \wedge\dots\wedge f_p \wedge (x \cdot f_i) \wedge  f_{p+1}
\wedge\dots\wedge f_{i-1} \wedge f_{i+1} \wedge\dots\wedge f_n, \wedge_j v_j \rangle \\
&= \omega_{n-p-1} \sum_{i > p} \det 
\begin{pmatrix}  
\langle f_1, v_1 \rangle & \cdots & \langle f_1, v_{p+1} \rangle \\
\vdots & \cdots & \vdots \\
\langle f_p, v_1 \rangle & \cdots & \langle f_p, v_{p+1} \rangle \\
\langle x\cdot f_i, v_1 \rangle & \cdots & \langle x\cdot f_i, v_{p+1} \rangle
\end{pmatrix}
\per\Biggl( (\langle f_r,v_s \rangle)_{\substack{i \neq r > p \\ s > p+1}} \Biggr) 
=_\eqref{omegaProp} \\
&= \omega_{n-p}(-1)^{n-p-1} \sum_{\substack{i > p \\ j \leq p+1}} (-1)^{j+p+1}
\det\Biggl( (\langle f_r,v_s \rangle)_{\substack{r \leq p \\ j \neq s \leq p+1}} \Biggr)
\langle x \cdot f_i,v_j \rangle
\per\Biggl( (\langle f_r,v_s \rangle)_{\substack{i \neq r > p \\ s > p+1}} \Biggr) \\
&= \omega_{n-p} \sum_{\substack{i > p \\ j \leq p+1}} (-1)^{n+j}
\det\Biggl( (\langle f_r,v_s \rangle)_{\substack{r \leq p \\ j \neq s \leq p+1}} \Biggr)
\langle x \cdot f_i,v_j \rangle
\per\Biggl( (\langle f_r,v_s \rangle)_{\substack{i \neq r > p \\ s > p+1}} \Biggr)
=_\eqref{bilinearFormsCorrespondence} \\
&= \omega_{n-p} \sum_{\substack{i > p \\ j \leq p+1}} (-1)^{n+j}
\widehat{\cM}_{ij}(B)  b\big(x, [f_i, v_j]\big)
= b\Big(x, \omega_{n-p} \sum_{i,j} \varsigmahat_{p,q}(i,j)
\widehat{\cM}_{ij}(B) [f_i, v_j] \Big).
\end{align*}

\noindent $\bullet$ Consider the third case, where $x$ is odd and $q = p-1$. Then:
\begin{align*}
b\big(&x, [\wedge_i f_i, \wedge_j v_j]\big)
=_\eqref{bilinearFormsCorrespondence} \langle x \cdot (\wedge_i f_i), \wedge_j v_j \rangle
=  \sum_i \Bigl( \prod_{t < i} \eta_{x, f_t} \Bigr)
\langle f_1 \wedge\dots\wedge (x \cdot f_i) \wedge\dots\wedge f_n, \wedge_j v_j\rangle \\
&= \sum_{i \leq p}
\langle f_1 \wedge\dots\wedge (x \cdot f_i) \wedge\dots\wedge f_n, \wedge_j v_j\rangle \\
&= \sum_{i \leq p} (-1)^{i+p}
\langle f_1 \wedge\dots\wedge f_{i-1} \wedge f_{i+1} \wedge\dots\wedge f_p
\wedge (x \cdot f_i) \wedge f_{p+1} \wedge\dots\wedge f_n, \wedge_j v_j\rangle \\
&= \omega_{n-p+1} \sum_{i \leq p} (-1)^{i+p}
\det\Biggl( (\langle f_r,v_s \rangle)_{\substack{i\neq r \leq p \\ s < p}} \Biggr)
\per \begin{pmatrix}
\langle x\cdot f_i, v_p \rangle & \cdots & \langle x\cdot f_i, v_n \rangle \\
\langle f_{p+1}, v_p \rangle & \cdots & \langle f_{p+1}, v_n \rangle \\
\vdots & \vdots & \vdots \\
\langle f_n, v_p \rangle & \cdots & \langle f_n, v_n \rangle
\end{pmatrix} 
=_\eqref{omegaProp} \\
&= \omega_{n-p} (-1)^{n+p} \sum_{\substack{i \leq p \\ j \geq p}} (-1)^{i+p}
\langle x\cdot f_i, v_j \rangle
\det\Biggl( (\langle f_r,v_s \rangle)_{\substack{i\neq r \leq p \\ s < p}} \Biggr)
\per\Biggl( (\langle f_r,v_s \rangle)_{\substack{r > p \\ j \neq s \geq p}} \Biggr) \\
&= \omega_{n-p} \sum_{\substack{i \leq p \\ j \geq p}} (-1)^{n+i}
\langle x\cdot f_i, v_j \rangle
\det\Biggl( (\langle f_r,v_s \rangle)_{\substack{i\neq r \leq p \\ s < p}} \Biggr)
\per\Biggl( (\langle f_r,v_s \rangle)_{\substack{r > p \\ j \neq s \geq p}} \Biggr)
=_\eqref{bilinearFormsCorrespondence} \\
&= \omega_{n-p} \sum_{\substack{i \leq p \\ j \geq p}} (-1)^{n+i}
\widehat{\cM}_{ij}(B)  b\big(x, [f_i, v_j]\big)
= b\Big(x, \omega_{n-p} \sum_{i,j} \varsigmahat_{p,q}(i,j)
\widehat{\cM}_{ij}(B) [f_i, v_j] \Big).
\end{align*}
Finally, we conclude that the property follows by nondegeneracy of $b$, and then applying the epimorphism $\Upsilon$ in \eqref{epimorphismFaulkner}.

\smallskip

3) The property follows since:
\begin{align*}
\{ \wedge_i f_i, \wedge_j & v_j, \wedge_k g_k \} 
= \nu(\wedge_i f_i, \wedge_j v_j) \cdot (\wedge_k g_k)
=_\eqref{inderAlt} \\
&= \omega_{n-p} \sum_{i,j} \varsigmahat_{p,q}(i,j) \widehat{\cM}_{ij}(B) \nu(f_i, v_j) \cdot (\wedge_k g_k)
=_\eqref{eq.action.alternating} \\
&= \omega_{n-p} \sum_{i,j,k} \varsigmahat_{p,q}(i,j) \widehat{\cM}_{ij}(B)
\Bigl( \prod_{t<k} \eta_{g_t, D(f_i,v_j)} \Bigr)
g_1 \wedge\dots\wedge \big(\nu(f_i, v_j) \cdot g_k\big) \wedge\dots\wedge g_n \\
&= \omega_{n-p} \sum_{i,j,k} \varsigmahat_{p,q}(i,j) \widehat{\cM}_{ij}(B)
\Bigl( \prod_{t<k} \eta_{g_t, D(f_i,v_j)} \Bigr)
g_1 \wedge \dots \wedge \{ f_i, v_j, g_k \} \wedge \dots \wedge g_n.
\end{align*}

\smallskip

4) Fix parity-ordered elements $\wedge_i f_i \in\bigwedge^{(p,n-p)}\cV^-_R$, $\wedge_j v_j \in\bigwedge^{(p,n-p)}\cV^+_R$, and $\varphi\in\Aut_R(\cV_R,\langle\cdot,\cdot\rangle)$. Then
\begin{equation*} \begin{split}
&\langle (\varphi^-)^{\wedge n}(\wedge_i f_i), (\varphi^+)^{\wedge n}(\wedge_j v_j) \rangle
= \langle \wedge_i \varphi^-(f_i), \wedge_j \varphi^+(v_j) \rangle \\
& \qquad = \omega_{n-p} \detper_{p,n-p} \Big( \big(\langle \varphi^-(f_i), \varphi^+(v_j) \rangle\big)_{ij} \Big)
= \omega_{n-p} \detper_{p,n-p} \Big( \big(\langle f_i, v_j \rangle\big)_{ij} \Big) \\
& \qquad = \langle \wedge_i f_i, \wedge_j v_j \rangle,
\end{split} \end{equation*}
thus $\langle\cdot,\cdot\rangle$ is $\bAut(\cV,\langle\cdot,\cdot\rangle)$-invariant, and consequently so it is the matrix $B$ (and its minors) associated to the elements $\wedge_i f_i$ and $\wedge_j v_j$.
Then, for $\wedge_i f_i$, $\wedge_j v_j$, $\wedge_k g_k$ as above, we have
\begin{align*}
& (\varphi^-)^{\wedge n}\big( \{ \wedge_i f_i, \wedge_j v_j, \wedge_k g_k \} \big) = \\
&= (\varphi^-)^{\wedge n}\Big(
\omega_{n-p} \sum_{i,j,k} \varsigmahat_{p,q}(i,j) \widehat{\cM}_{ij}(B)
\Bigl( \prod_{t<k} \eta_{g_t, D(f_i,v_j)} \Bigr)
g_1 \wedge \dots \wedge \{ f_i, v_j, g_k \} \wedge \dots \wedge g_n
\Big) \\
&= \omega_{n-p} \sum_{i,j,k} \varsigmahat_{p,q}(i,j) \widehat{\cM}_{ij}(B)
\Bigl( \prod_{t<k} \eta_{\varphi^-(g_t), D(\varphi^-(f_i),\varphi^+(v_j))} \Bigr) \cdot \\
& \qquad\qquad\qquad \cdot \varphi^-(g_1) \wedge \dots \wedge
\{ \varphi^-(f_i), \varphi^+(v_j), \varphi^-(g_k) \}
\wedge \dots \wedge \varphi^-(g_n) \\
&= \{ \wedge_i \varphi^-(f_i), \wedge_j \varphi^+(v_j), \wedge_k \varphi^-(g_k) \} \\
&= \{ (\varphi^-)^{\wedge n}(\wedge_i f_i), (\varphi^+)^{\wedge n}(\wedge_j v_j),
(\varphi^-)^{\wedge n}(\wedge_k g_k) \},
\end{align*}
which also holds, analogously, for the other triple product. We have proven that
$\varphi^{\wedge n} \in \bAut(\cW, \langle\cdot,\cdot\rangle)$.

$i)$ Since $\cW$ is $1$-dimensional, we have $\bAut(\cW,\langle\cdot,\cdot\rangle) \simeq \bG_m$.
By Proposition~\ref{propKernelAlt}, it is clear that
$\bker \widehat{\Omega}_n = \bSL_n \cap \bAut(\cV, \langle\cdot,\cdot\rangle)$.
It is obvious that
$\bAut(\cV,\langle\cdot,\cdot\rangle) / \bker \widehat{\Omega}_n
\lesssim \bAut(\cW,\langle\cdot,\cdot\rangle)$.

$ii)$ By Proposition~\ref{propKernelAlt} we get $\bker \widehat{\Omega}_n = \bmu_n$, and the result follows.
\end{proof}

\begin{example} \label{example.alt}
Recall from \cite{L75} that the simple Jordan pairs of type I are given by $\cV^{\text{(I)}}_{n,m} := (\cM_{n,m}(\FF), \cM_{n,m}(\FF))$ (here $n,m \in \NN$ are arbitrary with $n < m$ and $\chr\FF \neq 2$), with generic trace
$$ t(x,y) = t^{\text{(I)}}(x,y) :=  \tr(xy^\Tr), $$
and triple products
$$ \{ x,y,z \} := x y^\Tr z + z y^\Tr x. $$
It was shown in \cite[Ex.4.7]{A22} that $(\cV^{\text{(I)}}_{n,m}, t) \in \MGJP$.
Simple Jordan pairs of type II are the Jordan subpairs of $\cV^{\text{(I)}}_{n,n}$ given by $\cV^{\text{(II)}}_n := (A_n(\FF), A_n(\FF))$, where $A_n(\FF)$ is the vector space of $n \times n$ antisymmetric matrices, and their generic trace is given by
$$ t(x,y) = t^{\text{(II)}}(x,y) := \sum_{i < j} x_{ij} y_{ij}. $$

Consider the basis $\{ \widehat{E}_{ij} \med i < j \}$ of $A_n(\FF)$ where $\widehat{E}_{ij} := E_{ij} - E_{ji}$, and note that $\widehat{E}_{ij} = -\widehat{E}_{ji}$. For $i<j$ and $k<l$, it is easy to see that
$$t^{\text{(I)}}(\widehat{E}_{ij},\widehat{E}_{kl}) = 2(\delta_{ik}\delta_{jl} - \delta_{il}\delta_{jk})
= 2 \delta_{ik}\delta_{jl} = 2 t^{\text{(II)}}(\widehat{E}_{ij},\widehat{E}_{kl}),$$
so that $t^{\text{(II)}} = \frac{1}{2} t^{\text{(I)}}$ on $\cV^{\text{(II)}}_n$. Consequently, $t^{\text{(II)}}$ (which is nondegenerate) inherits the good properties from $t^{\text{(I)}}$, so that $(\cV^{\text{(II)}}_n, t) \in \MGJP$.
Then we have that
\begin{align*}
(E_{i_1 i_2} & - E_{i_2 i_1}) (E_{j_1 j_2} - E_{j_2 j_1}) (E_{k_1 k_2} - E_{k_2 k_1}) = \\
&= ( \delta_{i_2 j_1} E_{i_1 j_2} + \delta_{i_1 j_2} E_{i_2 j_1}
- \delta_{i_2 j_2} E_{i_1 j_1} - \delta_{i_1 j_1} E_{i_2 j_2} ) (E_{k_1 k_2} - E_{k_2 k_1}) \\
&= (\delta_{i_2 j_2} \delta_{j_1 k_2} - \delta_{i_2 j_1} \delta_{j_2 k_2}) E_{i_1 k_1}
+ (\delta_{i_2 j_1} \delta_{j_2 k_1} - \delta_{i_2 j_2} \delta_{j_1 k_1}) E_{i_1 k_2} \\
& \quad + (\delta_{i_1 j_1} \delta_{j_2 k_2} - \delta_{i_1 j_2} \delta_{j_1 k_2}) E_{i_2 k_1}
+ (\delta_{i_1 j_2} \delta_{j_1 k_1} - \delta_{i_1 j_1} \delta_{j_2 k_1}) E_{i_2 k_2},
\end{align*}
and swapping the labels $i \leftrightarrow k$ we get
\begin{align*}
(E_{k_1 k_2} & - E_{k_2 k_1}) (E_{j_1 j_2} - E_{j_2 j_1}) (E_{i_1 i_2} - E_{i_2 i_1}) = \\
&= (\delta_{k_2 j_2} \delta_{j_1 i_2} - \delta_{k_2 j_1} \delta_{j_2 i_2}) E_{k_1 i_1}
+ (\delta_{k_2 j_1} \delta_{j_2 i_1} - \delta_{k_2 j_2} \delta_{j_1 i_1}) E_{k_1 i_2} \\
& \quad + (\delta_{k_1 j_1} \delta_{j_2 i_2} - \delta_{k_1 j_2} \delta_{j_1 i_2}) E_{k_2 i_1}
+ (\delta_{k_1 j_2} \delta_{j_1 i_1} - \delta_{k_1 j_1} \delta_{j_2 i_1}) E_{k_2 i_2}.
\end{align*}
Therefore, the triple products of $\cV^{\text{(II)}}_n$ are given by
\begin{align*}
\{ \widehat{E}_{i_1 i_2}, & \widehat{E}_{j_1 j_2}, \widehat{E}_{k_1 k_2} \} = \\
&= (E_{i_1 i_2} - E_{i_2 i_1}) (E_{j_1 j_2} - E_{j_2 j_1}) (E_{k_1 k_2} - E_{k_2 k_1}) \\
& \quad + (E_{k_1 k_2} - E_{k_2 k_1}) (E_{j_1 j_2} - E_{j_2 j_1}) (E_{i_1 i_2} - E_{i_2 i_1}) \\
&= (\delta_{i_2 j_2} \delta_{j_1 k_2} - \delta_{i_2 j_1} \delta_{j_2 k_2}) \widehat{E}_{i_1 k_1} 
+ (\delta_{i_2 j_1} \delta_{j_2 k_1} - \delta_{i_2 j_2} \delta_{j_1 k_1}) \widehat{E}_{i_1 k_2} \\
& \quad + (\delta_{i_1 j_1} \delta_{j_2 k_2} - \delta_{i_1 j_2} \delta_{j_1 k_2}) \widehat{E}_{i_2 k_1}
+ (\delta_{i_1 j_2} \delta_{j_1 k_1} - \delta_{i_1 j_1} \delta_{j_2 k_1}) \widehat{E}_{i_2 k_2},
\end{align*}
and the generic trace by
$$ t(\widehat{E}_{i_1 i_2}, \widehat{E}_{j_1 j_2}) = \delta_{i_1 j_1} \delta_{i_2 j_2}. $$

Now, consider two copies of the canonical basis $\{ e_i \}_{i=1}^n$ of $\cM_{1,n}(\FF)$, regarded as bases of the subspaces of $\cV^{\text{(I)}}_{1,n}$, and note that
$$ t(e_i, e_j) = \delta_{ij}, $$
and
$$ \{ e_i, e_j, e_k \} = \delta_{ij} e_k + \delta_{kj} e_i. $$
Then $\{ e_i \wedge e_j \med 1 \leq i < j \leq n \}$ is a basis for both vector spaces of the pair $\cV = \bigwedge^2 \cV^{\text{(I)}}_{1,n}$.
Assuming $1 \leq i_1 < i_2 \leq n$ and $1 \leq j_1 < j_2 \leq n$, the bilinear form of $\cV$ is given by
\begin{align*}
\langle e_{i_1} \wedge e_{i_2}, e_{j_1} \wedge e_{j_2} \rangle
&= \det\big( (t(e_{i_k}, e_{j_l}))_{kl} \big) = \det\big( (\delta_{i_k j_l})_{kl} \big) \\
&= \delta_{i_1 j_1} \delta_{i_2 j_2} - \delta_{i_2 j_1} \delta_{i_1 j_2} = \delta_{i_1 j_1} \delta_{i_2 j_2}.
\end{align*}
Let $M_{ij}$ denote the determinant $(i,j)$-minor of $B = \big( t(e_{i_k}, e_{j_l}) \big)_{kl} = \big( \delta_{i_k j_l} \big)_{kl}$.
Then $M_{11} = \delta_{i_2 j_2}$, $M_{12} = \delta_{i_2 j_1}$, $M_{21} = \delta_{i_1 j_2}$, $M_{22} = \delta_{i_1 j_1}$, and the triple products of $\cV$ are given by
\begin{align*}
\{ e_{i_1} \wedge & e_{i_2}, e_{j_1} \wedge e_{j_2}, e_{k_1} \wedge e_{k_2} \} = \\
&= M_{11} ( \{ e_{i_1}, e_{j_1}, e_{k_1} \} \wedge e_{k_2} + e_{k_1} \wedge \{ e_{i_1}, e_{j_1}, e_{k_2} \}) \\
& \quad -M_{12} ( \{ e_{i_1}, e_{j_2}, e_{k_1} \} \wedge e_{k_2} + e_{k_1} \wedge \{ e_{i_1}, e_{j_2}, e_{k_2} \}) \\
& \quad -M_{21} ( \{ e_{i_2}, e_{j_1}, e_{k_1} \} \wedge e_{k_2} + e_{k_1} \wedge \{ e_{i_2}, e_{j_1}, e_{k_2} \}) \\
& \quad +M_{22} ( \{ e_{i_2}, e_{j_2}, e_{k_1} \} \wedge e_{k_2} + e_{k_1} \wedge \{ e_{i_2}, e_{j_2}, e_{k_2} \}) \\
&= \delta_{i_2 j_2} \Big( \delta_{i_1 j_1} e_{k_1} \wedge e_{k_2} + \delta_{k_1 j_1} e_{i_1} \wedge e_{k_2}
                        + \delta_{i_1 j_1} e_{k_1} \wedge e_{k_2} + \delta_{k_2 j_1} e_{k_1} \wedge e_{i_1} \Big) \\
&\quad - 
   \delta_{i_2 j_1} \Big( \delta_{i_1 j_2} e_{k_1} \wedge e_{k_2} + \delta_{k_1 j_2} e_{i_1} \wedge e_{k_2}
                        + \delta_{i_1 j_2} e_{k_1} \wedge e_{k_2} + \delta_{k_2 j_2} e_{k_1} \wedge e_{i_1} \Big) \\
&\quad - 
   \delta_{i_1 j_2} \Big( \delta_{i_2 j_1} e_{k_1} \wedge e_{k_2} + \delta_{k_1 j_1} e_{i_2} \wedge e_{k_2}
                        + \delta_{i_2 j_1} e_{k_1} \wedge e_{k_2} + \delta_{k_2 j_1} e_{k_1} \wedge e_{i_2} \Big) \\
&\quad + 
   \delta_{i_1 j_1} \Big( \delta_{i_2 j_2} e_{k_1} \wedge e_{k_2} + \delta_{k_1 j_2} e_{i_2} \wedge e_{k_2}
                        + \delta_{i_2 j_2} e_{k_1} \wedge e_{k_2} + \delta_{k_2 j_2} e_{k_1} \wedge e_{i_2} \Big) \\
&=        \Big( \delta_{i_2 j_1} \delta_{k_2 j_2} - \delta_{i_2 j_2} \delta_{k_2 j_1} \Big) e_{i_1} \wedge e_{k_1}
        + \Big( \delta_{i_2 j_2} \delta_{k_1 j_1} - \delta_{i_2 j_1} \delta_{k_1 j_2} \Big) e_{i_1} \wedge e_{k_2} \\
& \quad + \Big( \delta_{i_1 j_2} \delta_{k_2 j_1} - \delta_{i_1 j_1} \delta_{k_2 j_2} \Big) e_{i_2} \wedge e_{k_1}
        + \Big( \delta_{i_1 j_1} \delta_{k_1 j_2} - \delta_{i_1 j_2} \delta_{k_1 j_1} \Big) e_{i_2} \wedge e_{k_2} \\
& \quad	+4\Big( \delta_{i_1 j_1} \delta_{i_2 j_2} - \delta_{i_2 j_1} \delta_{i_1 j_2} \Big) e_{k_1} \wedge e_{k_2}
\end{align*}
Finally, consider the tensor-shift $\cV^{[-4]}$, which has the same bilinear form as $\cV$, and triple products
\begin{align*}
\{ e_{i_1} \wedge & e_{i_2}, e_{j_1} \wedge e_{j_2}, e_{k_1} \wedge e_{k_2} \} = \\
&=        \Big( \delta_{i_2 j_1} \delta_{k_2 j_2} - \delta_{i_2 j_2} \delta_{k_2 j_1} \Big) e_{i_1} \wedge e_{k_1}
        + \Big( \delta_{i_2 j_2} \delta_{k_1 j_1} - \delta_{i_2 j_1} \delta_{k_1 j_2} \Big) e_{i_1} \wedge e_{k_2} \\
& \quad + \Big( \delta_{i_1 j_2} \delta_{k_2 j_1} - \delta_{i_1 j_1} \delta_{k_2 j_2} \Big) e_{i_2} \wedge e_{k_1}
        + \Big( \delta_{i_1 j_1} \delta_{k_1 j_2} - \delta_{i_1 j_2} \delta_{k_1 j_1} \Big) e_{i_2} \wedge e_{k_2}.
\end{align*}
Assume now that there is some element $\bi \in \FF$ such that $\bi^2 = -1$ (we can extend the scalars if necessary). By comparison of the triple products, it follows that the pair of maps $f = (f^-, f^+)$ defined by
\begin{equation*}
f^\sigma \colon A_n(\FF) \longrightarrow \bigwedge^2 \cM_{1,n}(\FF),
\quad \widehat{E}_{ij} \longmapsto \bi e_i \wedge e_j,
\end{equation*}
gives the following isomorphism of (generalized) Jordan pairs:
\begin{equation}
\cV^{\text{(II)}}_n \cong \Big( \bigwedge^2 \cV^{\text{(I)}}_{1,n} \Big)^{[-4]}
= \Big( \bigwedge^2 \cV^{\text{(I)}}_{1,n} \Big) \otimes \cV_{-4}.
\end{equation}
Unfortunately, $f$ is not an isometry of the bilinear forms. However, $f$ is a similarity with multiplier $-1$, that is, $\langle f(x), f(y) \rangle = - t( x, y)$. In other words, $\cV^{\text{(II)}}_n$ and $\bigwedge^2 \cV^{\text{(I)}}_{1,n}$ are isomorphic up to a tensor-shift and a similarity (simultaneously), namely $(\cV^{\text{(II)}}_n, -t) \cong ( \bigwedge^2 \cV^{\text{(I)}}_{1,n}, \langle\cdot,\cdot\rangle)^{[-4]}$.

Let $\widetilde{f} := c_{-\bi} \circ f$ where $c_\lambda^\sigma(x) := \lambda^{\sigma1} x$ for $\lambda\in\FF^\times$. Then
\begin{equation}
\widetilde{f}^\sigma(\widehat{E}_{ij}) = \sigma e_i \wedge e_j,
\end{equation}
and $\widetilde{f}$ defines another isomorphism $(\cV^{\text{(II)}}_n, -t) \cong ( \bigwedge^2 \cV^{\text{(I)}}_{1,n}, \langle\cdot,\cdot\rangle)^{[-4]}$ which does not require that $\bi \in \FF$.
\end{example}

\section{Symmetric superpowers of metric generalized Jordan superpairs} \label{section.superpairs.symmetric} \hfill 

\begin{df}
Let $(L,M,b) \in \MFLSM$ and $(\cV,\langle\cdot,\cdot\rangle) \in \MGJSP$ be nonzero corresponding objects through the Faulkner correspondence. The Lie supermodule $(L,\bigvee^n M, b)$ is not neccesarily faithful, but by the Faulkner construction, it defines an object $(\bigvee^n \cV, \langle\cdot,\cdot\rangle) \in \MGJSP$ that will be called the \emph{$n$-th symmetric superpower} of $(\cV, \langle\cdot,\cdot\rangle)$ in the class $\MGJSP$. The object in $\MFLSM$ that corresponds to $(\bigvee^n \cV, \langle\cdot,\cdot\rangle)$ will be called the \emph{$n$-th symmetric superpower} of $(L,M,b)$ in the class $\MFLSM$, which is given by $(\widetilde{L}, \bigvee^n M, \widetilde{b})$, where $\widetilde{L}$ is a quotient of $\instr(L,\bigvee^n M)$, and $\widetilde{b}$ is determined by $b$ (this follows from \cite[Prop.3.3]{A22}).
\end{df}

\begin{remark}
Throughout this section, and without further mention unless otherwise stated, we will only consider objects $(\cV, \langle\cdot,\cdot\rangle)\in\MGJSP$ and $n > 1$ such that both vector superspaces $\cV^+$ and $\cV^-$ satisfy the conditions from Remark~\ref{remark.restrictions.symmetric}. This will avoid considering the case where $\bigvee^n\cV = 0$,
and the restrictions of $\chr(\FF)$ are necessary for nondegeneracy of the bilinear form (which is used in the Faulkner construction).
\end{remark}

\bigskip

\begin{notation}
Again, we need more auxiliary notation. For $1\leq i,j,p,q \leq n$, define:
\begin{equation}
\varsigmacheck_{p,q}(i,j) = \varsigmacheck_{p,q,n}(i,j) :=
\begin{cases}
1, \quad & \text{for $i \leq p$, $j \leq q$}, \\
(-1)^{j+n}, \quad & \text{for $i \leq p$, $j > q$}, \\
(-1)^{i+n}, \quad & \text{for $i > p$, $j \leq q$}, \\
(-1)^{i+j}, \quad & \text{for $i > p$, $j > q$}.
\end{cases}
\end{equation}
\end{notation}

\smallskip

\begin{proposition} \label{symmetricPowerPairs}
Let $\cV$ be a nonzero object in $\MGJSP$, $1 < n\in\NN$, and $\cW = \bigvee^n\cV$. Then:
\begin{itemize}
\item[1)] The bilinear form $\langle\cdot,\cdot\rangle$ on $\cW$ is given by the $n$-th symmetric superpower of the bilinear form of $\cV$. That is, it is determined, for parity-ordered elements $\vee_i f_i \in \bigvee^{(p,n-p)}\cV^-$, $\vee_j v_j \in \bigvee^{(q,n-q)}\cV^+$, by
$$ \langle \vee_i f_i, \vee_j v_j \rangle = \omega_{n-p} \perdet_{p,n-p}\Bigl( \big(\langle f_i,v_j \rangle\big)_{ij}\Bigr), $$
which is zero if $p \neq q$.
\item[2)] Fix $\sigma \in\{+, - \}$. For parity-ordered elements $\vee_i f_i \in \bigvee^{(p,n-p)}\cV^{-\sigma}$ and $\vee_j v_j \in \bigvee^{(q,n-q)}\cV^\sigma$, the spanning elements of $\instr(\cW)$ are of the form
\begin{equation} \label{inderSym}
\nu( \vee_i f_i, \vee_j v_j )
= \omega_{n-p} \sum_{i,j = 1}^n \varsigmacheck_{p,q}(i,j) \widecheck{\cM}_{ij}(B) \nu(f_i, v_j),
\end{equation}
where $\widecheck{\cM}_{ij}(B)$ is the symmetric $(i,j)$-superminor of the even supermatrix
$$B = \diag(B_{\bar0},B_{\bar1}) := (\langle f_i, v_j \rangle)_{ij}
\in \cM_{(p|n-p) \times (q|n-q)}(\FF).$$
\item[3)] Fix parity-ordered elements $\vee_i f_i \in \bigvee^{(p,n-p)}\cV^{-\sigma}$, $\vee_j v_j \in \bigvee^{(q,n-q)}\cV^\sigma$, $\vee_k g_k \in \bigvee^{(r,n-r)}\cV^{-\sigma}$, for some $\sigma = \pm$. Then the triple products of $\cW$ are given by
\begin{equation}\label{triple.product.symmetric}
\{ \vee_i f_i, \vee_j v_j, \vee_k g_k \}
= \omega_{n-p} \sum_{i,j,k = 1}^n \varsigmacheck_{p,q}(i,j) \widecheck{\cM}_{ij}(B)
\Bigl( \prod_{t<k} \eta_{g_t, D(f_i,v_j)} \Bigr)
g_1 \vee \dots \vee \{ f_i, v_j, g_k \} \vee \dots \vee g_n,
\end{equation}
where $B := (\langle f_i, v_j \rangle)_{ij} \in \cM_{(p|n-p) \times (q|n-q)}(\FF)$.
\item[4)] There is a morphism of affine group schemes
$ \widecheck{\Omega}_n \colon \bAut(\cV, \langle\cdot,\cdot\rangle)
\to \bAut(\cW, \langle\cdot,\cdot\rangle) $
given by
\begin{equation} \begin{split}
(\widecheck{\Omega}_n)_R \colon \Aut_R(\cV_R, \langle\cdot,\cdot\rangle)
& \longrightarrow \Aut_R(\cW_R, \langle\cdot,\cdot\rangle), \\
\varphi = (\varphi^-, \varphi^+) & \longmapsto \varphi^{\vee n}
:= \big( (\varphi^-)^{\vee n}, (\varphi^+)^{\vee n}) \big).
\end{split}\end{equation}
Furthermore: \\
\noindent $i)$ If $\cV$ is odd and $\dim\cV = n$, then
$\bker \widecheck{\Omega}_n = \bSL_n \cap \bAut(\cV, \langle\cdot,\cdot\rangle)$
and 
\begin{equation}
\bAut(\cV, \langle\cdot,\cdot\rangle)/\bker \widecheck{\Omega}_n
\lesssim \bAut(\cW, \langle\cdot,\cdot\rangle) \simeq \bG_m.
\end{equation}
\noindent $ii)$ Otherwise, $\bker \widecheck{\Omega}_n = \bmu_n$ and
\begin{equation}
\bAut(\cV, \langle\cdot,\cdot\rangle)/\bmu_n \lesssim \bAut(\cW, \langle\cdot,\cdot\rangle).
\end{equation}
\end{itemize}
\end{proposition}
\begin{proof}
1) The property follows from the Faulkner construction.

\smallskip

2) We will prove the property by nondegeneracy of $b$.
Fix a homogeneous element $x\in \instr(\cW)$ and parity-ordered elements $\vee_i f_i \in \bigvee^{(p,n-p)}\cV^{-\sigma}$, $\vee_j v_j \in \bigvee^{(q,n-q)}\cV^\sigma$. There are three nontrivial cases to check.

\noindent $\bullet$ First, consider the case where $x$ is even and $p = q$. Then:
\begin{align*}
b\big(&x, [\vee_i f_i, \vee_j v_j]\big)
=_\eqref{bilinearFormsCorrespondence} \langle x \cdot (\vee_i f_i), \vee_j v_j \rangle 
= \sum_{i} \langle f_1 \vee\dots\vee (x \cdot f_i) \vee\dots\vee f_n, \vee_{j} v_j  \rangle \\
&= \omega_{n-p} \sum_{i\leq p} \per \begin{pmatrix}
\langle f_1, v_1 \rangle & \cdots & \langle f_1, v_p \rangle \\
\vdots & \cdots & \vdots \\
\langle x \cdot f_i, v_1 \rangle & \cdots & \langle x\cdot f_i, v_p \rangle \\
\vdots & \cdots & \vdots \\
\langle f_p, v_1 \rangle & \cdots & \langle f_p, v_p \rangle
\end{pmatrix}
\det\Bigl( (\langle f_r,v_s \rangle)_{r,s > p} \Bigr) \\
& \quad + \omega_{n-p} \sum_{i > p}
\per\Bigl( (\langle f_r,v_s \rangle)_{r,s \leq p} \Bigr)
\det \begin{pmatrix}
\langle f_{p+1}, v_{p+1} \rangle & \cdots & \langle f_{p+1}, v_n \rangle \\
\vdots & \cdots & \vdots \\
\langle x \cdot f_i, v_{p+1} \rangle & \cdots & \langle x\cdot f_i, v_n \rangle \\
\vdots & \cdots & \vdots \\
\langle f_n, v_{p+1} \rangle & \cdots & \langle f_n, v_n \rangle
\end{pmatrix} \\
&= \omega_{n-p} \sum_{\substack{i,j \leq p}} \langle x\cdot f_i,  v_j \rangle
\per\Biggl( (\langle f_r,v_s \rangle)_{\substack{ i \neq r \leq p \\ j \neq s \leq p}} \Biggr)
\det\Bigl( (\langle f_r,v_s \rangle)_{r,s > p} \Bigr) \\
& \quad + \omega_{n-p} \sum_{\substack{i,j > p}} (-1)^{i+j} \langle x \cdot f_{i}, v_j \rangle \per\Bigl( (\langle f_r,v_s \rangle)_{r,s \leq p} \Bigr) \det \Biggl( (\langle f_r,v_s \rangle)_{\substack{i\neq r > p \\ j \neq s > p}} \Biggr) =_\eqref{bilinearFormsCorrespondence} \\
&= \omega_{n-p} \sum_{\substack{i,j \leq p}} \widecheck{\cM}_{ij}(B) b\big(x, [f_i, v_j]\big)
+ \omega_{n-p} \sum_{\substack{i,j > p}} (-1)^{i+j} \widecheck{\cM}_{ij}(B) b\big(x, [f_i, v_j]\big) \\
&= b\Big(x, \omega_{n-p} \sum_{i,j} \varsigmacheck_{p,q}(i,j)
\widecheck{\cM}_{ij}(B) [f_i, v_j] \Big).
\end{align*}

\noindent $\bullet$ Second, consider the case where $x$ is odd and $q = p+1$. Then:
\begin{align*}
b\big(&x, [\vee_i f_i, \vee_j v_j]\big)
=_\eqref{bilinearFormsCorrespondence}
\langle x \cdot (\vee_i f_i), \vee_j v_j \rangle
= \sum_i \Bigl( \prod_{t < i} \eta_{x, f_t} \Bigr)
\langle f_1 \vee\dots\vee (x \cdot f_i) \vee\dots\vee f_n, \vee_j v_j\rangle \\
&= \sum_{i > p} (-1)^{i+p+1}
\langle f_1 \vee\dots\vee (x \cdot f_i) \vee\dots\vee f_n, \vee_j v_j\rangle \\
&= \sum_{i > p}  (-1)^{i+p+1} 
\langle f_1 \vee\dots\vee f_p \vee (x\cdot f_i) \vee f_{p+1} \vee\dots\vee f_{i-1} \vee  f_{i+1} \vee\dots\vee f_n, \vee_j v_j\rangle \\
&= \omega_{n-p-1} \sum_{i > p} (-1)^{i+p+1}
\per \begin{pmatrix}
\langle f_1, v_1 \rangle & \cdots & \langle f_1, v_{p+1} \rangle \\
\vdots & \vdots & \vdots \\
\langle f_p, v_1 \rangle & \cdots & \langle f_p, v_{p+1} \rangle\\
\langle x\cdot f_i, v_1 \rangle & \cdots & \langle x\cdot f_i, v_{p+1} \rangle 
\end{pmatrix} 
\det \Biggl( (\langle f_r,v_s \rangle)_{\substack{i \neq r > p \\ s > p+1}} \Biggr)
=_\eqref{omegaProp} \\
&= \omega_{n-p} (-1)^{n-p-1} \sum_{\substack{i > p \\ j \leq p+1}} (-1)^{i+p+1}
\langle x\cdot f_i, v_j \rangle
\per\Biggl( (\langle f_r,v_s \rangle)_{\substack{r \leq p \\ j\neq s \leq p+1}} \Biggr)
\det\Biggl( (\langle f_r,v_s \rangle)_{\substack{i \neq r > p \\ s > p+1}} \Biggr) \\
&=\omega_{n-p} \sum_{\substack{i > p \\ j \leq p+1}} (-1)^{n+i}
\langle x\cdot f_i, v_j \rangle
\per\Biggl( (\langle f_r,v_s \rangle)_{\substack{r \leq p \\ j\neq s \leq p+1}} \Biggr)
\det\Biggl( (\langle f_r,v_s \rangle)_{\substack{i \neq r > p \\ s > p+1}} \Biggr)
=_\eqref{bilinearFormsCorrespondence}\\
&= \omega_{n-p} \sum_{\substack{i > p \\ j \leq p+1}} (-1)^{n+i}
\widecheck{\cM}_{ij}(B)  b\big(x, [f_i, v_j]\big)
= b\Big(x, \omega_{n-p} \sum_{i,j} \varsigmacheck_{p,q}(i,j)
\widecheck{\cM}_{ij}(B) [f_i, v_j] \Big).
\end{align*}

\noindent $\bullet$ Consider the third case, where $x$ is odd and $q = p-1$. Then:
\begin{align*}
b\big(&x, [\vee_i f_i, \vee_j v_j]\big)
=_\eqref{bilinearFormsCorrespondence}
\langle x \cdot (\vee_i f_i), \vee_j v_j \rangle
= \sum_i \Bigl( \prod_{t < i} \eta_{x, f_t} \Bigr)
\langle f_1 \vee\dots\vee (x \cdot f_i) \vee\dots\vee f_n, \vee_j v_j\rangle \\
&= \sum_{i \leq p}
\langle f_1 \vee\dots\vee (x \cdot f_i) \vee\dots\vee f_n, \vee_j v_j\rangle \\
&= \sum_{i \leq p} 
\langle f_1 \vee\dots\vee f_{i-1} \vee f_{i+1} \vee\dots\vee f_p
\vee (x \cdot f_i) \vee f_{p+1} \vee\dots\vee f_n, \vee_j v_j\rangle \\
&= \omega_{n-p+1} \sum_{i \leq p}
\per\Biggl( (\langle f_r,v_s \rangle)_{\substack{i\neq r \leq p \\ s < p}} \Biggr)
\det \begin{pmatrix}
\langle x\cdot f_i, v_p \rangle & \cdots & \langle x\cdot f_i, v_n \rangle \\
\langle f_{p+1}, v_p \rangle & \cdots & \langle f_{p+1}, v_n \rangle \\
\vdots & \vdots & \vdots \\
\langle f_n, v_p \rangle & \cdots & \langle f_n, v_n \rangle
\end{pmatrix} 
=_\eqref{omegaProp} \\
&= \omega_{n-p} (-1)^{n-p} \sum_{\substack{i \leq p \\ j \geq p}} (-1)^{j-p}
\langle x\cdot f_i, v_j \rangle
\per\Biggl( (\langle f_r,v_s \rangle)_{\substack{i\neq r \leq p \\ s < p}} \Biggr)
\det\Biggl( (\langle f_r,v_s \rangle)_{\substack{r > p \\ j \neq s \geq p}} \Biggr) \\
&= \omega_{n-p} \sum_{\substack{i \leq p \\ j \geq p}} (-1)^{n+j}
\langle x\cdot f_i, v_j \rangle
\per\Biggl( (\langle f_r,v_s \rangle)_{\substack{i\neq r \leq p \\ s < p}} \Biggr)
\det\Biggl( (\langle f_r,v_s \rangle)_{\substack{r > p \\ j \neq s \geq p}} \Biggr)
=_\eqref{bilinearFormsCorrespondence} \\
&= \omega_{n-p} \sum_{\substack{i \leq p \\ j \geq p}} (-1)^{n+j}
\widecheck{\cM}_{ij}(B)  b\big(x, [f_i, v_j]\big)
= b\Big(x, \omega_{n-p} \sum_{i,j} \varsigmacheck_{p,q}(i,j)
\widecheck{\cM}_{ij}(B) [f_i, v_j] \Big).
\end{align*}
Finally, we conclude that the property follows by nondegeneracy of $b$, and then applying the epimorphism $\Upsilon$ in \eqref{epimorphismFaulkner}.

\smallskip

3) The property follows since:
\begin{align*}
\{ \vee_i f_i, \vee_j & v_j, \vee_k g_k \} 
= \nu(\vee_i f_i, \vee_j v_j) \cdot (\vee_k g_k)=_\eqref{inderSym} \\
&= \omega_{n-p} \sum_{i,j} \varsigmacheck_{p,q}(i,j) \widecheck{\cM}_{ij}(B) \nu(f_i, v_j) \cdot (\vee_k g_k)=_\eqref{eq.action.symmetric} \\
&= \omega_{n-p} \sum_{i,j,k} \varsigmacheck_{p,q}(i,j) \widecheck{\cM}_{ij}(B)
\Bigl( \prod_{t<k} \eta_{g_t, D(f_i,v_j)} \Bigr)
g_1 \vee\dots\vee \big(\nu(f_i, v_j) \cdot g_k\big) \vee\dots\vee g_n \\
&= \omega_{n-p} \sum_{i,j,k} \varsigmacheck_{p,q}(i,j) \widecheck{\cM}_{ij}(B)
\Bigl( \prod_{t<k} \eta_{g_t, D(f_i,v_j)} \Bigr)
g_1 \vee \dots \vee \{ f_i, v_j, g_k \} \vee \dots \vee g_n.
\end{align*}

\smallskip

4) Fix parity-ordered elements $\vee_i f_i \in\bigvee^{(p,n-p)}\cV^-_R$, $\vee_j v_j \in\bigvee^{(p,n-p)}\cV^+_R$, and $\varphi\in\Aut_R(\cV_R,\langle\cdot,\cdot\rangle)$. Then
\begin{equation*} \begin{split}
&\langle (\varphi^-)^{\vee n}(\vee_i f_i), (\varphi^+)^{\vee n}(\vee_j v_j) \rangle
= \langle \vee_i \varphi^-(f_i), \vee_j \varphi^+(v_j) \rangle \\
& \qquad = \omega_{n-p} \perdet_{p,n-p} \Big( \big(\langle \varphi^-(f_i), \varphi^+(v_j) \rangle\big)_{ij} \Big)
= \omega_{n-p} \perdet_{p,n-p} \Big( \big(\langle f_i, v_j \rangle\big)_{ij} \Big) \\
& \qquad = \langle \vee_i f_i, \vee_j v_j \rangle,
\end{split} \end{equation*}
thus $\langle\cdot,\cdot\rangle$ is $\bAut(\cV,\langle\cdot,\cdot\rangle)$-invariant, and consequently so it is the matrix $B$ (and its minors) associated to the elements $\vee_i f_i$ and $\vee_j v_j$.
Then, for $\vee_i f_i$, $\vee_j v_j$, $\vee_k g_k$ as above, we have
\begin{align*}
& (\varphi^-)^{\vee n}\big( \{ \vee_i f_i, \vee_j v_j, \vee_k g_k \} \big) =_\eqref{triple.product.symmetric} \\
&= (\varphi^-)^{\vee n}\Big(
\omega_{n-p} \sum_{i,j,k} \varsigmacheck_{p,q}(i,j) \widecheck{\cM}_{ij}(B)
\Bigl( \prod_{t<k} \eta_{g_t, D(f_i,v_j)} \Bigr)
g_1 \vee \dots \vee \{ f_i, v_j, g_k \} \vee \dots \vee g_n
\Big) \\
&= \omega_{n-p} \sum_{i,j,k} \varsigmacheck_{p,q}(i,j) \widecheck{\cM}_{ij}(B)
\Bigl( \prod_{t<k} \eta_{\varphi^-(g_t), D(\varphi^-(f_i),\varphi^+(v_j))} \Bigr) \cdot \\
& \qquad\qquad\qquad \cdot \varphi^-(g_1) \vee \dots \vee
\{ \varphi^-(f_i), \varphi^+(v_j), \varphi^-(g_k) \}
\vee \dots \vee \varphi^-(g_n) \\
&= \{ \vee_i \varphi^-(f_i), \vee_j \varphi^+(v_j), \vee_k \varphi^-(g_k) \} \\
&= \{ (\varphi^-)^{\vee n}(\vee_i f_i), (\varphi^+)^{\vee n}(\vee_j v_j),
(\varphi^-)^{\vee n}(\vee_k g_k) \},
\end{align*}
which also holds, analogously, for the other triple product. We have proven that
$\varphi^{\vee n} \in \bAut(\cW, \langle\cdot,\cdot\rangle)$.

$i)$ Since $\cW$ is $1$-dimensional, we have $\bAut(\cW,\langle\cdot,\cdot\rangle) \simeq \bG_m$.
By Proposition~\ref{propKernelSym}, it is clear that
$\bker \widecheck{\Omega}_n = \bSL_n \cap \bAut(\cV, \langle\cdot,\cdot\rangle)$.
It is obvious that
$\bAut(\cV,\langle\cdot,\cdot\rangle) / \bker \widecheck{\Omega}_n
\lesssim \bAut(\cW,\langle\cdot,\cdot\rangle)$.

$ii)$ By Proposition~\ref{propKernelSym} we get $\bker \widecheck{\Omega}_n = \bmu_n$, and the result follows.
\end{proof}

\begin{example}
Recall from \cite{L75} that the simple Jordan pairs of type III are the Jordan subpairs of $\cV^{\text{(I)}}_{n,n}$ given by $\cV^{\text{(III)}}_{n} := (H_{n}(\FF), H_{n}(\FF))$ (here $n \in \NN$ is arbitrary and $\chr\FF \neq 2$) where $H_n(\FF)$ is the vector space of $n \times n$ symmetric matrices, and their generic trace is given by
$$ t(x,y) = t^{\text{(III)}}(x,y) := \sum_i x_{ii}y_{ii} + 2 \sum_{i < j} x_{ij} y_{ij}. $$

As in Example~\ref{example.alt}, we will use the fact that $(\cV^{\text{(I)}}_{n,m}, t) \in \MGJP$.
Consider the basis $\{ \widecheck{E}_{ij} \med i \leq j \}$ of $H(\FF)$ where $\widecheck{E}_{ij} := E_{ij} + E_{ji}$ and note that  $\widecheck{E}_{ij} = \widecheck{E}_{ji}$. 
For $i\leq j$ and $k\leq l$, it is easy to see that
$$t^{\text{(III)}}(\widecheck{E}_{ij},\widecheck{E}_{kl}) = 2( \delta_{il}\delta_{jk} + \delta_{ik}\delta_{jl} )= t^{\text{(I)}}(\widecheck{E}_{ij},\widecheck{E}_{kl})$$
so that $t^{\text{(III)}} = t^{\text{(I)}}$ on $\cV^{\text{(III)}}_n$.
 Consequently, $t^{\text{(III)}}$ (which is nondegenerate) inherits the good properties from $t^{\text{(I)}}$, so that $(\cV^{\text{(III)}}_n, t) \in \MGJP$.
  We have that
\begin{align*}
(E_{i_1 i_2} & + E_{i_2 i_1}) (E_{j_1 j_2} + E_{j_2 j_1}) (E_{k_1 k_2} + E_{k_2 k_1}) = \\
&= ( \delta_{i_2 j_1} E_{i_1 j_2} + \delta_{i_1 j_2} E_{i_2 j_1}
+ \delta_{i_2 j_2} E_{i_1 j_1} + \delta_{i_1 j_1} E_{i_2 j_2} ) (E_{k_1 k_2} + E_{k_2 k_1}) \\
&= (\delta_{i_2 j_2} \delta_{j_1 k_2} + \delta_{i_2 j_1} \delta_{j_2 k_2}) E_{i_1 k_1}
+ (\delta_{i_2 j_1} \delta_{j_2 k_1} + \delta_{i_2 j_2} \delta_{j_1 k_1}) E_{i_1 k_2} \\
& \quad + (\delta_{i_1 j_1} \delta_{j_2 k_2} + \delta_{i_1 j_2} \delta_{j_1 k_2}) E_{i_2 k_1}
+ (\delta_{i_1 j_2} \delta_{j_1 k_1} + \delta_{i_1 j_1} \delta_{j_2 k_1}) E_{i_2 k_2},
\end{align*}
and swapping the labels $i \leftrightarrow k$ we get
\begin{align*}
(E_{k_1 k_2} & + E_{k_2 k_1}) (E_{j_1 j_2} + E_{j_2 j_1}) (E_{i_1 i_2} + E_{i_2 i_1}) = \\
&= (\delta_{k_2 j_2} \delta_{j_1 i_2} + \delta_{k_2 j_1} \delta_{j_2 i_2}) E_{k_1 i_1}
+ (\delta_{k_2 j_1} \delta_{j_2 i_1} + \delta_{k_2 j_2} \delta_{j_1 i_1}) E_{k_1 i_2} \\
& \quad + (\delta_{k_1 j_1} \delta_{j_2 i_2} + \delta_{k_1 j_2} \delta_{j_1 i_2}) E_{k_2 i_1}
+ (\delta_{k_1 j_2} \delta_{j_1 i_1} + \delta_{k_1 j_1} \delta_{j_2 i_1}) E_{k_2 i_2}.
\end{align*}
Therefore, the triple products of $\cV^{\text{(III)}}_n$ are given by
\begin{align*}
\{ \widecheck{E}_{i_1 i_2}, & \widecheck{E}_{j_1 j_2}, \widecheck{E}_{k_1 k_2} \} = \\
&= (E_{i_1 i_2} + E_{i_2 i_1}) (E_{j_1 j_2} + E_{j_2 j_1}) (E_{k_1 k_2} + E_{k_2 k_1}) \\
& \quad + (E_{k_1 k_2} + E_{k_2 k_1}) (E_{j_1 j_2} + E_{j_2 j_1}) (E_{i_1 i_2} + E_{i_2 i_1}) \\
&= (\delta_{i_2 j_2} \delta_{j_1 k_2} + \delta_{i_2 j_1} \delta_{j_2 k_2}) \widecheck{E}_{i_1 k_1}
+ (\delta_{i_2 j_1} \delta_{j_2 k_1} + \delta_{i_2 j_2} \delta_{j_1 k_1}) \widecheck{E}_{i_1 k_2} \\
& \quad + (\delta_{i_1 j_1} \delta_{j_2 k_2} + \delta_{i_1 j_2} \delta_{j_1 k_2}) \widecheck{E}_{i_2 k_1}
+ (\delta_{i_1 j_2} \delta_{j_1 k_1} + \delta_{i_1 j_1} \delta_{j_2 k_1}) \widecheck{E}_{i_2 k_2},
\end{align*}
and the generic trace by
$$ t(\widecheck{E}_{i_1 i_2}, \widecheck{E}_{j_1 j_2}) = 2( \delta_{i_1 j_2}\delta_{i_2 j_1} + \delta_{i_1 j_1}\delta_{i_2 j_2}). $$
 
Now, like in Example \ref{example.alt}, consider two copies of the canonical basis $\{ e_i \}_{i=1}^n$ of $\cM_{1,n}(\FF)$, regarded as bases of the subspaces of $\cV^{\text{(I)}}_{1,n}$, and recall that
$$ t(e_i, e_j) = \delta_{ij}, $$
and
$$ \{ e_i, e_j, e_k \} = \delta_{ij} e_k + \delta_{kj} e_i. $$
 
Then $\{ e_i \vee e_j \med 1 \leq i \leq j \leq n \}$ is a basis for both vector spaces of the pair $\cV = \bigvee^2 \cV^{\text{(I)}}_{1,n}$.
Assuming $1 \leq i_1 \leq i_2 \leq n$ and $1 \leq j_1 \leq j_2 \leq n$, the bilinear form of $\cV$ is given by
\begin{align*}
\langle e_{i_1} \vee e_{i_2}, e_{j_1} \vee e_{j_2} \rangle
&= \per\big( (t(e_{i_k}, e_{j_l}))_{kl} \big) = \per \big( (\delta_{i_k j_l})_{kl} \big)= \delta_{i_1 j_1} \delta_{i_2 j_2} + \delta_{i_2 j_1} \delta_{i_1 j_2}.
\end{align*}
Let $M_{ij}$ denote the permanent $(i,j)$-minor of $B = \big( t(e_{i_k}, e_{j_l}) \big)_{kl} = \big( \delta_{i_k j_l} \big)_{kl}$.
Then $M_{11} = \delta_{i_2 j_2}$, $M_{12} = \delta_{i_2 j_1}$, $M_{21} = \delta_{i_1 j_2}$, $M_{22} = \delta_{i_1 j_1}$, and the triple products of $\cV$ are given by
\begin{align*}
\{ e_{i_1} \vee & e_{i_2}, e_{j_1} \vee e_{j_2}, e_{k_1} \vee e_{k_2} \} = \\
&=  M_{11} ( \{ e_{i_1}, e_{j_1}, e_{k_1} \} \vee e_{k_2} + e_{k_1} \vee \{ e_{i_1}, e_{j_1}, e_{k_2} \}) \\
& \quad +M_{12} ( \{ e_{i_1}, e_{j_2}, e_{k_1} \} \vee e_{k_2} + e_{k_1} \vee \{ e_{i_1}, e_{j_2}, e_{k_2} \}) \\
& \quad +M_{21} ( \{ e_{i_2}, e_{j_1}, e_{k_1} \} \vee e_{k_2} + e_{k_1} \vee \{ e_{i_2}, e_{j_1}, e_{k_2} \}) \\
& \quad +M_{22} ( \{ e_{i_2}, e_{j_2}, e_{k_1} \} \vee e_{k_2} + e_{k_1} \vee \{ e_{i_2}, e_{j_2}, e_{k_2} \}) \\
&=  \delta_{i_2 j_2} \Big( \delta_{i_1 j_1} e_{k_1} \vee e_{k_2} + \delta_{k_1 j_1} e_{i_1} \vee e_{k_2}
                        + \delta_{i_1 j_1} e_{k_1} \vee e_{k_2} + \delta_{k_2 j_1} e_{k_1} \vee e_{i_1} \Big) \\
&\quad + 
   \delta_{i_2 j_1} \Big( \delta_{i_1 j_2} e_{k_1} \vee e_{k_2} + \delta_{k_1 j_2} e_{i_1} \vee e_{k_2}
                        + \delta_{i_1 j_2} e_{k_1} \vee e_{k_2} + \delta_{k_2 j_2} e_{k_1} \vee e_{i_1} \Big) \\
&\quad + 
   \delta_{i_1 j_2} \Big( \delta_{i_2 j_1} e_{k_1} \vee e_{k_2} + \delta_{k_1 j_1} e_{i_2} \vee e_{k_2}
                        + \delta_{i_2 j_1} e_{k_1} \vee e_{k_2} + \delta_{k_2 j_1} e_{k_1} \vee e_{i_2} \Big) \\
&\quad + 
   \delta_{i_1 j_1} \Big( \delta_{i_2 j_2} e_{k_1} \vee e_{k_2} + \delta_{k_1 j_2} e_{i_2} \vee e_{k_2}
                        + \delta_{i_2 j_2} e_{k_1} \vee e_{k_2} + \delta_{k_2 j_2} e_{k_1} \vee e_{i_2} \Big) \\
&=        \Big( \delta_{i_2 j_1} \delta_{k_2 j_2} + \delta_{i_2 j_2} \delta_{k_2 j_1} \Big) e_{i_1} \vee e_{k_1}
        + \Big( \delta_{i_2 j_2} \delta_{k_1 j_1} + \delta_{i_2 j_1} \delta_{k_1 j_2} \Big) e_{i_1} \vee e_{k_2} \\
& \quad + \Big( \delta_{i_1 j_2} \delta_{k_2 j_1} + \delta_{i_1 j_1} \delta_{k_2 j_2} \Big) e_{i_2} \vee e_{k_1}
        + \Big( \delta_{i_1 j_1} \delta_{k_1 j_2} + \delta_{i_1 j_2} \delta_{k_1 j_1} \Big) e_{i_2} \vee e_{k_2} \\
& \quad	+4\Big( \delta_{i_1 j_1} \delta_{i_2 j_2} + \delta_{i_2 j_1} \delta_{i_1 j_2} \Big) e_{k_1} \vee e_{k_2}
\end{align*}
Finally, consider the tensor-shift $\cV^{[-4]}$, which has the same bilinear form as $\cV$, and triple products
\begin{align*}
\{ e_{i_1} \vee & e_{i_2}, e_{j_1} \vee e_{j_2}, e_{k_1} \vee e_{k_2} \} = \\
&=        \Big( \delta_{i_2 j_1} \delta_{k_2 j_2} + \delta_{i_2 j_2} \delta_{k_2 j_1} \Big) e_{i_1} \vee e_{k_1}
        + \Big( \delta_{i_2 j_2} \delta_{k_1 j_1} + \delta_{i_2 j_1} \delta_{k_1 j_2} \Big) e_{i_1} \vee e_{k_2} \\
& \quad + \Big( \delta_{i_1 j_2} \delta_{k_2 j_1} + \delta_{i_1 j_1} \delta_{k_2 j_2} \Big) e_{i_2} \vee e_{k_1}
        + \Big( \delta_{i_1 j_1} \delta_{k_1 j_2} + \delta_{i_1 j_2} \delta_{k_1 j_1} \Big) e_{i_2} \vee e_{k_2}.
\end{align*}
By comparison of the triple products, it follows that the pair of maps $f = (f^-, f^+)$ defined by
\begin{equation*}
f^\sigma \colon H_n(\FF) \longrightarrow \bigvee^2 \cM_{1,n}(\FF),
\quad \widecheck{E}_{ij} \longmapsto  e_i \vee e_j,
\end{equation*}
gives the following isomorphism of (generalized) Jordan pairs:
\begin{equation}
\cV^{\text{(III)}}_n \cong \Big( \bigvee^2 \cV^{\text{(I)}}_{1,n} \Big)^{[-4]}
= \Big( \bigvee^2 \cV^{\text{(I)}}_{1,n} \Big) \otimes \cV_{-4}.
\end{equation}
Unfortunately, $f$ is not an isometry of the bilinear forms. However, $f$ is a similarity with multiplier $\frac{1}{2}$, that is, $\langle f(x), f(y) \rangle = \frac{1}{2} t( x, y)$. In other words, $\cV^{\text{(III)}}_n$ and $\bigvee^2 \cV^{\text{(I)}}_{1,n}$ are isomorphic up to a tensor-shift and a similarity (simultaneously), namely $(\cV^{\text{(III)}}_n, \frac{1}{2} t) \cong ( \bigvee^2 \cV^{\text{(I)}}_{1,n}, \langle\cdot,\cdot\rangle)^{[-4]}$.
\end{example}

\textbf{Acknowledgements}
The authors are extremely thankful to the anonymous referee for finding important errors, and giving very useful comments and suggestions that have considerably improved this manuscript. We are also grateful to Ian M. Musson, for providing useful references.


\end{document}